\documentclass[12pt]{article}
\usepackage{amsmath,amssymb,fullpage,amsthm,bbm,booktabs,graphicx,subcaption,comment}
\usepackage[numbers]{natbib}
\usepackage[]{algorithm2e}

\newtheorem{theorem}{Theorem}[section]
\newtheorem{lemma}[theorem]{Lemma}
\newtheorem{remark}[theorem]{Remark}
\newtheorem{corollary}[theorem]{Corollary}
\newtheorem{proposition}[theorem]{Proposition}
\theoremstyle{definition}

\numberwithin{equation}{section}

\begin{document}
\begin{center}
\textbf{\large{The Bayes Lepski's Method and Credible Bands through Volume of Tubular Neighborhoods}}
\end{center}
\begin{center}
\textbf{William Weimin Yoo and Aad W. van der Vaart}\\
\textit{Leiden University}
\end{center}
\begin{abstract}
For a general class of priors based on random series basis expansion, we develop the Bayes Lepski's method to estimate unknown regression function. In this approach, the series truncation point is determined based on a stopping rule that balances the posterior mean bias and the posterior standard deviation. Equipped with this mechanism, we present a method to construct adaptive Bayesian credible bands, where this statistical task is reformulated into a problem in geometry, and the band's radius is computed based on finding the volume of certain tubular neighborhood embedded on a unit sphere. We consider two special cases involving B-splines and wavelets, and discuss some interesting consequences such as the uncertainty principle and self-similarity. Lastly, we show how to program the Bayes Lepski stopping rule on a computer, and numerical simulations in conjunction with our theoretical investigations concur that this is a promising Bayesian uncertainty quantification procedure.
\end{abstract}

\textbf{Keywords:} Bayes Lepski, volume of tube, adaptive credible bands, supremum norm posterior contraction, B-splines, CDV wavelets, uncertainty principle, self-similar.\vspace{5pt}

\textbf{MSC2010 classifications:} Primary 62G15 62G05; secondary 62G08 62C10

\section{Introduction}
Uncertainty quantification is now an important research direction in Bayesian nonparametrics, and substantial efforts have been made to answer the question whether Bayesian credible sets are indeed frequentist confidence regions, especially in infinite-dimensional settings. As Bayesian uncertainty quantification can be conducted automatically once the posterior distribution is derived from our modeling assumptions, a positive answer to this question will give statisticians a powerful alternative to other commonly used procedures such as the bootstrap.

In this paper, we will study the issue of constructing adaptive Bayesian simultaneous credible bands that will have both high levels of credibility and also coverage probability in the frequentist sense. To begin our investigation, let us consider the nonparametric regression model
\begin{align}\label{eq:model}
Y_i=f(X_i)+\varepsilon_i,\qquad i=1,\dotsc,n,
\end{align}
where $Y_i$ is the response variable, $X_i$ is the covariate, and the errors $\varepsilon_1,\dotsc,\varepsilon_n$ are independent and identically distributed (i.i.d.)~as $\mathrm{N}(0,\sigma^2)$ with unknown $0<\sigma<\infty$. Each $X_i$ takes values in some rectangular region in $\mathbb{R}$, which is assumed to be $[0,1]$ without loss of generality.

Credible bands are constructed based on the supremum norm distance. They are easily visualized and they allow one to immediately see the extent of uncertainty as delimited by the bands. It is for this reason that credible bands are more statistically meaningful than for example $L_2$-type credible sets. Several methods have been proposed in the Bayesian literature. We have the fixed radius $L_\infty$-ball around the posterior mean through B-splines random series prior (\citep{yoo2016}) and scaled Brownian motion prior (Chapter 3 of \citep{suzanne}), the technique of ``gluing" pointwise intervals for linear functionals to produce bands (\citep{castilloBVM2013,botondband}) and the intersection of two balls through the multiscale approach (\citep{kolyan,castilloBVM2014}).

Estimation and uncertainty quantification for nonparametric models, such as the regression problem considered in this paper, depend crucially on the choice of the tuning parameter in any proposed statistical procedure. Let us work with random series priors by projecting $f$ onto a general class of $J$-dimensional basis functions and eliciting normal priors on the coefficients. We know from minimax theory that the optimal $J$ when loss is measured using the $L_\infty$-distance is different from the case of $L_2$-loss. Specifically, if the true regression function $f_0$ is $\alpha$-H\"{o}lder smooth, then $J$ must be of the order $(n/\log{n})^{1/(2\alpha+1)}$ for $L_\infty$-norm and $n^{1/(2\alpha+1)}$ for $L_2$-norm. If we assign $J$ a prior or estimate it using empirical Bayes, the resulting posterior will concentrate around the $L_2$-optimal $J$, due to $L_2$-norm being the intrinsic metric for the likelihood function. Therefore for inference tasks based on the $L_\infty$-distance such as constructing Bayesian credible bands, the usual Bayesian procedures will not ``choose" the appropriate $L_\infty$-norm $J$ and this discrepancy causes the credible bands to have inadequate coverage probabilities.

To circumvent this problem, we propose a procedure to choose the correct $L_\infty$-norm $J$ by using only summaries of the posterior distribution such as its mean and variance. The dynamics of bias-variance tradeoff dictates that the tuning parameter $J$ is chosen to balance the supremum norm bias of the posterior mean and the posterior standard deviation. Based on this observation, we design a Lepski's type Bayesian stopping rule by successively reducing the value of $J$ until the point when the posterior mean bias starts to dominate the posterior standard deviation. If we adopt the framework that all Bayesian estimation and inference must be done through the posterior distribution, then this procedure is admissible within this framework since we are using posterior quantities to decide when to stop. Despite the intuitive nature of Lepski's type procedure, it has since its inception (\citep{lepski,lepski1997}) acquired the reputation of being impractical to use in actual practice, and consequently has only remained as a theoretical tool for statisticians to construct adaptive procedures.

In this paper, we hope to shatter this perception by showing that the aforementioned Bayes Lepski's method can be programmed on a computer to do actual tuning parameter selection. As the optimal $J$ is increasing in $n$, we tend to stop sooner when sample size is larger, and this is in stark contrast to established methods such as leave-one-out cross validation. It will be seen that pairwise differences in posterior means are used as a proxy for posterior mean bias, and due to their nature of being shrinkage estimators results in better performance (in terms of coverage probabilities) when compared to their frequentist counterpart. More importantly as a Bayesian procedure, the Bayes Lepski's method alleviates the need to design complicated reversible jump Markov Chain Monte Carlo algorithms if $J$ was assigned a prior, and it moreover ensures that the correct $J$ is chosen based on the loss/distance used in the problem.

Equipped with this mechanism, we present a method to construct Bayesian credible bands with guaranteed coverage probabilities, where this statistical inference task is reformulated into a problem in geometry, and the band's radius is computed based on finding the volume of certain tubular neighborhoods embedded on a unit sphere. In many cases, this volume can be explicitly calculated and it yields closed-form formulas for radius computation. Thus, the volume of tube approach bypasses the need to sample repeatedly from the posterior to estimate the radius empirically, and coupled with the Bayes Lepski's method introduced previously, results in an efficient procedure for adaptive Bayesian credible band construction with coverage guarantees. The idea of using volume of geometrical objects to compute statistical quantities goes back to \citep{hotelling} in the context of hypothesis testing. Subsequent authors used and generalized this idea to other applications in frequentist statistics, for example volume tests of significance, simultaneous inference, projection pursuit regression and others (see \citep{naiman1986,knowles,johnstone}). The Bayes Lepski volume of tube method presented above was partly inspired by these ideas and we address some important issues that were not considered by these authors, in particular tuning parameter selection (through Bayes Lepski) and self-similarity.

We apply our general result to B-splines and wavelets, and this inadvertently led us to some counterintuitive notions in uncertainty quantification. It should be intuitively clear that the Bayes Lepski's procedure cannot choose $J$ that is too small as to make the model bias unmanageable. If we were using B-splines of order $q$ as our prior, this event is prevented by requiring that the true function be a certain distance away from the space of $q$-order polynomial splines. However if the truth turns out to be a polynomial splines of the same order, then it is at distance zero from this space and we are not able to quantify uncertainty honestly for this truth, despite the fact that we have correctly modelled it using B-splines. This seems to give rise to the phenomenon that the more accurately we are able to model and estimate the truth, the less we are able to access its quality and statistical uncertainty and vice versa, where such statement is reminiscence of the uncertainty principle in quantum mechanics.

The paper is structured as follows. In the next section we state some notations used in this paper. In Section \ref{sec:prior}, we introduce the Bayes Lepski's method and discuss prior and model assumptions. This is then followed by a general result on $L_\infty$-posterior contraction. In Section \ref{sec:credible}, we first give an exposition on the volume of tube method to bound tail probabilities in the spirit of \citep{johnstone}, we then introduce our Bayesian credible band construction and present the main results on its coverage probability and size. Consequently, we specialize our results to B-splines (Section \ref{sec:bspline}) and wavelets (Section \ref{sec:wavelet}). In Section \ref{sec:sim}, we discuss practical implementation and present simulation results. Proofs are in Section \ref{sec:proof} while the Appendix (Section \ref{sec:appendix}) contains some useful auxiliary results that are of independent interests.

\section{Notations and Definitions}
For two sequences $a_n,b_n$ and when $n\rightarrow\infty$, we say that $a_n=O(b_n)$ or $a_n\lesssim b_n$ if $a_n\leq Cb_n$ for some constant $C>0$ not depending on $n$; and we say $a_n=o(b_n)$ if $a_n/b_n\rightarrow0$. The real line is $\mathbb{R}$ while positive integers are $\mathbb{N}$. Let us denote the $L_p(D)$-norms of $f$ as $\|f\|_p=(\int_D|f(x)|^pdx)^{1/p}$, with $\|f\|_\infty=\sup_{x\in D}|f(x)|$ for $p=\infty$. For the vector/series version $\ell_p$-norms, we have $\|\boldsymbol{x}\|_p=(\sum_i|x_i|^p)^{1/p}$ and $\|\boldsymbol{x}\|_\infty=\max_i|x_i|$. As in standard practice, we write $\|\boldsymbol{x}\|_2$ simply as $\|\boldsymbol{x}\|$. Let $\lambda_{\mathrm{max}}(\boldsymbol{A})$ and $\lambda_{\mathrm{min}}(\boldsymbol{A})$ be the maximum and minimum eigenvalues of a matrix $\boldsymbol{A}$. The induced matrix norm is defined as $\|\boldsymbol{A}\|_{(p,q)}=\sup_{\boldsymbol{x}\neq0}\|\boldsymbol{Ax}\|_q/\|\boldsymbol{x}\|_p$. Then it is known that $\|\boldsymbol{A}\|_{(2,2)}=\lambda_{\mathrm{max}}^{1/2}(\boldsymbol{A}^T\boldsymbol{A})$ is the spectral norm, $\|\boldsymbol{A}\|_{(\infty,\infty)}=\max_{1\leq i\leq m}\sum_{j=1}^n|a_{ij}|$ is the max of absolute value of row sums, and $\|\boldsymbol{A}\|_{(1,1)}=\max_{1\leq j\leq n}\sum_{i=1}^m|a_{ij}|$ is the max of absolute value of column sums. By H\"{o}lder's inequality, we have the relation $\|\boldsymbol{A}\|_{(2,2)}\leq\sqrt{\|\boldsymbol{A}\|_{(1,1)}\|\boldsymbol{A}\|_{(\infty,\infty)}}$. If $\boldsymbol{A}$ is symmetric, then $\|\boldsymbol{A}\|_{(2,2)}=|\lambda_{\mathrm{max}}(\boldsymbol{A})|$ and $\|\boldsymbol{A}\|_{(1,1)}=\|\boldsymbol{A}\|_{(\infty,\infty)}$, which in view of the relation above gives $\|\boldsymbol{A}\|_{(2,2)}\leq\|\boldsymbol{A}\|_{(\infty,\infty)}$.

\section{The Bayes Lepski's method}\label{sec:prior}
For $f\in L_2([0,1])$, let us project $f$ onto a lower $J$-dimensional linear subspace spanned by the basis functions $\boldsymbol{a}_J(x)=(a_1(x),\dotsc,a_J(x))^T$. These bases are compactly supported on some interval in $\mathbb{R}$ and they are linearly independent. Note that we do not require them to be orthonormal. We can write this projection as $f(x)=\boldsymbol{a}_J(x)^T\boldsymbol{\theta}$ and the associated linear operator arising from such projection as $K_J$. We then say that the bases are $\upsilon$-regular if $K_J(p)=p$ when $p$ is a polynomial of degree $\upsilon$ or less, i.e., $K_J$ reproduces polynomials of degree $\upsilon$ or less.

Let us choose end points $1\leq j_{\mathrm{min}}\leq j_{\mathrm{max}}\leq n$ such that $j_{\mathrm{min}}=(n/\log{n})^{1/(2\upsilon+1)}$ and $j_{\mathrm{max}}=o(n/\log{n})$. Write the candidate set as $\mathcal{J}:=[j_{\mathrm{min}},j_{\mathrm{max}}]\cap\mathbb{N}$. In what follows, we will only consider $J$ belonging to $\mathcal{J}$, as the latter serves as a reservoir of admissible $J$ for our model.

Next, we endow a Gaussian process prior on $f$ by eliciting a normal prior on the basis coefficients $\boldsymbol{\theta}|\sigma\sim\mathrm{N}(\boldsymbol{\eta},\sigma^2\boldsymbol{\Omega})$. We take $\|\boldsymbol{\eta}\|_\infty<\infty$ and select the prior covariance matrix such that for some constants $0<c_1\leq c_2<\infty$,
\begin{align}\label{eq:prior}
c_1\leq\lambda_{\mathrm{min}}(\boldsymbol{\Omega})\leq\lambda_{\mathrm{max}}(\boldsymbol{\Omega})\leq c_2.
\end{align}
Let $\boldsymbol{A}=(\boldsymbol{a}_J(X_1)^T,\dotsc,\boldsymbol{a}_J(X_n)^T)^T$ be the basis matrix constructed by evaluating the basis vector $\boldsymbol{a}_J(x)$ at the covariates $(X_1,\dotsc,X_n)^T$. Since the bases are linearly independent, the Gram matrix $\boldsymbol{A}^T\boldsymbol{A}$ is nonsingular and we choose the covariates such that for some constants $0<C_1\leq C_2<\infty$,
\begin{align}\label{eq:ceb}
C_1\lambda_{\boldsymbol{A},J}\leq\lambda_{\mathrm{min}}(\boldsymbol{A}^T\boldsymbol{A})\leq\lambda_{\mathrm{max}}(\boldsymbol{A}^T\boldsymbol{A})\leq C_2\lambda_{\boldsymbol{A},J}
\end{align}
as $n\rightarrow\infty$, where the common eigenvalue bound $\lambda_{\boldsymbol{A},J}\rightarrow\infty$ as $n\rightarrow\infty$ for any $J\in\mathcal{J}$, and we assume it is decreasing in $J$ for any fixed $n$. We further assume that $\boldsymbol{A}^T\boldsymbol{A}$ is $m$-banded in the sense that $(\boldsymbol{A}^T\boldsymbol{A})_{ij}=0$ if $|i-j|>m$ for some fixed $m$ not depending on $n$. This is a reasonable assumption in view of the fact that most bases used in applications such as B-splines and Daubechies wavelets produce banded Gram matrices. In fact, bandedness arises naturally for compactly supported basis systems that are constructed by shifting and scaling a base/template function. It will be seen that the shrinkage factor $(\boldsymbol{A}^T\boldsymbol{A}+\boldsymbol{\Omega}^{-1})^{-1}$ will be present in both the posterior mean and variance, and to exploit the bandedness of $\boldsymbol{A}^T\boldsymbol{A}$ in our proofs and also during actual computations, we let the prior precision matrix $\boldsymbol{\Omega}^{-1}$ be $r$-banded with $r$ possibly differing from $m$.

Let us denote the uniform $\ell_1$-and $\ell_2$-norms of the basis vector $\boldsymbol{a}_J$ as
\begin{align}
l_{1,J}:=\sup_{x\in[0,1]}\|\boldsymbol{a}_J(x)\|_1\quad\text{and}\quad l_{2,J}:=\sup_{x\in[0,1]}\|\boldsymbol{a}_J(x)\|,
\end{align}
and we know that $l_{1,J}\geq l_{2,J}$. For any $J\in\mathcal{J}$, we require the common eigenvalue bound be
\begin{align}\label{eq:Al1l2}
\lambda_{\boldsymbol{A},J}\geq l_{1,J}^2.
\end{align}
To deal with unknown $\sigma$, we will use empirical Bayes by maximizing the marginal likelihood $\boldsymbol{Y}|\sigma\sim\mathrm{N}[\boldsymbol{A\eta},\sigma^2(\boldsymbol{A\Omega A}^T+\boldsymbol{I})]$ to get
\begin{align}\label{eq:sigma}
\widehat{\sigma}_J^2=n^{-1}(\boldsymbol{Y}-\boldsymbol{A\eta})^T(\boldsymbol{A\Omega A}^T+\boldsymbol{I})^{-1}(\boldsymbol{Y}-\boldsymbol{A\eta}).
\end{align}
For some large enough constant $\tau>0$, we estimate the optimal number of basis $J$ in $\mathcal{J}$ by
\begin{align}\label{eq:optimalj}
\widehat{j}_n=\min\left\{j\in\mathcal{J}:\|\mathrm{E}_j(f|\boldsymbol{Y})-\mathrm{E}_i(f|\boldsymbol{Y})\|_\infty\leq\tau\widehat{\sigma}_il_{1,i}\sqrt{\frac{\log{i}}{\lambda_{\boldsymbol{A},i}}},\forall i>j,i\in\mathcal{J}\right\}.
\end{align}
Here $\mathrm{E}_J[f(x)|\boldsymbol{Y}]$ is the posterior mean by making its dependence on number of basis $J$ explicit, which by conjugacy is
\begin{align}\label{eq:pmean}
\mathrm{E}_J[f(x)|\boldsymbol{Y}]=\boldsymbol{a}_J(x)^T(\boldsymbol{A}^T\boldsymbol{A}+\boldsymbol{\Omega}^{-1})^{-1}(\boldsymbol{A}^T\boldsymbol{Y}+\boldsymbol{\Omega}^{-1}\boldsymbol{\eta}).
\end{align}
In words, our stopping rule is when the less-than inequality of \eqref{eq:optimalj} is flipped as we take successively smaller $j\in\mathcal{J}$, and set $\widehat{j}_n$ as the previous element when this occurs. Note that since the optimal $J$ is increasing in $n$, we tend to stop sooner when sample size is larger. Compare this to established methods such as leave-one-out cross validation, where we have to iterate across all $J\in\mathcal{J}$ and execution time for each iteration increases exponentially with $n$. Although we begin with a large $J$, computation is still manageable due to the bandedness of $\boldsymbol{A}^T\boldsymbol{A}$ and compact support of the basis function, which has only a fixed number of nonzero values at any given $x$.

If the minimum is over an empty set, i.e., the sup-norm difference in posterior means is always larger than the right hand side on $\mathcal{J}$, we can either let $\widehat{j}_n=j_{\mathrm{max}}$ by default or increase $j_{\mathrm{max}}$ until an appropriate $\widehat{j}_n$ is chosen. On the other hand, if this difference is always less than the right hand side over all $\mathcal{J}$, we conclude that $j_{\mathrm{min}}$ is too large and we enlarge $\mathcal{J}$ by increasing the regularity $\upsilon$ of the bases. We then substitute $J$ for $\widehat{j}_n$ in the conditional posterior distribution $\Pi_J(f|\boldsymbol{Y},\sigma=\sigma_J)$, and we write this Lepski's (Gaussian) posterior simply as $\Pi(f|\boldsymbol{Y})$.

\begin{remark}
The form given in \eqref{eq:optimalj} (especially the right hand side) is asymptotic in nature, and it is used to derive large sample results for credible band coverage. In finite samples, some extra considerations have to be taken into account, for example inducing undersmoothing by encouraging earlier stopping. We do not do sample splitting as we feel that it is more natural to achieve the same aim by finding a good finite sample proxy for the right hand side. More implementation issues will be discussed in Section \ref{sec:sim} on simulations.
\end{remark}

To study Bayesian procedures from a frequentist viewpoint, let us assume the existence of a true regression function $f_0:[0,1]\rightarrow\mathbb{R}$ such that the corresponding true model has the following assumption:\vspace{5pt}

\noindent\textbf{Global assumption on true model}:\\
Under the true distribution $P_0$, $Y_i=f_0(X_i)+\varepsilon_i, i=1,\dotsc,n$, where $\varepsilon_i$ are i.i.d.~Gaussian with mean $0$ and finite variance $\sigma_0^2>0$ for $i=1,\dotsc,n$.

\subsection{Preliminary: adaptive contraction rates in $L_\infty$}
Let $\mathcal{G}$ be a Banach ball defined through the norm $\|\cdot\|_{\mathcal{G}}$, such that for any $g\in\mathcal{G}$, $\|g\|_{\mathcal{G}}\leq R$ for some $R>0$, and there exists a $\boldsymbol{\theta}_0\in\mathbb{R}^J$ where for some constant $C_0>0$ depending on the basis used, we have
\begin{align}\label{eq:approxG}
\|\boldsymbol{a}_J(\cdot)^T\boldsymbol{\theta}_0-g\|_\infty\leq C_0\|g\|_{\mathcal{G}}h(J),
\end{align}
with $h$ being a nonincreasing function such that $h(j_{\mathrm{min}})\rightarrow0$ as $n\rightarrow\infty$, and $\|\boldsymbol{\theta}_0\|_\infty<\infty$. Moreover for $f_0\in\mathcal{G}$, let us define
\begin{align}\label{eq:j0}
j_n^{*}:=\min\left\{j\in\mathcal{J}:\|f_0\|_{\mathcal{G}}h(j)\leq l_{1,j}\sqrt{\frac{\log{j}}{\lambda_{\boldsymbol{A},j}}}\right\}.
\end{align}
Here if $j\geq j_n^{*}$, then the posterior mean bias $\|f_0\|_{\mathcal{G}}h(j)\leq l_{1,j}\sqrt{\log{(j)}/\lambda_{\boldsymbol{A},j}}$; while this inequality will reverse when $j<j_n^{*}$. Therefore we deduce that $j_n^{*}$ solves $\|f_0\|_{\mathcal{G}}h(j)=l_{1,j}\sqrt{\log{(j)}/\lambda_{\boldsymbol{A},j}}$ up to some universal constants not depending on $n$, i.e., $j_n^{*}$ balances the order of the posterior mean bias and the posterior standard deviation as represented by the expression on the right hand side. In this context, the Bayes Lepski's stopping rule in \eqref{eq:optimalj} can be seen as a data driven procedure to estimate this true $j_n^{*}$.

For all results in this section, we formulate our rates in the bias-variance-tradeoff fashion without explicit reference to the smoothness scale that $f_0$ might belong to. A key requirement before establishing contraction rate is consistency of the empirical Bayes estimate $\widehat{\sigma}_J,J\in\mathcal{J}$:
\begin{proposition}[Consistency]\label{prop:sigma}
For any $J\in\mathcal{J}$, there are constants $\xi,C,Q>0$ such that as $n\rightarrow\infty$,
\begin{align*}
\inf_{f_0\in\mathcal{G}}P_0\left(\sigma_0-\xi\delta_{n,J}\leq\widehat{\sigma}_J\leq\sigma_0+\xi\delta_{n,J}\right)\geq1-Ce^{-QJ},
\end{align*}
where $\delta_{n,J}=\sqrt{J/n}+\|f_0\|_{\mathcal{G}}h(J)\rightarrow0$.
\end{proposition}
Note that this is an exponential type inequality, as the usual route through mean square error and Markov's inequality prove to be inadequate for $L_\infty$-contraction, this is due to the need of controlling simultaneously $\widehat{\sigma}_i$ for all $i>j$ in the stopping rule of \eqref{eq:optimalj}. Another key ingredient is to ensure that the Bayes Lepski procedure will not stop sooner before reaching the true $j_n^{*}$ with probability tending to $1$.

\begin{proposition}\label{lem:upper}
For large enough $\tau$, there exists a constant $\mu>2$ such that for $n\rightarrow\infty$,
\begin{align*}
\sup_{f_0\in\mathcal{G}}P_0(\widehat{j}_n>j_n^{*})\lesssim\frac{1}{(j_n^{*})^{\mu-2}}.
\end{align*}
\end{proposition}
These two key perliminaries will then enable us to establish the following result on adaptive sup-norm posterior contraction.

\begin{theorem}\label{th:rate}
As $n\rightarrow\infty$, there are constants $\xi,M>0$ such that
\begin{align}
\sup_{f_0\in\mathcal{G}}\mathrm{E}_0\Pi\left[\|f-f_0\|_\infty>\xi l_{1,j_n^{*}}\sqrt{\frac{\log{j_n^{*}}}{\lambda_{\boldsymbol{A},j_n^{*}}}}\middle|\boldsymbol{Y}\right]\lesssim\left(j_n^{*}\right)^{-M}.
\end{align}
\end{theorem}
The rate above is in abstract form that is applicable to general basis systems that fulfil our modeling assumptions. Concrete rates will be given in Sections \ref{sec:bspline} and \ref{sec:wavelet} to B-splines and wavelets respectively. Using the same techniques utilized in Corollary 4.5 of \citep{yoo2017}, we further deduce that the posterior mean as a point estimator converges to $f_0$ at the rate indicated above uniformly over $[0,1]$ in $P_0$-probability.

\section{Adaptive credible bands through geometry}\label{sec:credible}
In Bayesian nonparametrics, the standard construction of credible bands is a $L_\infty$-ball of fixed radius around the posterior mean. However in many problems, it is more natural to allow the width to depend on the regression function's domain such that its radius is a function of $x\in[0,1]$. This enables the bands to adapt to local characteristics of the curve to be estimated and hence is expected to give better uncertainty quantification. As an example, for some fixed $J\in\mathcal{J}$ and a given credibility level $1-\gamma$, we can construct bands by using the posterior mean as center and some multiple of the posterior standard deviation as radius as follows:
\begin{align*}
\mathcal{C}_J:=\left\{f:|f(x)-\mathrm{E}_J[f(x)|\boldsymbol{Y}]|\leq w_{\gamma,J}\sqrt{\mathrm{Var}_J[f(x)|\boldsymbol{Y}]},\forall x\in[0,1]\right\},
\end{align*}
and we choose the quantile $w_{\gamma,J}>0$ such that the posterior probability of this set is at least $1-\gamma$. For our prior formulation $f(x)=\boldsymbol{a}_J(x)^T\boldsymbol{\theta}$, the computation of $w_{\gamma,J}$ involves estimating
\begin{align}\label{eq:credc}
\Pi\left[\sup_{x\in[0,1]}\left|\frac{\boldsymbol{a}_J(x)^T[\boldsymbol{\theta}-\mathrm{E}_J(\boldsymbol{\theta}|\boldsymbol{Y})]}{\sqrt{\boldsymbol{a}_J(x)^T\mathrm{Var}_J(\boldsymbol{\theta}|\boldsymbol{Y})\boldsymbol{a}_J(x)}}\right|>w_{\gamma,J}\middle|\boldsymbol{Y}\right].
\end{align}
Now if we have a sharp upper bound $U(w_{\gamma,J})$ to this tail probability, then we find $w_{\gamma,J}$ such that it solves $U(w_{\gamma,J})=\gamma$.

With this choice, let us now access the coverage (in the frequentist sense) of the Bayesian credible band $\mathcal{C}_J$ by estimating $P_0(f_0\notin\mathcal{C}_J)$. By selecting $J\in\mathcal{J}$ in an optimal fashion (e.g., through the Bayes Lepski's method), it will be seen that this quantity is bounded above up to some negligible terms by another tail probability of the form
\begin{align}\label{eq:p0}
P_0\left[\sup_{x\in[0,1]}\left|\frac{\boldsymbol{a}_J(x)^T\{\mathrm{E}_J(\boldsymbol{\theta}|\boldsymbol{Y})-\mathrm{E}_0\mathrm{E}_J(\boldsymbol{\theta}|\boldsymbol{Y})\}}{\sqrt{\boldsymbol{a}_J(x)^T\mathrm{Var}_0\mathrm{E}_J(\boldsymbol{\theta}|\boldsymbol{Y})\boldsymbol{a}_J(x)}}\right|>w_{\gamma,J}\right].
\end{align}
Again if we have a sharp upper bound $G(w_{\gamma,J})$ to the above and establish that $G(w_{\gamma,J})\leq U(w_{\gamma,J})\leq\gamma$, then Bayesian credible band $\mathcal{C}_J$ of credibility at least $1-\gamma$ corresponds to a frequentist confidence band of confidence level at least $1-\gamma$. Therefore to achieve this almost exact correspondence between Bayesian and frequentist quantification of uncertainty, it is then imperative that we are able to get sharp upper bounds for tail probabilities of ``standardized" Gaussian processes in supremum norm.

\subsection{Tail probabilities and the volume-of-tube formula}
In this subsection, let us temporarily leave the regression framework and work with general Gaussian processes generated through a random series. We will divide our ordeal into 3 steps and will revisit the issue of Bayesian credible band construction in the next subsection.\vspace{10pt}

\noindent\textit{(i) Recast statistical problem for geometric formulation}\\
Let $\boldsymbol{X}\sim\mathrm{N}_J(\boldsymbol{\mu},\boldsymbol{\Sigma})$ and define
\begin{align*}
W:=\sup_{x\in[0,1]}\frac{\boldsymbol{a}_J(x)^T(\boldsymbol{X}-\boldsymbol{\mu})}{\sqrt{\boldsymbol{a}_J(x)^T\boldsymbol{\Sigma}\boldsymbol{a}_J(x)}}.
\end{align*}
Our task is then to estimate as accurately as possible $P(W>w)$, which with appropriate $\boldsymbol{\mu}$ and $\boldsymbol{\Sigma}$ corresponds to the two tail probabilities mentioned above in \eqref{eq:credc} and \eqref{eq:p0}(since $P(|W|>w)\leq2P(W>w)$). We further decompose $W=QR$ such that $R^2=(\boldsymbol{X}-\boldsymbol{\mu})^T\boldsymbol{\Sigma}^{-1}(\boldsymbol{X}-\boldsymbol{\mu})$ and
\begin{align*}
Q&=\sup_{x\in[0,1]}\frac{\boldsymbol{a}_J(x)^T(\boldsymbol{X}-\boldsymbol{\mu})}{\sqrt{\boldsymbol{a}_J(x)^T\boldsymbol{\Sigma}\boldsymbol{a}_J(x)}}\frac{1}{\sqrt{(\boldsymbol{X}-\boldsymbol{\mu})^T\boldsymbol{\Sigma}^{-1}(\boldsymbol{X}-\boldsymbol{\mu})}}\\
&=\sup_{x\in[0,1]}\underbrace{\frac{\left[\boldsymbol{\Sigma}^{1/2}\boldsymbol{a}_J(x)\right]^T}{\left\|\boldsymbol{\Sigma}^{1/2}\boldsymbol{a}_J(x)\right\|}}_{\beta_J(x)}
\underbrace{\frac{\boldsymbol{\Sigma}^{-1/2}(\boldsymbol{X}-\boldsymbol{\mu})}{\left\|\boldsymbol{\Sigma}^{-1/2}(\boldsymbol{X}-\boldsymbol{\mu})\right\|}}_U=:\sup_{x\in[0,1]}\langle\beta_J(x),U\rangle.
\end{align*}
Since $\boldsymbol{\Sigma}^{-1/2}(\boldsymbol{X}-\boldsymbol{\mu})\sim\mathrm{N}_J(\boldsymbol{0},\boldsymbol{I})$, it follows that $U$ is uniformly distributed on the unit sphere $S^{J-1}$; while $\beta_J$ is a curve that maps $[0,1]$ to $S^{J-1}$. Observe that $Q$ is independent of $R$, with $R^2$ distributed as chi-square with $J$ degrees of freedom $\chi^2_J$. Therefore,
\begin{align}\label{eq:w}
P(W>w)=\int_{w}^\infty P\left(\sup_{x\in[0,1]}\langle\beta_J(x),U\rangle>wr^{-1}\right)p_{\chi^2_J}(r)dr,
\end{align}
where $p_{\chi_J^2}$ is the density function of $\chi^2_J$, and we compute the tail probability of $W$ by using chi-square mixtures. Thus, our task is to evaluate the probability in the integrand and this is where concepts from geometry comes into play.\vspace{10pt}

\noindent\textit{(ii) Tubular neighborhoods and the geometric connection}\\
Let $u$ be some point on the unit sphere $S^{J-1}$, and we define a metric $d$ to measure distance between $u$ and the curve $\beta_J$ by finding the nearest point on $\beta_J$ to $u$ in $\ell_2$-norm, such that $d(u,\beta_J)^2=\inf_{x\in[0,1]}\|u-\beta_J(x)\|$. Since $\|u\|^2=1=\|\beta_J(x)\|^2$ for any $x$, this squared distance can also be written as $2(1-\sup_{x\in[0,1]}\langle\beta_J(x),u\rangle)$. Keeping this relation in mind, let us define tubular neighborhoods parameterized by $\theta$ around $\beta_J$ as
\begin{align*}
\mathcal{T}_{\theta}=\{u\in S^{J-1}:d(u,\beta_J)\leq[2(1-\cos{\theta})]^{1/2}\}=\left\{u\in S^{{J-1}}:\sup_{x\in[0,1]}\langle\beta_J(x),u\rangle\geq\cos{\theta}\right\},
\end{align*}
and we call $\theta$ the geodesic or angular radius that controls the (geodesic) width of these tubes. To visualize these objects, see Figure \ref{fig:tube}.
\begin{figure}[h!]
\centering
\begin{subfigure}[b]{0.32\textwidth}
\includegraphics[width=8cm,trim={4cm 17cm 0 2cm},clip]{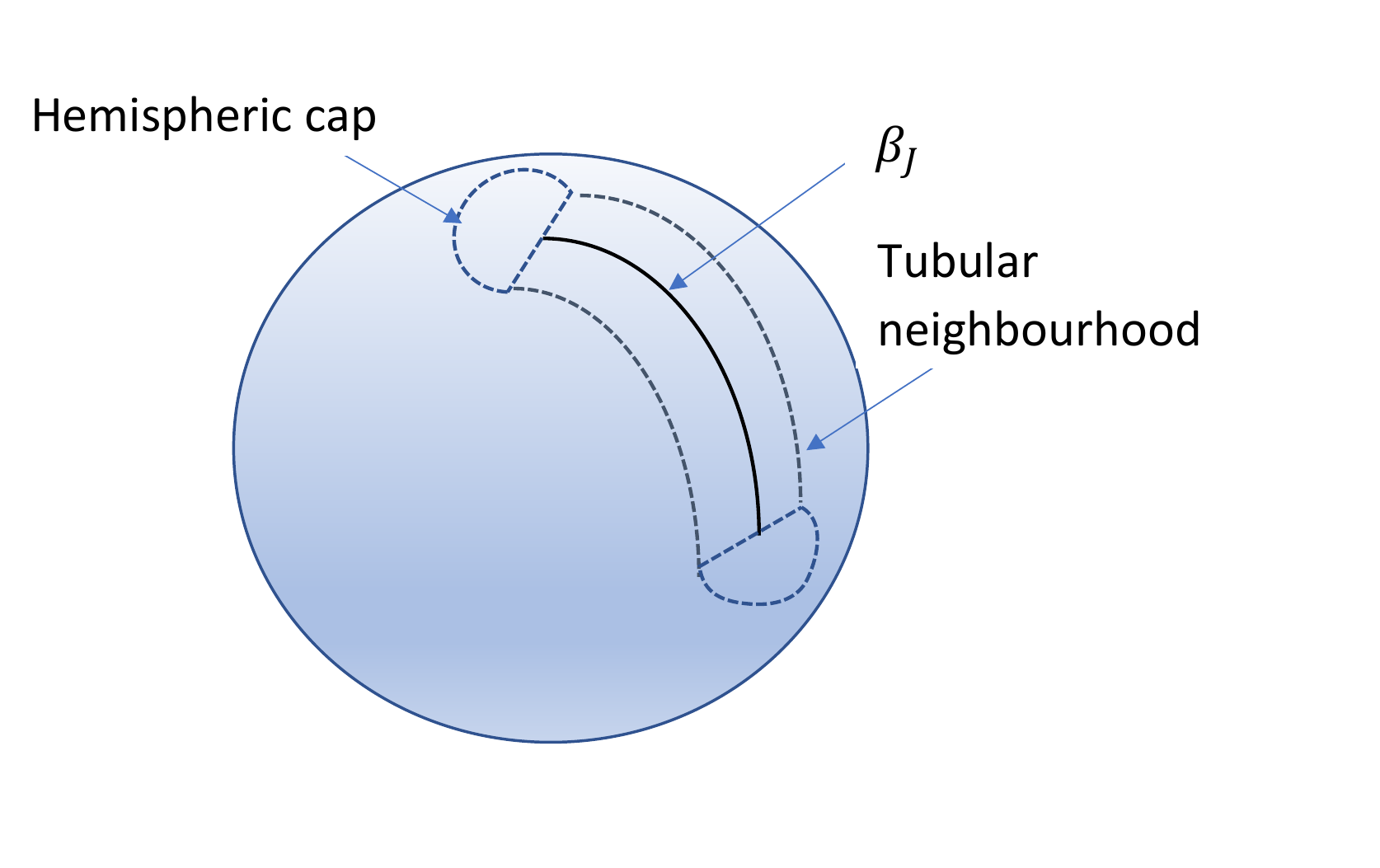}
\caption{Tubular neighborhood on unit sphere}
\label{fig:tubea}
\end{subfigure}
\qquad\quad\quad
\begin{subfigure}[b]{0.32\textwidth}
\includegraphics[width=8cm,trim={0 15cm 0 2cm},clip]{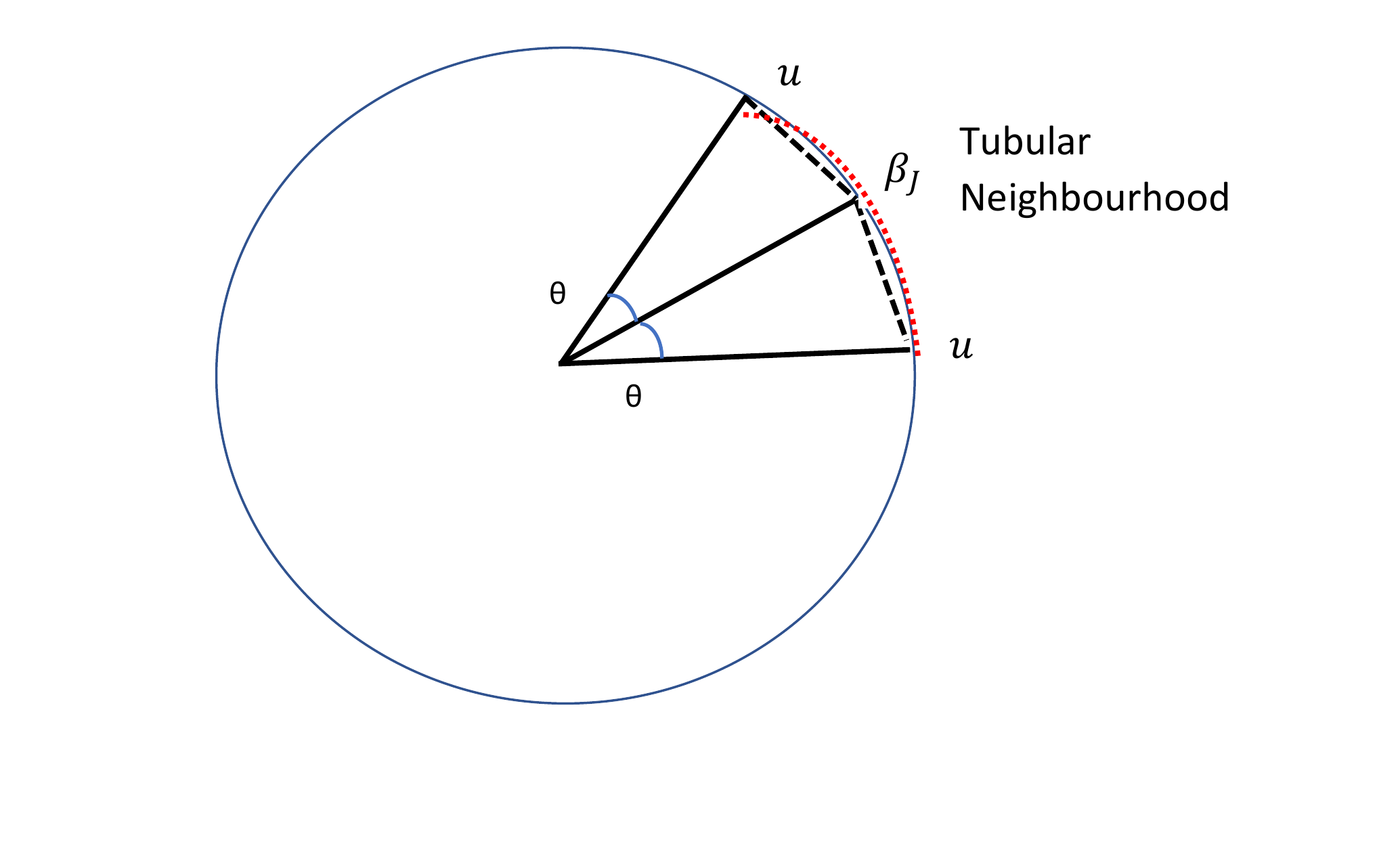}
\caption{Cross-section of tubular neighborhood}
\end{subfigure}
\caption{Tubular neighborhood around curve $\beta_J$}
\label{fig:tube}
\end{figure}

For a set $A$ in $\mathbb{R}^J$, write $V(A)$ to be the volume of $A$ with respect to the Lebesgue measure. Now recall that $V(S^{J-1})=2\pi^{J/2}/\Gamma(J/2)$. Let $U$ be a random variable uniformly distributed on $S^{J-1}$ (such as the $U$ in \eqref{eq:w} above), then
\begin{align}\label{eq:cprod}
P\left(\sup_{x\in[0,1]}\langle\beta_J(x),U\rangle\geq\cos{\theta}\right)=P(U\in\mathcal{T}_\theta)=\frac{V(\mathcal{T}_{\theta})}{V(S^{J-1})}=\frac{\Gamma{(J/2)}}{2\pi^{J/2}}V(\mathcal{T}_{\theta}),
\end{align}
and what remains is to find the volume of the tube $\mathcal{T}_{\theta}$.\vspace{10pt}

\noindent\textit{(iii) Volume-of-tube formula}\\
The volume of $\mathcal{T}_{\theta}$ depends on the shape the tube can possibly take, and this in turn is affected by 3 situations: $\beta_J$ is closed (a loop), or open with endpoints at $0$ and $1$, or that the tube $\mathcal{T}_{\theta}$ self overlaps. Among these circumstances, the maximal volume is achieved when $\beta_J$ is open with boundaries. In that case, the enveloping tube has 3 pieces (see Figure \ref{fig:tubea}), the main ``cylinder" around the curve whose volume is the curve's length $|\beta_J|$ times its cross-sectional area, and two hemispheric caps at the end points. The exact formula for this upper bound was calculated by \citep{naiman1986}:
\begin{theorem}[Naiman]\label{th:tube}
Let $\beta_J:[0,1]\rightarrow S^{J-1}$ be a continuously differentiable curve, with nowhere vanishing first derivative, and of finite length. Then we have for any $0\leq\theta\leq2\pi$,
\begin{align*}
V(\mathcal{T}_{\theta})\leq|\beta_J|V(\mathcal{B}_{J-2})\sin^{J-2}{\theta}+V(S^{J-2})\int_{\cos{\theta}}^1(1-z^2)^{(J-3)/2}dz,
\end{align*}
where $|\beta_J|$ is the arc length of $\beta_J$ and $\mathcal{B}_{J-2}$ is the unit ball in $\mathbb{R}^{J-2}$.
\end{theorem}
This volume inequality is sharp, in fact for open $\beta_J$, the above will be an equality when $\theta$ is smaller than some critical radius, which is the threshold for tube self-overlap. Since $V(\mathcal{B}_{J-2})=\pi^{(J-2)/2}/\Gamma(J/2)$ and $V(S^{J-2})=2\pi^{(J-1)/2}/\Gamma[(J-1)/2]$, we can perform a change of variable $u=z^2$ in the integral above and write \eqref{eq:cprod} as
\begin{align*}
P\left(\sup_{x\in[0,1]}\langle\beta_J(x),U\rangle\geq\cos{\theta}\right)\leq\frac{|\beta_J|}{2\pi}\sin^{J-2}{\theta}+\frac{\Gamma(J/2)}{2\pi^{\frac{1}{2}}\Gamma[(J-1)/2]}P\left[B\left(\frac{1}{2},\frac{J-1}{2}\right)\geq\cos^2{\theta}\right]
\end{align*}
where $B(\alpha_1,\alpha_2)$ is a beta distributed random variable with shape and scale parameters $\alpha_1$ and $\alpha_2$ respectively. Choose $\theta$ such that $\cos{\theta}=wr^{-1}$ and note that $\sin{\theta}=(1-\cos^2{\theta})^{1/2}=(1-w^2r^{-2})^{1/2}$, if we plug this bound back to \eqref{eq:w} and compute the chi-square mixtures, we will arrive at
\begin{align}\label{eq:tube}
P\left(\sup_{x\in[0,1]}\frac{\boldsymbol{a}_J(x)^T(\boldsymbol{X}-\boldsymbol{\mu})}{\sqrt{\boldsymbol{a}_J(x)^T\boldsymbol{\Sigma}\boldsymbol{a}_J(x)}}>w\right)\leq\frac{|\beta_J|}{2\pi}e^{-w^2/2}+1-\Phi(w),
\end{align}
with $|\beta_J|=\int_0^1\|\beta_J^{'}(x)\|dx$ for $\beta_J(x)=\boldsymbol{\Sigma}^{1/2}\boldsymbol{a}_J(x)/\|\boldsymbol{\Sigma}^{1/2}\boldsymbol{a}_J(x)\|$ and $\Phi(\cdot)$ is the cumulative distribution function of a standard normal.

\subsection{Variable width Bayesian credible bands}
Let us return to our variable width credible band in the beginning. For $J\in\mathcal{J}$, we construct
\begin{align*}
\mathcal{C}_J:=\left\{f:|f(x)-\mathrm{E}_J[f(x)|\boldsymbol{Y}]|\leq w_{\gamma,J}\sqrt{\mathrm{Var}_J[f(x)|\boldsymbol{Y}]},\forall x\in[0,1]\right\}.
\end{align*}
Motivated by the previous exposition on volume of tubes, for any $J\in\mathcal{J}$ and a given credibility level $1-\gamma$, we choose $w_{\gamma,J}$ such that it solves
\begin{align}
\gamma=\frac{|\beta_J|}{\pi}e^{-w_{\gamma,J}^2/2}+2[1-\Phi(w_{\gamma,J})]\label{eq:gamma},
\end{align}
where the arc length $|\beta_J|=\int_0^1\|\beta_J^{'}(x)\|dx$ under our hierarchical Bayes prior is
\begin{align}
|\beta_J|&=\int_0^1\left\{\frac{d}{dx}\left(\frac{\boldsymbol{M}^{1/2}\boldsymbol{a}_J(x)^T}{[\boldsymbol{a}_J(x)^T\boldsymbol{M}\boldsymbol{a}_J(x)^T]^{1/2}}\right)
\frac{d}{dx}\left(\frac{\boldsymbol{M}^{1/2}\boldsymbol{a}_J(x)}{[\boldsymbol{a}_J(x)^T\boldsymbol{M}\boldsymbol{a}_J(x)^T]^{1/2}}\right)\right\}^{1/2}dx\nonumber\\
&=\int_0^1\frac{[\boldsymbol{a}_J(x)^T\boldsymbol{M}\boldsymbol{R}^T\boldsymbol{M}\boldsymbol{RM}\boldsymbol{a}_J(x)]^{1/2}}
{[\boldsymbol{a}_J(x)^T\boldsymbol{M}\boldsymbol{a}_J(x)]^{3/2}}dx,\label{eq:beta}
\end{align}
for $\boldsymbol{M}:=(\boldsymbol{A}^T\boldsymbol{A}+\boldsymbol{\Omega}^{-1})^{-1}$ and $\boldsymbol{R}:=\boldsymbol{\dot{a}}_J(x)\boldsymbol{a}_J(x)^T-\boldsymbol{a}_J(x)\boldsymbol{\dot{a}}_J(x)^T$ where we denote $\boldsymbol{\dot{a}}_J(x)=(a^{'}_1(x),\dotsc,a^{'}_J(x))^T$. To complete our construction, we choose $J=\widehat{j}_n$ by the Bayes Lepski's stopping rule given in \eqref{eq:optimalj} to give a $1-\gamma$ adaptive simultaneous credible band $\mathcal{C}_{\widehat{j}_n}$.

Now by setting $\boldsymbol{\mu}=\mathrm{E}_{\widehat{j}_n}(\boldsymbol{\theta}|\boldsymbol{Y})$ and $\boldsymbol{\Sigma}=\mathrm{Var}_{\widehat{j}_n}(\boldsymbol{\theta}|\boldsymbol{Y})$, we can use the tail estimate of \eqref{eq:tube} to bound \eqref{eq:credc} and ensure that our credible band $\mathcal{C}_{\widehat{j}_n}$ has at least $1-\gamma$ credibility:
\begin{align*}
\Pi\left[\sup_{0\leq x\leq 1}\frac{|f(x)-\mathrm{E}_{\widehat{j}_n}[f(x)|\boldsymbol{Y}]|}{\sqrt{\mathrm{Var}_{\widehat{j}_n}[f(x)|\boldsymbol{Y}]}}\geq w_{\gamma,\widehat{j}_n}\middle|\boldsymbol{Y}\right]\leq\frac{|\beta_{\widehat{j}_n}|}{\pi}e^{-w_{\gamma,\widehat{j}_n}^2/2}+2[1-\Phi(w_{\gamma,\widehat{j}_n})]=\gamma.
\end{align*}
Is $\mathcal{C}_{\widehat{j}_n}$ also a confidence band of level at least $1-\gamma$? The theorem below answers this in the affirmative. Let us define another curve $\beta_{0,J}:[0,1]\rightarrow S^{J-1}$ such that its arc length has the same functional form as in \eqref{eq:beta} but with $\boldsymbol{M}$ replaced by $\boldsymbol{M}_0:=(\boldsymbol{A}^T\boldsymbol{A}+\boldsymbol{\Omega}^{-1})^{-1}\boldsymbol{A}^T\boldsymbol{A}(\boldsymbol{A}^T\boldsymbol{A}+\boldsymbol{\Omega}^{-1})^{-1}$.

\begin{theorem}[Credible band coverage]\label{th:credible}
Suppose the following 3 assumptions hold ($n\to\infty$):
\begin{enumerate}
\item There exists a set $\mathcal{F}$ such that $\inf_{f_0\in\mathcal{G}\bigcap\mathcal{F}}P_0(\widehat{j}_n\geq j_n^{*}/\kappa)\rightarrow1$ for some constant $\kappa\geq1$.
\item The arc length $|\beta_J|\geq\max\{|\beta_{0,J}|,CJ\}$ for any $J\in\mathcal{J}$ and some constant $C>0$.
\item $(\log{\widehat{j}_n})^{-1/2}\left(\inf_{x\in[0,1]}\|\boldsymbol{a}_{\widehat{j}_n}(x)\|\right)^{-1}\left[l_{1,\widehat{j}_n}\lambda^{-1/2}_{\boldsymbol{A},\widehat{j}_n}+\lambda^{1/2}_{\boldsymbol{A},\widehat{j}_n}h(\widehat{j}_n)\right]=o_{P_0}(1)$ under the event $\{j_n^{*}/\kappa\leq\widehat{j}_n\leq j_n^{*}\}$.
\end{enumerate}
Then as $n\rightarrow\infty$, we have
\begin{align*}
\inf_{f_0\in\mathcal{G}\bigcap\mathcal{F}}P_0\left(f_0\in\mathcal{C}_{\widehat{j}_n}\right)\geq1-\gamma+o(1).
\end{align*}
\end{theorem}
The inequality in the coverage statement dictates that this is a conservative procedure. The 3 conditions are satisfied for the most commonly used basis systems and we verify them using B-spline and wavelets in the next two sections. The first condition says that the Bayes Lepski's rule cannot stop too late and choose $J$ that is too small as to make the approximation bias unmanageable. As we will discuss in detail later in Section \ref{sec:uncertain}, this give rise to the counterintuitive notion that we need a certain amount of error in order to quantify uncertainty honestly. The second condition says that the arc length under our model/posterior must be longer than the one under the true distribution, so that credible level under the posterior can be translated to confidence level when coverage is measured. At the same time, the model arc length must be larger (up to some constant) than $J\in\mathcal{J}$ so that in the last condition, the size of the quantile $w_{\gamma,J}$ is of the order larger than the standardized approximation bias, where the seemingly technical appearance is the result of reducing the aforementioned condition to one formulated by the basis components only.

Therefore, if the Bayes Lepski rule do not stop too early or too late, we have the following result on the size of the credible band generated through the volume of tube method.
\begin{corollary}[Length of radius function]\label{cor:radius}
Let $r(x):=w_{\gamma,\widehat{j}_n}\sqrt{\mathrm{Var}_{\widehat{j}_n}[f(x)|\boldsymbol{Y}]}$ be the radius function for the variable width credible band. Let assumption 1 of Theorem \ref{th:credible} hold for some set $\mathcal{F}$ and constant $\kappa\geq1$ and assume in addition that $|\beta_J|\asymp J$ for any $J\in\mathcal{J}$. Then as $n\rightarrow\infty$,
\begin{align*}
\inf_{f_0\in\mathcal{G}\bigcap\mathcal{F}}P_0\left[\|\boldsymbol{a}_{j_n^{*}/\kappa}(x)\|\sqrt{\frac{\log{(j_n^{*})}}{\lambda_{\boldsymbol{A},j_n^{*}/\kappa}}}\lesssim r(x)\lesssim\|\boldsymbol{a}_{j_n^{*}}(x)\|\sqrt{\frac{\log{(j_n^{*})}}{\lambda_{\boldsymbol{A},j_n^{*}}}},\forall x\in[0,1]\right]\rightarrow1.
\end{align*}
\end{corollary}
Since $\kappa$ is a constant not depending on $n$, the upper and lower bounds will have the same asymptotic order, this further implies that the radius cannot concentrate to quickly to the $0$ function, and it must do so at least at a certain rate indicated above.

To showcase the applicability of the methods discussed, let us consider two special but widely used classes of basis functions, with the B-splines as the non-orthonormal example and the wavelets serving as the canonical orthonormal example.

\section{Application I: B-splines bases}\label{sec:bspline}
Suppose we project $f$ onto the $J$-dimensional space of polynomial splines spanned by B-splines basis functions $B_{j,q}(x),j=1,\dotsc,J$. Each of this basis is a piecewise polynomial defined on the knot sequence $\mathcal{T}=\{0=t_0<t_1,\dotsc,t_N<t_{N+1}=1\}$, such that $B_{j,q}$ restricted to a knot interval is a polynomial of order $q$ (or degree $q-1$), and $B_{j,q}$ is $q-2$ times differentiable at the knot points. We select the knots such that their distribution is quasi-uniform, in the sense that $\max_{1\leq k\leq N+1}(t_k-t_{k-1})\lesssim\min_{1\leq k\leq N+1}(t_k-t_{k-1})$, i.e., the max knot increment is of the same order as the min knot increment. Here $N$ is the number of interior knots and it is related to the number of basis by $J=N+q$.

Let $\boldsymbol{b}_{J,q}(x)=(B_{1,q}(x),\dotsc,B_{J,q}(x))^T$. Our hierarchical prior can then be formulated as $f(x)=\boldsymbol{b}_{J,q}(x)^T\boldsymbol{\theta}$ with $\boldsymbol{\theta}|\sigma\sim\mathrm{N}_J(\boldsymbol{\eta},\sigma^2\boldsymbol{\Omega})$ and $\sigma^2$ estimated by empirical Bayes \eqref{eq:sigma}. Construct the B-spline basis matrix as $\boldsymbol{B}=(\boldsymbol{b}_{J,q}(X_1)^T,\dotsc,\boldsymbol{b}_{J,q}(X_n))^T$. By Theorem 4.18 of \citep{lschumaker}, we know that B-splines are linearly independent and hence $\boldsymbol{B}^T\boldsymbol{B}$ is invertible. For a fixed cumulative distribution function $F(x)$ with positive and continuous density on $[0,1]$, let us choose the covariates such that
\begin{align*}
\sup_{x\in[0,1]}|F_n(x)-F(x)|=o(N^{-1}),
\end{align*}
where $F_n(x)=n^{-1}\sum_{i=1}^n\mathbbm{1}_{[0,X_i]}(x)$ is the empirical distribution of $(X_1,\dotsc,X_n)$. Then by Lemma A.9 of \citep{yoo2016}, there exist constants $0<C_1\leq C_2<\infty$ such that
\begin{align}\label{eq:BB}
C_1(n/J)\leq\lambda_{\mathrm{min}}(\boldsymbol{B}^T\boldsymbol{B})\leq\lambda_{\mathrm{max}}(\boldsymbol{B}^T\boldsymbol{B})\leq C_2(n/J),
\end{align}
and hence $\lambda_{\boldsymbol{B},J}=n/J$ in this case. Now since $\sum_{j=1}^JB_{j,q}(x)=1$ (partition of unity) and $0\leq B_{j,q}(x)\leq1$ for any $j$ and $x$, it follows that $l_{1,J}=1$. Therefore if $J\leq n$, then $\lambda_{\boldsymbol{B},J}\geq l_{1,J}^2$. For any $X_i\in[0,1]$, we can always find a $k$ such that $X_i\in[t_{k-1},t_k]$, and since B-splines have compact support, this implies that only $q$-consecutive $B_{k,q}(X_i),\dotsc,B_{k+q-1,q}(X_i)$ are nonzero. Therefore if $|u-v|>q$, then $B_{u,q}(X_i)B_{v,q}(X_i)=0$ and hence $(\boldsymbol{B}^T\boldsymbol{B})_{u,v}=0$. As a result, $\boldsymbol{B}^T\boldsymbol{B}$ is $q$-banded.

By Theorem 6.18 of \citep{lschumaker}, there is a linear operator $K_J$ that maps $f$ onto the space of $J$-dimensional polynomial splines, and $K_J$ reproduces polynomials of order $q$. Hence we choose the lower point in the candidate set as $j_{\mathrm{min}}=(n/\log{n})^{1/(2q+1)}$, and this further implies that we adapt $\alpha$ up to $q$. Thus, we estimate the optimal number of B-splines bases as (for some large enough $\tau$)
\begin{align*}
\widehat{j}_n=\min\left\{j\in\mathcal{J}:\|\mathrm{E}_j(f|\boldsymbol{Y})-\mathrm{E}_i(f|\boldsymbol{Y})\|_\infty\leq\tau\widehat{\sigma}_i\sqrt{\frac{i\log{i}}{n}},\forall i>j,i\in\mathcal{J}\right\}.
\end{align*}

Let $\mathcal{H}^{\alpha}$ denote the H\"{o}lder space of order $\alpha>0$ consisting of functions $g:[0,1]\rightarrow\mathbb{R}$ such that $\|g\|_{\mathcal{H}^{\alpha}}<\infty$, where the H\"{o}lder norm is defined as
\begin{align*}
\|g\|_{\mathcal{H}^{\alpha}}=\max_{0\leq k<[\alpha]}\sup_{x\in[0,1]}|g^{(k)}(x)|+\sup_{x,y\in[0,1]:x\neq y}\frac{|g^{([\alpha])}(x)-g^{([\alpha])}(y)|}{|x-y|^{\alpha-[\alpha]}},
\end{align*}
and $[\alpha]$ is the integer part of $\alpha$. We then take $\mathcal{G}=\{g:\|g\|_{\mathcal{H}^{\alpha}}\leq R\}=:\mathcal{H}^{\alpha}(R)$ to be the H\"{o}lder ball of radius $R>0$. By Theorem 22 in Chapter XII of \citep{deBoor} or Theorem 6.31 of \citep{lschumaker}, we know that for any $g\in\mathcal{H}^{\alpha}(R)$, there exists a $\boldsymbol{\theta}_0\in\mathbb{R}^J$ and a constant $C_q>0$ depending on $q$ such that
\begin{align}
\|g-\boldsymbol{b}_{J,q}(\cdot)^T\boldsymbol{\theta}_0\|_\infty\leq C_q\|g\|_{\mathcal{H}^{\alpha}}J^{-\alpha}.
\end{align}
Thus $h(x)=x^{-\alpha}$ for the B-splines case and it follows that $h(j_{\mathrm{min}})=(\log{n}/n)^{\alpha/(2q+1)}\to0$ as $n\to\infty$. By (7.2) of \citep{yoo2016}, we have $\|\boldsymbol{\theta}_0\|_{\infty}<\infty$. Suppose the true regression function $f_0\in\mathcal{H}^{\alpha}(R)$ for some unknown $\alpha>0$. Then $j_n^{*}$ defined in \eqref{eq:j0} balances the posterior mean bias $\|f_0\|_{\mathcal{H}^{\alpha}}J^{-\alpha}$ on one side with the posterior standard deviation $\sqrt{J\log{(J)}/n}$ on the other, and we deduce that $j_n^{*}\asymp\|f_0\|_{\mathcal{H}^{\alpha}}^{2/(2\alpha+1)}(n/\log{n})^{1/(2\alpha+1)}$. Therefore Theorem \ref{th:rate} when specializing to B-splines bases yield
\begin{align}\label{eq:Brate}
\sup_{0<\alpha\leq q}\sup_{f_0\in\mathcal{H}^{\alpha}(R)}\mathrm{E}_0\Pi\left[\|f-f_0\|_\infty>\xi(\log{n}/n)^{\alpha/(2\alpha+1)}\middle|\boldsymbol{Y}\right]\lesssim(\log{n}/n)^{M}
\end{align}
for some constants $R,\xi,M>0$. When compared to Theorem 4.4 of \citep{yoo2016}, our result is fully rate adaptive and the $\xi$ appearing in the radius is a constant rather than some arbitrary sequence going to infinity.

We now turn to the problem of constructing credible bands. Using the volume-of-tube method discussed in the previous section, we construct
\begin{align}\label{eq:bsplineband}
\mathcal{B}_n:=\left\{f:|f(x)-\mathrm{E}_{\widehat{j}_n}[f(x)|\boldsymbol{Y}]|\leq w_{\gamma,\widehat{j}_n}\sqrt{\mathrm{Var}_{\widehat{j}_n}[f(x)|\boldsymbol{Y}]},\forall x\in[0,1]\right\},
\end{align}
where the posterior mean and variance with the Bayes Lepski's stopping rule $J=\widehat{j}_n$ is
\begin{align}
\mathrm{E}_{\widehat{j}_n}[f(x)|\boldsymbol{Y}]&=\boldsymbol{b}_{\widehat{j}_n,q}(x)^T(\boldsymbol{B}^T\boldsymbol{B}+\boldsymbol{\Omega}^{-1})^{-1}(\boldsymbol{B}^T\boldsymbol{Y}+\boldsymbol{\Omega}^{-1}\boldsymbol{\eta}),\\ \mathrm{Var}_{\widehat{j}_n}[f(x)|\boldsymbol{Y}]&=\widehat{\sigma}_{\widehat{j}_n}^2\boldsymbol{b}_{\widehat{j}_n,q}(x)^T(\boldsymbol{B}^T\boldsymbol{B}+\boldsymbol{\Omega}^{-1})^{-1}\boldsymbol{b}_{\widehat{j}_n,q}(x),
\end{align}
and we choose $w_{\gamma,\widehat{j}_n}$ such that it solves $\gamma=|\beta_{\widehat{j}_n}|\pi^{-1}e^{-w_{\gamma,\widehat{j}_n}^2/2}+2[1-\Phi(w_{\gamma,\widehat{j}_n})]$. The expression for $|\beta_{\widehat{j}_n}|$ can be found by substituting $\boldsymbol{b}_{\widehat{j}_n,q}(x)$ for $\boldsymbol{a}_{\widehat{j}_n}(x)$ and $\boldsymbol{B}$ for $\boldsymbol{A}$ in \eqref{eq:beta}. For B-splines however, we can further use (8) of Chapter X in \citep{deBoor} to write explicitly $d\boldsymbol{b}_{\widehat{j}_n,q}(x)/dx$ and this in turn gives
\begin{align}\label{eq:betaB}
|\beta_{\widehat{j}_n}|=\int_0^1\frac{[\boldsymbol{b}_{\widehat{j}_n,q}(x)^T\boldsymbol{M}\boldsymbol{R}^T\boldsymbol{M}\boldsymbol{RM}\boldsymbol{b}_{\widehat{j}_n,q}(x)]^{1/2}}
{[\boldsymbol{b}_{\widehat{j}_n,q}(x)^T\boldsymbol{M}\boldsymbol{b}_{\widehat{j}_n,q}(x)]^{3/2}}dx
\end{align}
for $\boldsymbol{M}:=(\boldsymbol{B}^T\boldsymbol{B}+\boldsymbol{\Omega}^{-1})^{-1}$ and $\boldsymbol{R}:=\boldsymbol{Wb}_{\widehat{j}_n,q-1}(x)\boldsymbol{b}_{\widehat{j}_n,q}(x)^T-\boldsymbol{b}_{\widehat{j}_n,q}(x)\boldsymbol{b}_{\widehat{j}_n,q-1}(x)^T\boldsymbol{W}^T$, where $\boldsymbol{W}$, a matrix of dimension $\widehat{j}_n\times(\widehat{j}_n-1)$, is $q-1$ times
\begin{align}\label{eq:matrix}
\begin{pmatrix}
-(t_1-t_{2-q})^{-1}&0&0&\cdots&0&0\\
(t_1-t_{2-q})^{-1}& -(t_2-t_{3-q})^{-1}&0&\cdots&0&0\\
0&(t_2-t_{3-q})^{-1}&-(t_3-t_{4-q})^{-1}&\cdots&0&0\\
\vdots&\vdots&\vdots&\ddots&\vdots&\vdots\\
0&0&0&\cdots&0&(t_{\widehat{j}_n-1}-t_{\widehat{j}_n-q})^{-1}
\end{pmatrix}.
\end{align}
To apply Theorem \ref{th:credible} and conclude that $\mathcal{B}_n$ has high coverage probability in the frequentist sense, we proceed to verify the 3 conditions needed.

Let us first verify the last condition, which dictates that the quantile $w_{\gamma,\widehat{j}_n}$ is of the order larger than the standardized posterior mean bias. Now by Lemma \ref{lem:b2}, we will have $\inf_{x\in[0,1]}\|\boldsymbol{b}_{\widehat{j}_n,q}(x)\|\geq q^{-1}$. Note that under the event $\{j_n^{*}/\kappa\leq\widehat{j}_n\leq j_n^{*}\}$ with $\kappa\geq1$ coming from the first assumption, we deduce that $(\widehat{j}_n/n)^{1/2}\asymp\widehat{j}_n^{-\alpha}$ and $\widehat{j}/n=o_{P_0}(1)$. Thus if we intersect with the aforementioned event, the third condition reads
\begin{align*}
(\log{\widehat{j}_n})^{-1/2}q\left[\sqrt{\frac{\widehat{j}_n}{n}}+\sqrt{\frac{n}{\widehat{j}_n}}\widehat{j}_n^{-\alpha}\right]\lesssim(\log{n})^{-1/2}(1+o_{P_0}(1))=o_{P_0}(1),
\end{align*}
as $n\rightarrow\infty$. Now for the second arc length condition, Lemmas \ref{lem:cgamma} and \ref{lem:curve} below jointly say that B-splines do indeed satisfy the arc length inequality for any $J\in\mathcal{J}$ as $n\rightarrow\infty$.

\begin{lemma}\label{lem:cgamma}
Let $|\beta_J|=\int_0^1\|\beta_J^{'}(x)\|dx$ be the arc length of the curve $\beta_J:[0,1]\rightarrow S^{J-1}$ given in \eqref{eq:betaB}. Then as $n\rightarrow\infty$,
\begin{align*}
|\beta_J|\asymp J,\quad\text{for any $J\in\mathcal{J}$}.
\end{align*}
\end{lemma}

\begin{lemma}\label{lem:curve}
There exist constants $\xi_1,\xi_2,\xi_3>0$ such that
\begin{align*}
|\beta_{0,J}|\leq\frac{\left[1-\xi_1(J/n)+\xi_2(J/n)^2\right]^{1/2}}{\left[1-\xi_3(J/n)\right]^{3/2}}|\beta_J|.
\end{align*}
Consequently for any $J\in\mathcal{J}$, $|\beta_{0,J}|\leq|\beta_J|$ as $n\to\infty$.
\end{lemma}

It now remains to verify the first condition dictating a lower bound for the Bayes Lepski's estimate $\widehat{j}_n$. To this end, let us describe the set $\mathcal{F}$ suited for this purpose. First, some notations. Let $\delta_q$ be some positive constant depending only on $q$ and $\mathcal{T}$ a quasi-uniform partition of $[0,1]$. Denote $\Delta_{\mathcal{T}}$ to be the maximum knot increment associated with this partition $\mathcal{T}$. Furthermore, let $\mathbb{T}$ be the collection of all quasi-uniform partitions of $[0,1]$. We then take our set $\mathcal{F}$ to be
\begin{align}\label{eq:similar}
\mathcal{F}^{\alpha}=\bigcap_{\mathcal{T}\in\mathbb{T}}\left\{f:\inf_{p\in\mathcal{P}_q(\mathcal{T})}\|f-p\|_\infty\geq \delta_q\|f\|_{\mathcal{H}^{\alpha}}\Delta_{\mathcal{T}}^{\alpha}\right\},
\end{align}
where $0<\alpha\leq q$ and $\mathcal{P}_q(\mathcal{T})$ is the space of polynomial splines of order $q$ with knots in $\mathcal{T}$. This set together with the H\"{o}lder ball $\mathcal{H}^{\alpha}(R)$ then make it possible to establish the following lemma and at the same time fulfils the first lower bound condition.

\begin{lemma}\label{lem:lower}
Let $\kappa\in\mathbb{N}$ be some constant that depends on $\delta_q,\alpha$ and $\tau$ where $\tau$ is a large enough positive constant contained in the stopping rule $\widehat{j}_n$ of \eqref{eq:optimalj}. Then for any $R>0$, there exist constants $\mu>1$ and $Q>0$ such that
\begin{align*}
\sup_{0<\alpha\leq q}\sup_{f_0\in\mathcal{F}^{\alpha}}P_0(\widehat{j}_n<j_n^{*}/\kappa)\lesssim\frac{1}{j_{\mathrm{min}}^{\mu-1}}+j_n^{*}e^{-Qj_n^{*}},
\end{align*}
and consequently in view of Proposition \ref{lem:upper}, we have as $n\rightarrow\infty$,
\begin{align}\label{eq:uplow}
\inf_{0<\alpha\leq q}\inf_{f_0\in\mathcal{H}^{\alpha}(R)\bigcap\mathcal{F}^{\alpha}}P_0(j_n^{*}/\kappa\leq\widehat{j}_n\leq j_n^{*})\rightarrow1.
\end{align}
\end{lemma}
As a result, we can appeal to Theorem \ref{th:credible} and conclude that for any $R>0$,
\begin{align}
\inf_{0<\alpha\leq q}\inf_{f_0\in\mathcal{H}^{\alpha}(R)\bigcap\mathcal{F}^{\alpha}}P_0(f_0\in\mathcal{B}_n)\geq1-\gamma+o(1).
\end{align}
By Lemmas \ref{lem:cgamma}, \ref{lem:lower} and \ref{lem:b2}, we have the following B-splines version of Corollary \ref{cor:radius} on the size of the radius function
\begin{align}
\inf_{0<\alpha\leq q}\inf_{f_0\in\mathcal{H}^{\alpha}(R)\bigcap\mathcal{F}^{\alpha}}P_0\left[r(x)\asymp(\log{n}/n)^{\alpha/(2\alpha+1)},\forall x\in[0,1]\right]\rightarrow1.
\end{align}
We note that the uniform length of the radius is exactly equal up to some constant to the adaptive posterior contraction rate in \eqref{eq:Brate}, and we say that this Bayesian uncertainty quantification procedure is honest.

\subsection{The uncertainty principle and self-similar functions}\label{sec:uncertain}
In this subsection, we will investigate the function space $\mathcal{F}^{\alpha}$ in order to gain a deeper understanding as to its nature and its implication towards statistical estimation and uncertainty quantification. In a quick glance, the restriction imposed on $\mathcal{F}^{\alpha}$ involves a lower bound on the approximation power of polynomial splines, and it further suggests that we can only quantify uncertainty for functions that are at least a certain distance away from the space of splines.

Now if we were using posterior based on random series of order $q$ B-splines, and the true regression function $f_0$ turns out to be a polynomial spline of the same order. Then we ``hit the right spot" and our posterior will concentrate around $f_0$ at the optimal rate $(\log{n}/n)^{\alpha/(2\alpha+1)}$, given that we have correctly modelled the truth. However for this truth $f_0$, we have $\inf_{p\in\mathcal{P}_q(\mathcal{T})}\|f_0-p\|_\infty=0$ and there is nothing to prevent $\widehat{j}_n$ to stop too late and choose a $J\in\mathcal{J}$ that is too small. Consequently, we will not be able to quantify uncertainty honestly owing to the fact that the posterior mean bias has a much larger order than the posterior standard deviation. This example, although extreme, serves to shed light on a fundamental limitation in statistics. \textit{The more accurately we are able to model and estimate the truth, the less we are able to access its quality and statistical uncertainty and vice versa}. Limitation of this type is reminiscence to the uncertainty principle in quantum mechanics, where it states that the more precisely we know the position of a particle, the less precisely its momentum can be determined and vice versa.

If $f_0\in\mathcal{F}^{\alpha}$, we know that by Theorem 22 in Chapter XII of \citep{deBoor} or Theorem 6.31 of \citep{lschumaker}, there is a $p_0\in\mathcal{P}_q(\mathcal{T})$ and some constant $C_q>0$ depending only on $f_0$ and $q$ such that
\begin{align}\label{eq:approx}
\delta_q\|f_0\|_{\mathcal{H}^{\alpha}}\Delta_{\mathcal{T}}^{\alpha}\leq\inf_{p\in\mathcal{P}_q(\mathcal{T})}\|f_0-p\|_\infty\leq\|f_0-p_0\|_\infty\leq C_q\|f_0\|_{\mathcal{H}^{\alpha}}\Delta_{\mathcal{T}}^{\alpha}.
\end{align}
In this formulation, condition \eqref{eq:similar} can be interpreted as being self-similar in regularity across different levels of quasi-uniform knot partition $\mathcal{T}\in\mathbb{T}$, or in other words, functions that have the same approximation power by splines across different quasi-uniform knot partitions of $[0,1]$. In fact, it is even possible to redefine \eqref{eq:approx} in a recursice version so as to remove its dependence on $\alpha$. In view of Lemma 6.17 of \citep{lschumaker}, we can construct a sequence of quasi-uniform knot partitions $\mathcal{T}_1\subset\mathcal{T}_2\subset\cdots\subset\mathcal{T}_M$ such that $\Delta_{\mathcal{T}_{k+1}}/2\leq\Delta_{\mathcal{T}_{k}}\leq3\Delta_{\mathcal{T}_{k+1}}/2$. Thus it view \eqref{eq:approx}, it holds that
\begin{align*}
\inf_{p\in\mathcal{P}_q(\mathcal{T}_{k+1})}\|f_0-p\|_\infty\geq\delta_q\|f_0\|_{\mathcal{H}^{\alpha}}\Delta_{\mathcal{T}_{k+1}}^{\alpha}
&\geq\frac{\delta_q}{C_q}\inf_{p\in\mathcal{P}_q(\mathcal{T}_k)}\|f_0-p\|_\infty\left(\frac{\Delta_{\mathcal{T}_{k+1}}}{\Delta_{\mathcal{T}_k}}\right)^{\alpha}\\
&\geq Q\inf_{p\in\mathcal{P}_q(\mathcal{T}_k)}\|f_0-p\|_\infty,
\end{align*}
with $Q:=(2/3)^{\alpha}\delta_q/C_q$. Naturally, we expect that $\inf_{p\in\mathcal{P}_q(\mathcal{T}_{k+1})}\|f_0-p\|_\infty\leq\inf_{p\in\mathcal{P}_q(\mathcal{T}_k)}\|f_0-p\|_\infty$ since finer knot partitions will result in better approximations. Functions in $\mathcal{F}^{\alpha}$ however are able to invert this natural order by incurring an extra constant. Note that self-referencing in splines approximation error can be seen as the $L_\infty$-splines analogue of the polished tail condition introduced in Definition 3.2 of \citep{szabo}.

Now are sets like \eqref{eq:similar} a reasonable construction in practice? Here, we give two classes of functions that satisfy the lower bound of \eqref{eq:similar}.\vspace{10pt}

\noindent\textbf{Example 1:} Let $m\in\mathbb{N}$ be such that $t_{m+1}-t_m=\Delta_{\mathcal{T}}$. Let $\|g^{(\alpha)}\|_\infty<\infty$ for some $\alpha\in\mathbb{N}$, $g$ is not a polynomial spline of order $q$ or less, with $g$ and its derivatives up to $\alpha-1$ are $0$ at its boundary $0$ (or at $1$ with slight modification of the argument below).  We construct
\begin{align*}
f_0(x)=\begin{cases}
\Delta_{\mathcal{T}}^{\alpha}g\left(\dfrac{x-t_m}{t_{m+1}-t_m}\right),&\quad\text{for $t_m\leq x\leq t_{m+1}$},\\
0,&\quad\text{otherwise.}
\end{cases}
\end{align*}
Let us denote the $L_\infty$-Sobolev space $L_{\infty}^{\alpha}:=\{f:\sum_{k=0}^{\alpha}\|f^{(k)}\|_\infty<\infty\}$. By Taylor's expansion on $t_m$ and the assumption that $g$ and its derivatives up to $\alpha-1$ vanish at $0$, it follows that $f_0(x)=f_0^{(\alpha)}(\xi)(x-t_m)^{\alpha}/\alpha!$ for some $\xi$ in between $x$ and $t_m$. Therefore since our domain is $[0,1]$ and $\|f_0^{(\alpha)}\|_\infty<\infty$ because of $\|g^{(\alpha)}\|_\infty<\infty$, it follows that $f_0\in L_\infty^{\alpha}$. By Remark 1 of \citep{nickl2011}, we have the embedding $L_{\infty}^{\alpha}\subset\mathcal{B}^{\alpha}_{\infty,\infty}$, where the latter Besov space is equivalent to $\mathcal{H}^{\alpha}$. Therefore, we conclude $\|f_0\|_{\mathcal{H}^{\alpha}}\leq C$ for some constant $C>0$.

Define $\beta:=\inf_{p\in\mathcal{P}_q(\mathcal{T})}\|g-p\|_\infty$ and denote $\overline{p}(x):=\Delta^{\alpha}p[(x-t_m)/(t_{m+1}-t_m)]$. Because $g$ is not a spline of order $q$, we have by a change of variable that
\begin{align*}
0<\beta\leq\sup_{x\in[0,1]}|g(x)-p(x)|=\Delta_{\mathcal{T}}^{-\alpha}\sup_{x\in[t_m,t_{m+1}]}|f(x)-\overline{p}(x)|.
\end{align*}
Now since the inequality above holds for any arbitrary $p\in\mathcal{P}_q(\mathcal{T})$ and $\overline{p}$ is also a polynomial spline, we can conclude that
\begin{align*}
\inf_{p\in\mathcal{P}_q(\mathcal{T})}\|f_0-p\|_\infty\geq\Delta_{\mathcal{T}}^{\alpha}(\beta/C)C\geq\delta_q\|f_0\|_{\mathcal{H}^{\alpha}([0,1])}\Delta_{\mathcal{T}}^{\alpha},
\end{align*}
where $\delta_q:=\beta/C$ depends only on $f_0$ and $q$.

Another important class of examples are those functions who exhibit self-similarity in scale, such that the function looks approximately the same at different scaling of its domain.\vspace{10pt}
\textbf{Example 2:} Let $f_0$ be self-similar (or invariant) in scale of order $\alpha$, i.e., $f_0(x/\lambda)=\lambda^{-\alpha}f_0(x)$ for any $\lambda>0$, and $f_0$ is not a polynomial spline of order $q$. For $p\in\mathcal{P}_q(\mathcal{T})$, let $\widehat{p}(x)=\lambda^{\alpha}p(x/\lambda)$. By a change of variables, we have
\begin{align*}
\sup_{0\leq x\leq 1}|f_0(x)-p(x)|=\lambda^{-\alpha}\sup_{0\leq x\leq\lambda}|\lambda^{\alpha}f_0(x/\lambda)-\lambda^{\alpha}p(x/\lambda)|=\lambda^{-\alpha}\sup_{0\leq x\leq\lambda}|f_0(x)-\widehat{p}(x)|.
\end{align*}
Note that $\widehat{p}$ is also a polynomial splines of order $q$. Since the above equality is valid for any $p\in\mathcal{P}_q(\mathcal{T})$ and $f_0$ is not a polynomial splines of order $q$,
\begin{align*}
0<\delta_q:=\inf_{p\in\mathcal{P}_q(\mathcal{T})}\|f_0-p\|_{L_{\infty}[0,1]}\leq\lambda^{-\alpha}\inf_{p\in\mathcal{P}_q(\mathcal{T})}\|f_0-p\|_{L_{\infty}[0,\lambda]},
\end{align*}
where $\delta_q$ is some constant depending only on $q$ and $f_0$. The lower bound can then be recovered through setting $\lambda=\|f_0\|_{\mathcal{H}^{\alpha}}^{1/\alpha}\Delta_{\mathcal{T}}$. This notion of self-similarity can be substantially generalized to include being invariant with respect to certain transformations of the domain, for example see (2.3) of \citep{similar}.

\section{Application II: Wavelet bases}\label{sec:wavelet}
Since our domain is $[0,1]$, we will use the Cohen-Daubechies-Vial (CDV) father and mother wavelets $\varphi$ and $\psi$. These wavelets are compactly supported and boundary corrected. Assume that the wavelets are $r$-regular, in the sense that $|(d^p/dx^p)\varphi(x)|\leq C_m(1+|x|)^{-m}$ for some constants $C_m,m\in\mathbb{N}$ and $0\leq p\leq r$. We construct an orthonormal basis for $L_2([0,1])$ by dilating and translating these wavelets as follows: $\varphi_{j,k}(x)=2^{j/2}\varphi(2^jx-k)$ and $\psi_{j,k}(x)=2^{j/2}\psi(2^jx-k)$. We then project $f$ onto the $2^J$-dimensional CDV wavelet bases as
\begin{align*}
f(x)=\sum_{k=0}^{2^N-1}\vartheta_k\varphi_{N,k}(x)+\sum_{j=N}^{J-1}\sum_{k=0}^{2^j-1}\theta_{j,k}\psi_{j,k}(x)=:\boldsymbol{\psi}_J(x)^T\boldsymbol{\theta},
\end{align*}
where we have concatenated the father and mother wavelets into a single vector of functions $\boldsymbol{\psi}_J(x)$, and their coefficients collected into $\boldsymbol{\theta}$. The total number of wavelets is $2^J$ which we will identify with the ``$J$" used in previous sections. Here we choose $N$ such that $2^N\geq2r$. As before, we put prior $\boldsymbol{\theta}|\sigma\sim\mathrm{N}(\boldsymbol{\eta}, \sigma^2\boldsymbol{\Omega})$ and estimate $\sigma^2$ by empirical Bayes \eqref{eq:sigma}. Write the wavelet basis matrix as $\boldsymbol{\Psi}=(\boldsymbol{\psi}_J(X_1)^T,\dotsc,\boldsymbol{\psi}_J(X_n)^T)^T$. We choose the coviarates such that $\|F_n-F\|_\infty=O(1/n)$. Since $2^{j_{\mathrm{max}}}=o(n)$ by assumption, Lemma 6.4 of \citep{yoo2017} tells us that there exist constants $C_1,C_2>0$ such that
\begin{align}
C_1n\leq\lambda_{\mathrm{min}}(\boldsymbol{\Psi}^T\boldsymbol{\Psi})\leq\lambda_{\mathrm{max}}(\boldsymbol{\Psi}^T\boldsymbol{\Psi})\leq C_2n,
\end{align}
and we have $\lambda_{\boldsymbol{\Psi},J}=n$ in this case. Since the wavelets are compact supported and uniformly bounded, we have $l_{1,J}\lesssim 2^{J/2}$ (Lemma \ref{lem:w2}). If $2^J\leq n$, then $\lambda_{\boldsymbol{\Psi},J}\gtrsim l_{1,J}^2$. Suppose the support of $\psi$ is $[a,b]$, then for any $x$ and $j$, $\psi_{j,k}(x)\neq0$ for $2^jx-b\leq k\leq 2^jx-a$ and only a fixed number (not depending on $j$) of consecutive $k$'s for $\psi_{j,k}$ are nonzero. Thus as in the B-splines case before, $\boldsymbol{\Psi}^T\boldsymbol{\Psi}$ is banded.

Define $K_J(x,y)=\sum_k\varphi_{J,k}(x)\varphi_{J,k}(y)$ and let $K_J(f)(x)=\int K_J(x,y)f(y)dy$ for any $f\in L_2([0,1])$ be the wavelet projection operator. By Theorem 4 of Section 2.6 \citep{yves}, we know that $K_J(p)=p$ for any polynomial $p$ of degree $r$ or less. Thus, $2^{j_{\mathrm{min}}}=(n/\log{n})^{1/(2r+1)}$ and we adapt $\alpha$ up to $r$. Then in the present case
\begin{align}
\widehat{j}_n=\min\left\{j\in\mathcal{J}:\|\mathrm{E}_j(f|\boldsymbol{Y})-\mathrm{E}_i(f|\boldsymbol{Y})\|_\infty\leq\tau\widehat{\sigma}_i\sqrt{\frac{i2^i}{n}},\forall i>j,i\in\mathcal{J}\right\}.
\end{align}

Let $\mathcal{B}^{\alpha}_{\infty,\infty}$ be the Besov space consisting of functions $g:[0,1]\rightarrow\infty$ such that $\|g\|_{\mathcal{B}^{\alpha}_{\infty,\infty}}<\infty$, where the Besov norm is defined as
\begin{align}
\|g\|_{\mathcal{B}^{\alpha}_{\infty,\infty}}=\max_{0\leq k\leq2^N-1}|\langle g,\varphi_{N,k}\rangle|+\max_{j\geq N}2^{j(\alpha+1/2)}\max_{0\leq k\leq2^j-1}|\langle g,\psi_{j,k}\rangle|.
\end{align}
Here we take $\mathcal{G}=\{g:\|g\|_{\mathcal{B}^{\alpha}_{\infty,\infty}}\leq R\}=:\mathcal{B}^{\alpha}_{\infty,\infty}(R)$ to be the Besov ball of radius $R$. By Proposition 4.3.8 of \citep{nickl2016}, we have for any $g\in\mathcal{B}^{\alpha}_{\infty,\infty}(R)$ that there is a constant $C_r>0$ depending on $r$ such that
\begin{align}\label{eq:approxw}
\|K_J(g)-g\|_\infty\leq C_r\|g\|_{\mathcal{B}^{\alpha}_{\infty,\infty}}2^{-J\alpha},
\end{align}
and $h(x)=2^{-x\alpha}$ in this case. Hence we see that $h(j_{\mathrm{min}})=(\log{n}/n)^{\alpha/(2r+1)}\to0$ as $n\to\infty$. This wavelet projection can be represented as $K_J(g)=\langle\boldsymbol{\psi}_J,\boldsymbol{\theta}_0\rangle$ for some $\boldsymbol{\theta}_0\in\mathbb{R}^{2^J}$ and by virtue of Proposition 4.3.5 of \citep{nickl2016}, we have $\|\boldsymbol{\theta}_0\|_\infty<\infty$. Suppose $f_0\in\mathcal{B}^{\alpha}_{\infty,\infty}(R)$ with $\alpha>0$ unknown. Then $j_n^{*}$ of \eqref{eq:j0} balances $\|f_0\|_{\mathcal{B}^{\alpha}_{\infty,\infty}}2^{-j\alpha}$ and $\sqrt{2^jj/n}$ to yield $2^{j_n^{*}}\asymp\|f_0\|_{\mathcal{B}^{\alpha}_{\infty,\infty}}^{2/(2\alpha+1)}(n/\log{n})^{1/(2\alpha+1)}$. As a result, the wavelet version of Theorem \ref{th:rate} is
\begin{align}
\sup_{0<\alpha\leq r}\sup_{f_0\in\mathcal{B}^{\alpha}_{\infty,\infty}(R)}\mathrm{E}_0\Pi[\|f-f_0\|_\infty>\xi(\log{n}/n)^{\alpha/(2\alpha+1)}|\boldsymbol{Y}]\lesssim(\log{n}/n)^{M}
\end{align}
as $n\rightarrow\infty$, for some constants $R,\xi,M>0$.

The wavelet credible band constructed using the volume-of-tube method is
\begin{align*}
\mathcal{W}_n:=\left\{f:|f(x)-\mathrm{E}_{\widehat{j}_n}[f(x)|\boldsymbol{Y}]|\leq w_{\gamma,\widehat{j}_n}\sqrt{\mathrm{Var}_{\widehat{j}_n}[f(x)|\boldsymbol{Y}]},\forall x\in[0,1]\right\},
\end{align*}
where the posterior mean and variance is
\begin{align}
\mathrm{E}_{\widehat{j}_n}[f(x)|\boldsymbol{Y}]&=\boldsymbol{\psi}_{\widehat{j}_n}(x)^T(\boldsymbol{\Psi}^T\boldsymbol{\Psi}+\boldsymbol{\Omega}^{-1})^{-1}(\boldsymbol{\Psi}^T\boldsymbol{Y}+\boldsymbol{\Omega}^{-1}\boldsymbol{\eta}),\\ \mathrm{Var}_{\widehat{j}_n}[f(x)|\boldsymbol{Y}]&=\widehat{\sigma}_{\widehat{j}_n}^2\boldsymbol{\psi}_{\widehat{j}_n}(x)^T(\boldsymbol{\Psi}^T\boldsymbol{\Psi}+\boldsymbol{\Omega}^{-1})^{-1}\boldsymbol{\psi}_{\widehat{j}_n}(x),
\end{align}
and we choose $w_{\gamma,\widehat{j}_n}$ such that it solves $\gamma=|\beta_{\widehat{j}_n}|\pi^{-1}e^{-w_{\gamma,\widehat{j}_n}^2/2}+2[1-\Phi(w_{\gamma,\widehat{j}_n})]$, where the arc length is
\begin{align}\label{eq:betaW}
|\beta_{\widehat{j}_n}|=\int_0^1\frac{[\boldsymbol{\psi}_{\widehat{j}_n}(x)^T\boldsymbol{M}\boldsymbol{R}^T\boldsymbol{M}\boldsymbol{RM}\boldsymbol{\psi}_{\widehat{j}_n}(x)]^{1/2}}
{[\boldsymbol{\psi}_{\widehat{j}_n}(x)^T\boldsymbol{M}\boldsymbol{\psi}_{\widehat{j}_n}(x)]^{3/2}}dx,
\end{align}
for $\boldsymbol{M}:=(\boldsymbol{\Psi}^T\boldsymbol{\Psi}+\boldsymbol{\Omega}^{-1})^{-1}$ and $\boldsymbol{R}:=\boldsymbol{\dot{\psi}}_{\widehat{j}_n}(x)\boldsymbol{\psi}_{\widehat{j}_n}(x)^T-\boldsymbol{\psi}_{\widehat{j}_n}(x)\boldsymbol{\dot{\psi}}_{\widehat{j}_n}(x)^T$. As before, we proceed to verify the 3 conditions. Let us remind again the reader that $J$ used in the previous sections now translates to $2^J$ and the event $\{j_n^{*}/\kappa\leq\widehat{j}_n\leq j_n^{*}\}$ appearing in conditions 1 and 3 now is $\{j_n^{*}-\kappa\leq\widehat{j}_n\leq j_n^{*}\}$.

We start with the third condition. By applying Lemma \ref{lem:w2}, we have $\inf_{x\in[0,1]}\|\boldsymbol{\psi}_{\widehat{j}_n}(x)\|\gtrsim2^{\widehat{j}_n/2}$. Then under event $\{j_n^{*}-\kappa\leq\widehat{j}_n\leq j_n^{*}\}$, we have $(2^{\widehat{j}_n}/n)^{1/2}\asymp2^{-\alpha\widehat{j}_n}$ and the third condition is
\begin{align*}
(\widehat{j}_n)^{-1/2}2^{-\widehat{j}_n/2}\left[\sqrt{\frac{2^{\widehat{j}_n}}{n}}+\sqrt{n}2^{-\alpha\widehat{j}_n}\right]\lesssim(\log{n})^{-1/2}\left(\frac{1}{\sqrt{n}}+1\right)=o_{P_0}(1),
\end{align*}
as $n\rightarrow\infty$. The second condition is fulfilled due to the following two lemmas.
\begin{lemma}\label{lem:cgammaw}
Let $|\beta_J|=\int_0^1\|\beta_J^{'}(x)\|dx$ be the arc length of the curve $\beta_J:[0,1]\rightarrow S^{J-1}$ given in \eqref{eq:betaW}. Then as $n\rightarrow\infty$,
\begin{align*}
|\beta_J|\asymp2^J\quad\text{for any $J\in\mathcal{J}$}.
\end{align*}
\end{lemma}

\begin{lemma}\label{lem:curvew}
For any $J\in\mathcal{J}$, there exist constants $\xi_1,\xi_2,\xi_3>0$ such that
\begin{align*}
|\beta_{0,J}|\leq\frac{\left[1-\xi_1n^{-1}+\xi_2n^{-2}\right]^{1/2}}{\left[1-\xi_3n^{-1}\right]^{3/2}}|\beta_J|.
\end{align*}
Consequently, $|\beta_{0,J}|\leq|\beta_J|$ as $n\to\infty$.
\end{lemma}

Let us now describe the set $\mathcal{F}$ used in the first condition. Define $\delta_r$ to be a positive constant that depends on $r$ the regularity of the CDV wavelets. Let $\{V_j:j\in\mathbb{Z}\}$ generate a multiresolution analysis of $L_2([0,1])$ associated with these wavelets, such that $V_{j-1}\subset V_j$, $\bigcap_{j\in\mathbb{Z}}V_j=\{0\}$ and $\bigcup_{j\in\mathbb{Z}}V_j$ is dense in $L_2([0,1])$. For a fixed $J_0\in\mathbb{N}$, we take $\mathcal{F}$ to be
\begin{align}
\mathcal{F}^{\alpha}=\bigcap_{j\geq J_0}\left\{f:\inf_{p\in V_j}\|p-f\|_\infty\geq\delta_r\|f\|_{\mathcal{B}^{\alpha}_{\infty,\infty}}2^{-j\alpha}\right\}.
\end{align}
Equipped with this set, we can use the same reasoning as in the proof of Lemma \ref{lem:lower} to deduce the following lemma, which serves to fulfil the first condition of Theorem \ref{th:credible} for the CDV wavelet bases.
\begin{lemma}\label{lem:lowerw}
Let $\kappa\in\mathbb{N}$ be some large enough constant depending on $\delta_r,\alpha$ and $\tau$. Then for any $R>0$ and $j_{\mathrm{min}}\geq J_0$, there are constants $\mu>0$ such that
\begin{align*}
\sup_{0<\alpha\leq r}\sup_{f_0\in\mathcal{F}^{\alpha}}P_0(\widehat{j}_n<j_n^{*}-\kappa)\lesssim e^{-\mu j_{\mathrm{min}}},
\end{align*}
and consequently in view of Lemma 4.4, we have as $n\rightarrow\infty$,
\begin{align}
\inf_{0<\alpha\leq r}\inf_{f_0\in\mathcal{B}^{\alpha}_{\infty,\infty}(R)\bigcap\mathcal{F}^{\alpha}}P_0(j_n^{*}-\kappa\leq\widehat{j}_n\leq j_n^{*})\rightarrow1.
\end{align}
\end{lemma}
Now in view of \eqref{eq:approxw}, the self-referencing version is
\begin{align*}
\inf_{p\in\mathcal{V}_{j+1}}\|p-f_0\|_\infty\geq\delta_r\|f_0\|_{\mathcal{B}^{\alpha}_{\infty,\infty}}2^{-(j+1)\alpha}\geq\frac{\delta_r2^{-\alpha}}{C_r}\|K_j(f_0)-f_0\|_\infty\geq G\inf_{p\in\mathcal{V}_j}\|p-f_0\|_\infty,
\end{align*}
for $G:=\delta_r2^{-\alpha}/C_r$. As before, the uncertainty principle applies here in the wavelet case. That is, if we used a posterior based on wavelet random series at resolution $j$ and $f_0$ so happened is in $\mathcal{V}_j$ in the multiresolution analysis, then we are ``correct" in estimating the truth, and our posterior should contract to this truth at the optimal rate. However, $\inf_{p\in\mathcal{V}_j}\|p-f_0\|=0$ and we will not be able to quantify uncertainty honestly for this truth.

As a result, we can invoke Theorem \ref{th:credible} and conclude that for any $R>0$,
\begin{align}
\inf_{0<\alpha\leq r}\inf_{f_0\in\mathcal{B}^{\alpha}_{\infty,\infty}(R)\bigcap\mathcal{F}^{\alpha}}P_0(f_0\in\mathcal{W}_n)\geq1-\gamma+o(1).
\end{align}
By Lemmas \ref{lem:cgammaw}, \ref{lem:lowerw} and \ref{lem:w2}, we have the following CDV wavelet version of Corollary \ref{cor:radius} on the size of the radius function
\begin{align}
\inf_{0<\alpha\leq r}\inf_{f_0\in\mathcal{B}^{\alpha}_{\infty,\infty}(R)\bigcap\mathcal{F}^{\alpha}}P_0\left[r(x)\asymp(\log{n}/n)^{\alpha/(2\alpha+1)},\forall x\in[0,1]\right]\rightarrow1.
\end{align}

\section{Simulations}\label{sec:sim}
The right hand side of the Bayes Lepski stopping rule \eqref{eq:optimalj} is a measure of posterior variation in sup-norm, and it serves to check the increase in posterior mean bias on the left hand side. The expression presented however is asymptotic and the use of (maximal) eigenvalues at each iteration makes it not suitable for fast computations. Moreover, to use this rule and obtain sufficient coverage for our bands, we need to induce some form of undersmoothing such that we will on average choose $J$ that is larger than necessary. The difficulty part is to figure out a way to do this without knowing the true smoothness of $f_0$.

To achieve this, we need to find a good finite sample proxy on the right hand side, such that its asymptotic formula coincides with the one in \eqref{eq:optimalj}. There are several different choices for such proxies with the same large sample property, and we choose the one which tends to give smaller values on average. For our simulations, we found that $V_i:=\widehat{\sigma}_i\sqrt{\inf_{x\in[0,1]}\boldsymbol{a}_i(x)^T(\boldsymbol{A}^T\boldsymbol{A}+\boldsymbol{\Omega}^{-1})^{-1}\boldsymbol{a}_i(x)}\sqrt{\log{j}}$ works. Here we set $\tau=1$ for the sake of algorithmic simplicity, since a large $\tau$ is there for purely technical reason (to ensure exponent of certain powers of $n$ is negative and so they approach $0$ as $n\to\infty$). Based on this line of reasoning, we then substitute $V_i$ as the right hand side of \eqref{eq:optimalj}, and this give the following practical version in pseudocode for credible bands computation using the Bayes Lepski procedure.\vspace{10pt}

\begin{algorithm}[H]
Initialize $j_{\mathrm{max}}=n/\log^2{n}$ and $j_{\mathrm{min}}=(n/\log{n})^{1/(2\upsilon+1)}$\;
Set $j=j_{\mathrm{max}}-1$\;
\While{$j\geq j_{\mathrm{min}}$}{
 rule $\leftarrow0$\;
 \For{$i$: $j<i\leq j_{\mathrm{max}}$}{
  rule $\leftarrow$ rule $+\mathbbm{1}\left\{\|\mathrm{E}_j(f|\boldsymbol{Y})-\mathrm{E}_i(f|\boldsymbol{Y})\|_\infty>V_i\right\}$;
 }
 \eIf{$\mathrm{rule }>0$}{
 $\widehat{j}_n\leftarrow j+1$\;
 break\;
 }
 {
 $j\leftarrow j-1$\;
 }
}
\caption{The Bayes Lepski algorithm (undersmoothing version)}
\end{algorithm}\vspace{10pt}

We store values of $\mathrm{E}_i(f|\boldsymbol{Y})$ and $V_i$ at each iteration and they are used again for comparison at a lower $j$ if needed. The computation of these two quantities is made easier by the bandedness of $\boldsymbol{A}^T\boldsymbol{A}$ and the compact support property of the bases utilized.

The true function is taken to be $f_0(x)=2x-x^3+\exp\{-50(x-0.5)^2\}$ for $x\in[0,1]$, and we observed $Y_i=f_0(X_i)+\varepsilon_i$ such that $X_i=i/n,i=1,\dotsc,n$ and $\sigma_0^2=0.1$. We will use cubic B-splines ($q=4$) with uniformly distributed knots. For the Gaussian prior on the coefficients, we set $\boldsymbol{\eta}=\boldsymbol{0}$ and $\boldsymbol{\Omega}=10\boldsymbol{I}$. We estimate $\sigma_0^2$ using empirical Bayes with formula given in \eqref{eq:sigma}, and we compute $J=\widehat{j}_n$ using the Bayes Lepski algorithm given above. Once this is done, we construct the B-splines volume of tube credible band as in \eqref{eq:bsplineband}, and we calculate the quantile $w_{\gamma,\widehat{j}_n}$ by computing the arc length \eqref{eq:betaB} and solving the equation relating them. We set the credibility level $1-\gamma$ as $0.95$ and see whether the resulting credible band has at least that amount of coverage probability. To highlight the benefits of the Bayes Lepski and volume of tube construction, we compare it against two other methods: the frequentist version where least squares is used instead of the posterior mean, and the $L_\infty$-ball around the posterior mean with (non-inflated) fixed radius that was considered in \citep{yoo2016}. All computations were carried in the statistical software \texttt{R} and the code can be found in the first author's web page (type William Weimin Yoo in search engines).

Table \ref{tab:coverage} shows coverage probabilities computed for all 3 methods, where these probabilities are calculated empirically by repeating the experiment $1000$ times and recording whether $f_0$ lies entirely within the credible bands for each Monte Carlo iterate.

\begin{table}[htbp]
  \centering
  \caption{Coverage probabilities for all 3 methods. Nominal level is $0.95$.}
    \begin{tabular}{ccccccc}
    \toprule
    $n$     & \textbf{50} & \textbf{100} & \textbf{300} & \textbf{500} & \textbf{1000} & \textbf{2000} \\
    \midrule
    Bayes Lepski & 0.908 & 0.937 & 0.943 & 0.958 & 0.955 & 0.957 \\
    Frequentist  & 0.87  & 0.924 & 0.945 & 0.947 & 0.956 & 0.958 \\
    Fixed radius  & 0.919 & 0.938 & 0.953 & 0.955 & 0.941 & 0.946 \\
    \bottomrule
    \end{tabular}%
  \label{tab:coverage}%
\end{table}%

All methods have at least $0.95$ coverage for large samples and the results are comparable in this case. For very small samples, the Bayes Lepski (and also the Bayes fixed radius) has slightly higher coverage than the frequentist Lepski procedure. At first glance, it seems that the Bayes Lepski and the fixed radius method have comparable performance. But before making any hasty conclusions, let us take a look at the $1000$ Monte Carlo mean radius for each method visualized as box-plots in Figure \ref{fig:radius}.

\begin{figure}[h!]
\centering
\begin{subfigure}[b]{0.32\textwidth}
\includegraphics[width=\textwidth]{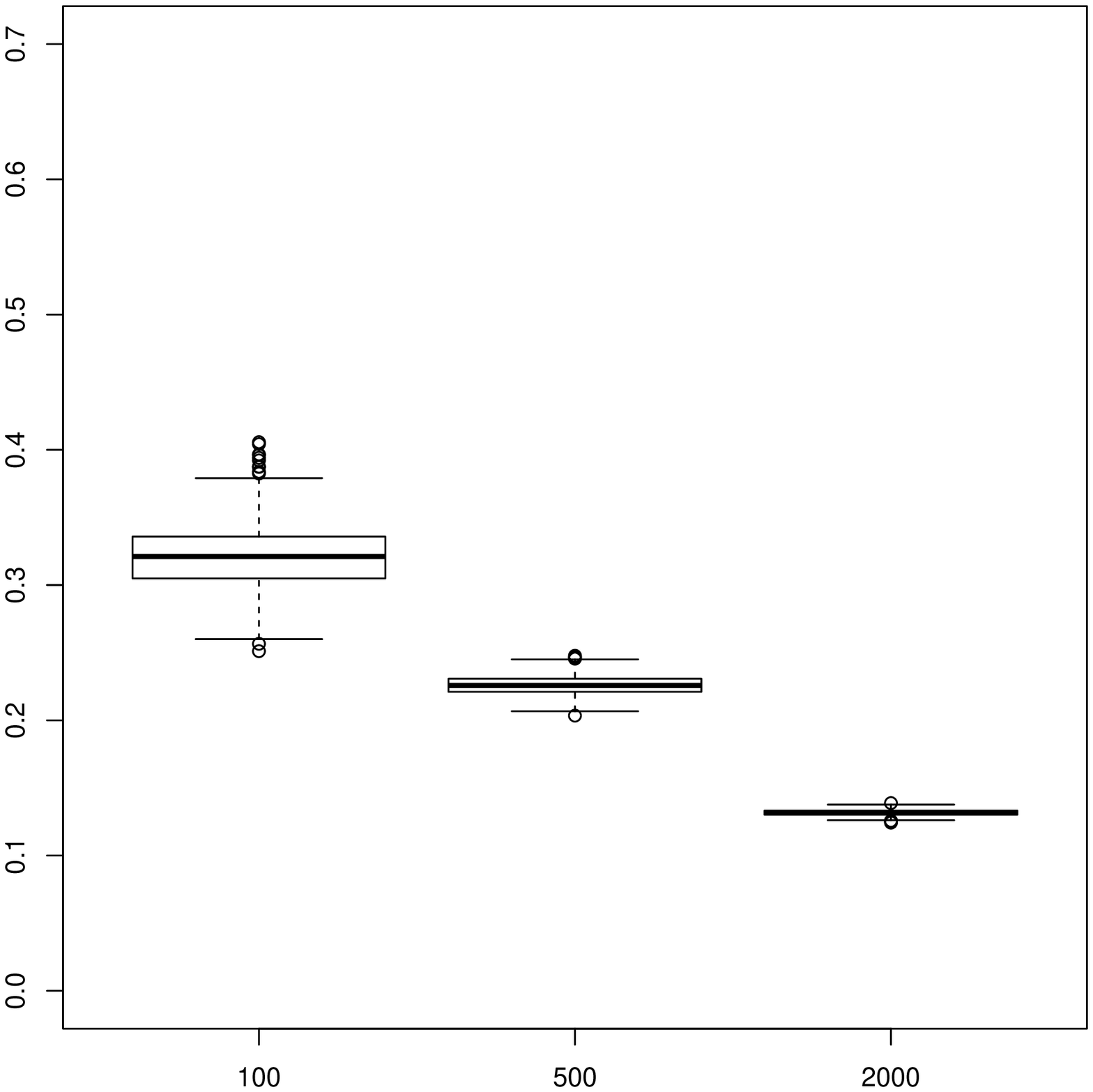}
\caption{Bayes Lepski}
\end{subfigure}
\begin{subfigure}[b]{0.32\textwidth}
\includegraphics[width=\textwidth]{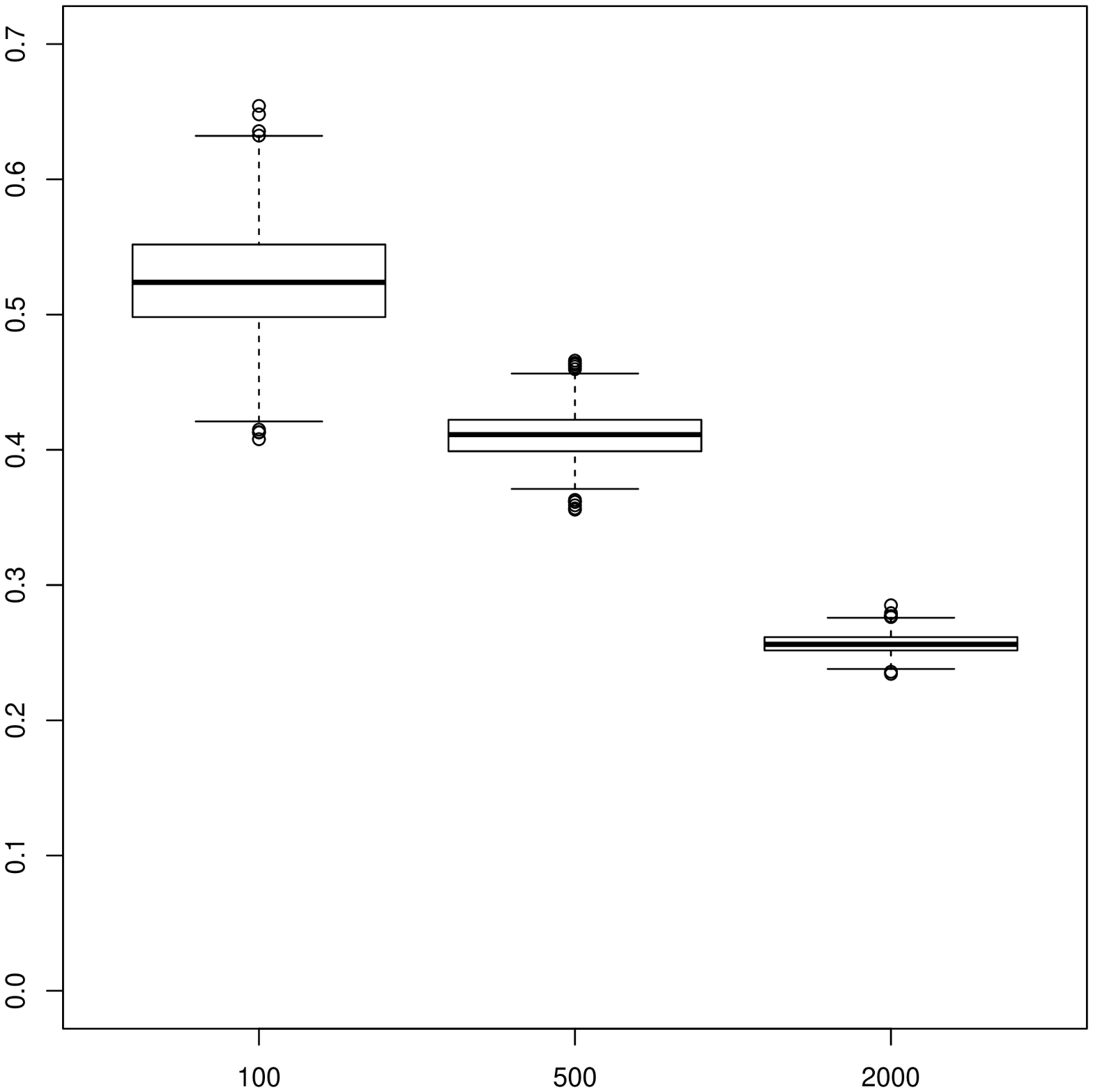}
\caption{Fixed $L_\infty$-ball}
\end{subfigure}
\begin{subfigure}[b]{0.32\textwidth}
\includegraphics[width=\textwidth]{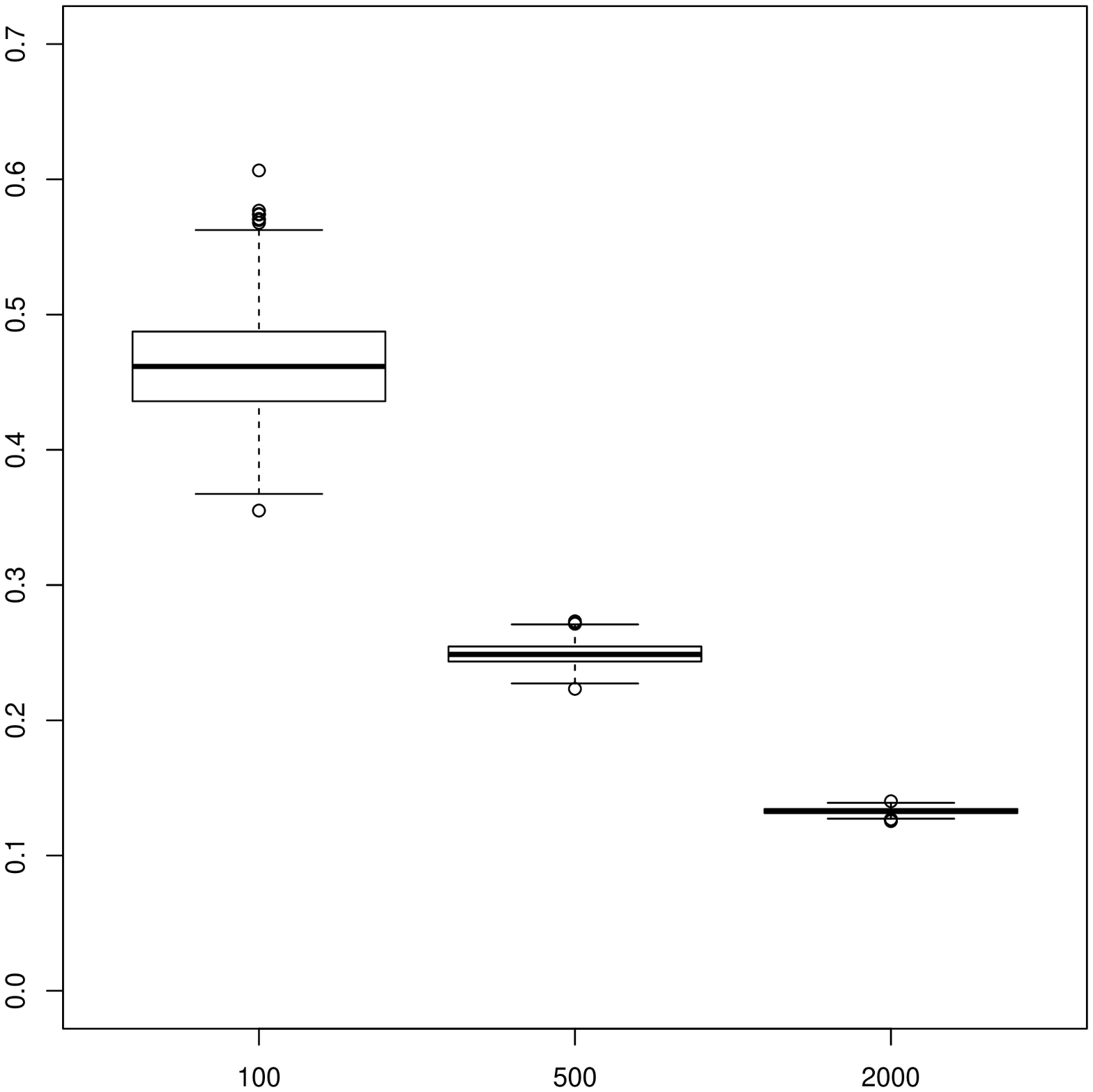}
\caption{Freq. Lepski}
\end{subfigure}
\caption{1000 Monte Carlo mean radius for $n=100,500,2000$}
\label{fig:radius}
\end{figure}

We see that the Bayes Lepski with volume of tube method has narrower bands on average when compared to the fixed $L_\infty$-ball (roughly half of its size). The frequentist volume of tube method has slightly larger bands for small sample sizes i.e., $n=100$ but is otherwise comparable to the Bayesian counterpart for larger sample sizes considered. For all methods, the radius decreases with $n$ as expected. The table and figure together show that the slight deterioration in coverage and the sudden increase in radius of the frequentist Lepski at $n\leq100$ is due to random fluctuations associated with small sample sizes. It appears that the presence of shrinkage through $\boldsymbol{\Omega}=10\boldsymbol{I}$ stabilizes this erratic fluctuations and resulting in a more predictable behavior in terms of the rate the radius decreases to $0$. Among the three methods, the Bayes fixed radius takes the longest time to compute the bands, this is due to the fact that it must further estimate the quantile empirically by drawing samples from the posterior, while the volume of tube method provides an explicit expression to compute the quantile through arc length calculations.

Simulation results in conjunction with our theoretical investigations concur that the Bayes Lepski volume of tube methods is a very promising Bayesian uncertainty quantification procedure. In particular, the Bayes Lepski rule by itself is flexible enough to be deployed in other statistical applications where tuning parameters need to be determined, for example the number of components in mixture models or spike-and-slab prior weights in high-dimensional linear regression. We envision that this rule together with the volume of tube method can be applied to other modeling situations, and this includes binary or Poisson regression and other generalized linear models. We hope to pursuit these research directions in the future.

\section{Proofs}\label{sec:proof}
Let us record here the common eigenvalue bound assumption and its corollary for ease of reference since they are be used frequently in the proofs. Recall that:
\begin{align}
\lambda_{\boldsymbol{A},J}&\lesssim\lambda_{\mathrm{min}}(\boldsymbol{A}^T\boldsymbol{A})\leq\lambda_{\mathrm{max}}(\boldsymbol{A}^T\boldsymbol{A})\lesssim\lambda_{\boldsymbol{A},J}.\label{eq:AA}
\end{align}
Observe that for symmetric and positive definite $\boldsymbol{T}$, $\|\boldsymbol{T}\|_{(2,2)}=\lambda_{\mathrm{max}}(\boldsymbol{T})$ and $\lambda_{\mathrm{max}}(\boldsymbol{T}^{-1})=\lambda_{\mathrm{min}}^{-1}(\boldsymbol{T})$ and vice versa. Now for $J\in\mathcal{J}$ in the candidate set, $\lambda_{\boldsymbol{A},J}$ will dominate any constants for large enough $n$ by assumption, and these facts together with the prior assumption in \eqref{eq:prior} imply that as $n\rightarrow\infty$,
\begin{align}\label{eq:AAO}
\lambda_{\boldsymbol{A},J}^{-1}\lesssim\lambda_{\mathrm{min}}\left\{(\boldsymbol{A}^T\boldsymbol{A}+\boldsymbol{\Omega}^{-1})^{-1}\right\}\leq
\lambda_{\mathrm{max}}\left\{(\boldsymbol{A}^T\boldsymbol{A}+\boldsymbol{\Omega}^{-1})^{-1}\right\}\lesssim\lambda_{\boldsymbol{A},J}^{-1}.
\end{align}
For notational simplicity, let us write $\widetilde{f}_J(x)=\mathrm{E}_J[f(x)|\boldsymbol{Y}]$ and $\widetilde{V}_J(x)=\mathrm{Var}_J[f(x)|\boldsymbol{Y}]$. Also, denote $\boldsymbol{\dot{a}}_J(x)=(a_1^{'}(x),\dotsc,a_J^{'}(x))^T$. For $J\in\mathcal{J}$, define $\mathcal{U}_n=\{\sigma:|\sigma-\sigma_0|\leq \xi\delta_{n,J}\}$ to be a shrinking neighborhood around $\sigma_0$ with $\delta_{n,J}=o(1)$ and the constant $\xi>0$ given in Proposition \ref{prop:sigma}.

\begin{proof}[Proof of Proposition \ref{prop:sigma}]
Let us define $U=n^{-1}(\boldsymbol{A\Omega A}^T+\boldsymbol{I}_n)^{-1}$ and $\boldsymbol{\varepsilon}=\boldsymbol{Y}-\boldsymbol{F}_0\sim\mathrm{N}(0,\sigma_0^2\boldsymbol{I}_n)$ where $\boldsymbol{F}_0=(f_0(X_1),\dotsc,f_0(X_n))^T$. For two square matrices of equal size, we say $\boldsymbol{A}<\boldsymbol{B}$ if $\boldsymbol{B}-\boldsymbol{A}$ is positive definite. Now by adding and subtracting $\boldsymbol{F}_0$ and $\mathrm{E}_0\left(\boldsymbol{\varepsilon}^T\boldsymbol{U\varepsilon}\right)$, we have $|\widehat{\sigma}_J^2-\sigma_0^2|$ is
\begin{align*}
|\boldsymbol{\varepsilon}^T\boldsymbol{U\varepsilon}-\mathrm{E}_0\left(\boldsymbol{\varepsilon}^T\boldsymbol{U\varepsilon}\right)+2(\boldsymbol{F}_0-\boldsymbol{A\eta})^T\boldsymbol{U\varepsilon}
+(\boldsymbol{F}_0-\boldsymbol{A\eta})^T\boldsymbol{U}(\boldsymbol{F}_0-\boldsymbol{A\eta})+\mathrm{E}_0\left(\boldsymbol{\varepsilon}^T\boldsymbol{U\varepsilon}\right)-\sigma_0^2|.
\end{align*}
Using the triangle inequality and the fact that $\mathrm{E}_0\left(\boldsymbol{\varepsilon}^T\boldsymbol{U\varepsilon}\right)=\mathrm{E}_0\mathrm{tr}\left(\boldsymbol{\varepsilon}^T\boldsymbol{U\varepsilon}\right)
=\mathrm{tr}\boldsymbol{U}\mathrm{E}_0\left(\boldsymbol{\varepsilon\varepsilon}^T\right)=\sigma_0^2\mathrm{tr}(\boldsymbol{U})$, the above is bounded by
\begin{align*}
|\boldsymbol{\varepsilon}^T\boldsymbol{U\varepsilon}-\mathrm{E}_0\left(\boldsymbol{\varepsilon}^T\boldsymbol{U\varepsilon}\right)|+2|(\boldsymbol{F}_0-\boldsymbol{A\eta})^T\boldsymbol{U\varepsilon}|+B_J,
\end{align*}
where $B_J:=|\sigma_0^2[\mathrm{tr}(\boldsymbol{U})-1]+(\boldsymbol{F}_0-\boldsymbol{A\eta})^T\boldsymbol{U}(\boldsymbol{F}_0-\boldsymbol{A\eta})|$. Now using the fact that for two random variables $X,Y$, $P(X+Y>a)\leq P(X>a/2)+P(Y>a/2)$ for any $a$, we can write
\begin{align}\label{eq:borell}
P_0(|\widehat{\sigma}_J^2-\sigma_0^2|>M\delta_{n,J})&\leq P_0\left(|\boldsymbol{\varepsilon}^T\boldsymbol{U\varepsilon}-\mathrm{E}_0\left(\boldsymbol{\varepsilon}^T\boldsymbol{U\varepsilon}\right)|>\frac{M\delta_{n,J}-B_J}{2}\right)\nonumber\\
&\quad+P_0\left(|(\boldsymbol{F}_0-\boldsymbol{A\eta})^T\boldsymbol{U\varepsilon}|>\frac{M\delta_{n,J}-B_J}{4}\right).
\end{align}
Let us first work on the second term and we start by bounding $B_J$. Note that since $\boldsymbol{U}<n^{-1}\boldsymbol{I}_n$, we have $\mathrm{tr}(\boldsymbol{U})<1$. Then in view of the inequality $(\boldsymbol{x}+\boldsymbol{y})^T\boldsymbol{T}(\boldsymbol{x}+\boldsymbol{y})\leq2\boldsymbol{x}^T\boldsymbol{Tx}+2\boldsymbol{y}^T\boldsymbol{Ty}$ for any positive definite $\boldsymbol{T}$, we have
\begin{align}\label{eq:BJ}
B_J\leq\sigma_0^2[1-\mathrm{tr}(\boldsymbol{U})]+2(\boldsymbol{F}_0-\boldsymbol{A\theta}_0)^T\boldsymbol{U}(\boldsymbol{F}_0-\boldsymbol{A\theta}_0)+2(\boldsymbol{\theta}_0-\boldsymbol{\eta})^T\boldsymbol{A}^T\boldsymbol{UA}
(\boldsymbol{\theta}_0-\boldsymbol{\eta}),
\end{align}
with $\boldsymbol{\theta}_0$ taken from \eqref{eq:approxG}. To proceed, we need the following result called the binomial inverse theorem (Theorem 18.2.8 of \citep{binomial}), and it says that for matrices $\boldsymbol{B},\boldsymbol{C},\boldsymbol{D},\boldsymbol{E}$ of conformable dimensions, we will have
\begin{align*}
(\boldsymbol{B}+\boldsymbol{CDE})^{-1}=\boldsymbol{B}^{-1}-\boldsymbol{B}^{-1}\boldsymbol{C}(\boldsymbol{D}^{-1}+\boldsymbol{EB}^{-1}
\boldsymbol{C})^{-1}\boldsymbol{EB}^{-1}.
\end{align*}
Hence by applying the above twice to $n\boldsymbol{U}$, we get
\begin{align}\label{eq:inverse}
n\boldsymbol{U}=\boldsymbol{I}_n-\boldsymbol{A}(\boldsymbol{A}^T\boldsymbol{A}+\boldsymbol{\Omega}^{-1})^{-1}\boldsymbol{A}^T
=\boldsymbol{I}_n-\boldsymbol{P}_{\boldsymbol{A}}+\boldsymbol{V},
\end{align}
where $\boldsymbol{P}_{\boldsymbol{A}}=\boldsymbol{A}(\boldsymbol{A}^T\boldsymbol{A})^{-1}\boldsymbol{A}^T$ is the orthogonal projection matrix and $\boldsymbol{V}=\boldsymbol{A}(\boldsymbol{A}^T\boldsymbol{A})^{-1}[\boldsymbol{\Omega}+(\boldsymbol{A}^T\boldsymbol{A})^{-1}]^{-1}
(\boldsymbol{A}^T\boldsymbol{A})^{-1}\boldsymbol{A}^T$.

For the first term in \eqref{eq:BJ}, we use \eqref{eq:inverse} to write
\begin{align*}
1-\mathrm{tr}(\boldsymbol{U})=n^{-1}\mathrm{tr}(\boldsymbol{I}_n-n\boldsymbol{U})=n^{-1}\mathrm{tr}(\boldsymbol{P}_{\boldsymbol{A}}-\boldsymbol{V})
\leq n^{-1}\mathrm{rank}(\boldsymbol{A})\leq J/n
\end{align*}
because $\boldsymbol{V}$ is positive definite. Since $\boldsymbol{U}<n^{-1}\boldsymbol{I}_n$, we have $\lambda_{\mathrm{max}}(\boldsymbol{U})<n^{-1}$, then using the inequality $\boldsymbol{x}^T\boldsymbol{Tx}\leq\lambda_{\mathrm{max}}(\boldsymbol{T})\|\boldsymbol{x}\|^2$ for any square matrix $\boldsymbol{T}$, we can bound the second term of \eqref{eq:BJ} by
\begin{align*}
\lambda_{\mathrm{max}}(\boldsymbol{U})\|\boldsymbol{F}_0-\boldsymbol{A\theta}_0\|^2\leq\|\boldsymbol{F}_0-\boldsymbol{A\theta}_0\|_\infty^2\leq C_0^2\|f_0\|_{\mathcal{G}}^2h(J)^2,
\end{align*}
in view of \eqref{eq:approxG} in the definition of $\mathcal{G}$. Next, observe that $(\boldsymbol{I}_n-\boldsymbol{P}_{\boldsymbol{A}})\boldsymbol{A}=\boldsymbol{0}$ and  $\boldsymbol{A}^T\boldsymbol{VA}=[\boldsymbol{\Omega}+(\boldsymbol{A}^T\boldsymbol{A})^{-1}]^{-1}<\boldsymbol{\Omega}^{-1}$. These facts then enable us to write the last term in \eqref{eq:BJ} as
\begin{align*}
n^{-1}(\boldsymbol{\theta}_0-\boldsymbol{\eta})^T\boldsymbol{A}^T\boldsymbol{VA}(\boldsymbol{\theta}_0-\boldsymbol{\eta})\leq
n^{-1}\lambda_{\mathrm{max}}(\boldsymbol{\Omega}^{-1})\|\boldsymbol{\theta}_0-\boldsymbol{\eta}\|^2
\leq(J/n)\lambda_{\mathrm{min}}^{-1}(\boldsymbol{\Omega})\|\boldsymbol{\theta}_0-\boldsymbol{\eta}\|^2_{\infty}\lesssim J/n,
\end{align*}
since $\|\boldsymbol{\eta}\|_\infty<\infty$ and $\lambda_{\mathrm{min}}^{-1}(\boldsymbol{\Omega})\leq c_1^{-1}$ by our prior assumption in \eqref{eq:prior}, and $\|\boldsymbol{\theta}_0\|_{\infty}<\infty$ is from the definition of $\mathcal{G}$. By considering these bounds together, we conclude that $B_J\leq Q_1[J/n+\|f_0\|_{\mathcal{G}}^2h(J)^2]$ for some constant $Q_1>0$.

Returning to the original task of controlling the second term in \eqref{eq:borell}, which we will do through the tail bound of a standard normal, we first note $(\boldsymbol{F}_0-\boldsymbol{A\eta})^T\boldsymbol{U\varepsilon}\sim\mathrm{N}[0,\sigma_0^2(\boldsymbol{F}_0-\boldsymbol{A\eta})^T
\boldsymbol{U}^2(\boldsymbol{F}_0-\boldsymbol{A\eta})]$. This variance can in turn be decomposed into two quadratic forms as in the last two terms of $B_J$ in \eqref{eq:BJ} above but with $\boldsymbol{U}^2$ in the middle, and the first is
\begin{align*}
\|\boldsymbol{U}(\boldsymbol{F}_0-\boldsymbol{A\theta}_0)\|^2\leq n\lambda_{\mathrm{max}}^2(\boldsymbol{U})\|\boldsymbol{F}_0-\boldsymbol{A\theta}_0\|_\infty^2\leq n^{-1}C_0^2\|f_0\|_{\mathcal{G}}^2h(J)^2,
\end{align*}
while the second is $\|\boldsymbol{AU}(\boldsymbol{\theta}_0-\boldsymbol{\eta})\|^2$. To deal with this term, recognize that $\boldsymbol{I}_n-\boldsymbol{P}_{\boldsymbol{A}}$ is idempotent and projects into the null space of $\boldsymbol{A}$, it then holds that $n^2\boldsymbol{A}^T\boldsymbol{U}^2\boldsymbol{A}=\boldsymbol{A}^T(\boldsymbol{I}_n-\boldsymbol{P}_{\boldsymbol{A}}+\boldsymbol{V})^2\boldsymbol{A}=\boldsymbol{A}^T\boldsymbol{V}^2\boldsymbol{A}<[\boldsymbol{\Omega}
+(\boldsymbol{A}^T\boldsymbol{A})^{-1}]^{-1}<\boldsymbol{\Omega}^{-1}$. Thus,
\begin{align*}
\|\boldsymbol{AU}(\boldsymbol{\theta}_0-\boldsymbol{\eta})\|^2\leq\lambda_{\mathrm{max}}(\boldsymbol{A}^T\boldsymbol{UA})\|\boldsymbol{\theta}_0-\boldsymbol{\eta}\|^2\lesssim J/n^2.
\end{align*}
As a conclusion, our analysis yield $\mathrm{Var}_0[(\boldsymbol{F}_0-\boldsymbol{A\eta})^T\boldsymbol{U\varepsilon}]\leq Q_2n^{-1}[J/n+\|f_0\|_{\mathcal{G}}^2h(J)^2]$ for some constant $Q_2>0$.

Let us take $\delta_{n,J}=\sqrt{J/n}+\|f_0\|_{\mathcal{G}}h(J)$ and note that $\delta_{n,J}\gg B_J$ for all $J\in\mathcal{J}$. Then using the tail bound $P(|Z|>z)\leq2e^{-z^2/(2\sigma^2)}$ for $Z\sim\mathrm{N}(0,\sigma^2)$ and $z\geq0$, we deduce that
\begin{align*}
P_0\left(|(\boldsymbol{F}_0-\boldsymbol{A\eta})^T\boldsymbol{U\varepsilon}|>\frac{M\delta_{n,J}-B_J}{4}\right)&\leq2\exp\left\{-\frac{(M\delta_{n,J}/2)^2}{8Q_2n^{-1}\left[J/n+\|f_0\|_{\mathcal{G}}^2h(J)^2\right]}\right\},
\end{align*}
and this is further bounded above by $2e^{-Q_3n}$ for some constant $Q_3>0$ when $n$ is large enough. It now remains to derive an equivalent exponential bound for the first term in \eqref{eq:borell} and this is accomplished via the Hanson-Wright inequality for quadratic forms (Theorem 1.1 of \citep{hw}), i.e., for $t\geq0$,
\begin{align}\label{eq:hw}
P_0\left(|\boldsymbol{\varepsilon}^T\boldsymbol{U\varepsilon}-\mathrm{E}_0\left(\boldsymbol{\varepsilon}^T\boldsymbol{U\varepsilon}\right)|>t\right)\leq2\exp\left[-Q_4\min\left(
\frac{t^2}{K^4\|\boldsymbol{U}\|^2_{\mathrm{HS}}},\frac{t}{K^2\|\boldsymbol{U}\|_{(2,2)}}\right)\right],
\end{align}
where $Q_4>0$ is some constant, $K$ is an upper bound for the Orlicz norm with Orlicz function $\psi_2(x)=e^{x^2}-1$ of $\varepsilon_i$, which we can take $K=\sqrt{8/3}\sigma_0$ by direct calculations, and $\|\boldsymbol{U}\|_{\mathrm{HS}}$ is the Hilbert-Schmidt or Frobenius norm of $\boldsymbol{U}$. We know that $\|\boldsymbol{U}\|_{(2,2)}<n^{-1}$, and $\|\boldsymbol{U}\|_{\mathrm{HS}}^2=\mathrm{tr}(\boldsymbol{U}^2)=n^{-2}\mathrm{tr}(\boldsymbol{I}_n-\boldsymbol{P}_{\boldsymbol{A}}+\boldsymbol{V}^2)=n^{-2}[n-J+\mathrm{tr}(\boldsymbol{V}^2)]$ since $\boldsymbol{I}_n-\boldsymbol{P}_{\boldsymbol{A}}$ is idempotent, of rank $n-J$, and projects into the null space of $\boldsymbol{A}$. Then by using the cyclic permutation of the trace operator and Lemma \ref{lem:tr},
\begin{align*}
\mathrm{tr}(\boldsymbol{V}^2)&\leq\lambda_{\mathrm{max}}\left\{(\boldsymbol{A}^T\boldsymbol{A})^{-1}\right\}\mathrm{tr}\left\{[\boldsymbol{\Omega}+(\boldsymbol{A}^T\boldsymbol{A})^{-1}]^{-1}(\boldsymbol{A}^T\boldsymbol{A})^{-1}[\boldsymbol{\Omega}+(\boldsymbol{A}^T\boldsymbol{A})^{-1}]^{-1}\right\}\\
&\leq\lambda_{\mathrm{min}}^{-1}(\boldsymbol{A}^T\boldsymbol{A})\mathrm{tr}(\boldsymbol{\Omega}^{-1})\lesssim J/\lambda_{\boldsymbol{A},J},
\end{align*}
where we deduce the second inequality by adding and subtracting $\boldsymbol{\Omega}$ to $(\boldsymbol{A}^T\boldsymbol{A})^{-1}$, and had used \eqref{eq:AA} with \eqref{eq:prior} for the last inequality. As a result, $\|\boldsymbol{U}\|_{\mathrm{HS}}^2\lesssim n^{-1}$. By collecting the intermediate results so far and substitute $t=(M\delta_{n,J}-B_J)/2$ into \eqref{eq:hw}, we can conclude that the first term in \eqref{eq:borell} is bounded by $2e^{-Q_5J}$ for some constant $Q_5>0$. Together with the normal tail bound above,
\begin{align*}
P_0(|\widehat{\sigma}_J^2-\sigma_0^2|>M\delta_{n,J})\leq2e^{-Q_5J}+2^{-Q_3n}\lesssim e^{-QJ},
\end{align*}
for some constant $Q>0$ as $n\rightarrow\infty$. The statement for $\widehat{\sigma}_J$ is then implied by
\begin{align*}
\left|\sqrt{\widehat{\sigma}_J^2}-\sqrt{\sigma_0^2}\right|=\frac{\left|\widehat{\sigma}_J^2-\sigma_0^2\right|}{\sqrt{\widehat{\sigma}_J^2}+\sqrt{\sigma_0^2}}\leq\frac{\left|\widehat{\sigma}_J^2-\sigma_0^2\right|}{\sigma_0}.
\end{align*}
The claim that $\delta_{n,J}\rightarrow0$ can be seen by $\delta_{n,J}\leq\sqrt{j_{\mathrm{max}}/n}+Rh(j_{\mathrm{min}})=o[(\log{n})^{-1/2}]+o(1)$ by definition of $j_{\mathrm{max}}$ and the assumption $h(j_{\mathrm{min}})=o(1)$.
\end{proof}

\begin{proof}[Proof of Proposition \ref{lem:upper}]
Let $j\in\mathcal{J}$ such that $j>j_n^{*}$. Then if $\widehat{j}_n=j$, this implies that for the previous element $j-1$, the corresponding sup-norm difference in posterior means must be greater than the $\tau\widehat{\sigma}_il_{1,i}\sqrt{\log{(i)}/\lambda_{\boldsymbol{A},i}}$ threshold. It follows that using a union bound, we will have
\begin{align}\label{eq:probj}
P_0(\widehat{j}_n=j)\leq\sum_{i\in\mathcal{J}:i\geq j}\left[P_0\left(\|\widetilde{f}_{j-1}-\widetilde{f}_i
\|_\infty>\tau\widehat{\sigma}_i l_{1,i}\sqrt{\frac{\log{i}}{\lambda_{\boldsymbol{A},i}}},\widehat{\sigma}_i\in\mathcal{U}_n\right)+P_0(\widehat{\sigma}_i\notin\mathcal{U}_n)\right].
\end{align}
The second term is $O(e^{-Qj})$ for some constant $Q>0$ by virtue of Proposition \ref{prop:sigma}. Now by the triangle inequality,
\begin{align*}
\|\widetilde{f}_{j-1}-\widetilde{f}_i\|_\infty&\leq\|\widetilde{f}_{j-1}-\mathrm{E}_0\widetilde{f}_{j-1}\|_\infty
+\|\widetilde{f}_i-\mathrm{E}_0\widetilde{f}_i\|_\infty+\|\mathrm{E}_0\widetilde{f}_{j-1}-\mathrm{E}_0\widetilde{f}_i\|_\infty.
\end{align*}
Observe that $l_{1,i}$ is increasing in $i$ by definition, $\lambda_{\boldsymbol{A},i}$ and $h(i)$ are decreasing in $i$ by assumption. Then by another application of the triangle inequality and using Lemma \ref{lem:pbias} with $i>j-1\geq j_n^{*}$, the last term on the right hand side is bounded above by
\begin{align*}
\|\mathrm{E}_0\widetilde{f}_{j-1}-f_0\|_\infty+\|\mathrm{E}_0\widetilde{f}_i-f_0\|_\infty&\lesssim\frac{l_{1,j-1}}{\lambda_{\boldsymbol{A},j-1}}
+\frac{l_{1,i}}{\lambda_{\boldsymbol{A},i}}+\|f_0\|_{\mathcal{G}}[h(j-1)+h(i)]\\
&\lesssim\|f_0\|_{\mathcal{G}}h(j_n^{*})+\frac{l_{1,i}}{\lambda_{\boldsymbol{A},i}}\leq\xi l_{1,i}\sqrt{\frac{\log{i}}{\lambda_{\boldsymbol{A},i}}}
\end{align*}
when $n$ is large enough. The last inequality follows from the definition of $j_n^{*}$ given in \eqref{eq:j0} and $\xi>0$ is some universal constant. Suppose $n$ is large enough so that $\widehat{\sigma}_i>\sigma_0/2$ over $\widehat{\sigma}_i\in\mathcal{U}_n$. If we choose $\tau$ large enough so that $\tau>(2/\sigma_0)\xi$, then the first term on the right hand side of \eqref{eq:probj} is bounded above by
\begin{align}\label{eq:2prob}
&\sum_{i\in\mathcal{J}:i\geq j}P_0\left(\|\widetilde{f}_{j-1}-\mathrm{E}_0\widetilde{f}_{j-1}\|_\infty>\frac{\tau(\sigma_0/2)-\xi}{2}l_{1,i}\sqrt{\frac{\log{i}}{\lambda_{\boldsymbol{A},i}}}\right)\nonumber\\
&\qquad+\sum_{i\in\mathcal{J}:i\geq j}P_0\left(\|\widetilde{f}_i-\mathrm{E}_0\widetilde{f}_i\|_\infty>\frac{\tau(\sigma_0/2)-\xi}{2}l_{1,i}\sqrt{\frac{\log{i}}{\lambda_{\boldsymbol{A},i}}}\right).
\end{align}
Apply Lemma \ref{lem:concentrate} by taking $x=\mu\log{(i)}$ for both the cases $J=j-1$ and $J=i$ above, such that $i>j-1$ and for any $2<\mu\leq[\tau(\sigma_0/2)-\xi-2G_1]^2/(8G_2)$ when $\tau$ is large enough and $G_1,G_2>0$ some constants, it follows that \eqref{eq:2prob} is bounded above by $2\sum_{i\in\mathcal{J}:i\geq j}e^{-\mu\log{i}}$. Suppose $n$ is large enough so that $j_n^{*}\geq2$, then using the fact that $\sum_{k=i}^{I}k^{-\beta}\leq\int_{i-1}^{I}x^{-\beta}dx$ for $\beta>0$, we deduce that
\begin{align*}
P_0(\widehat{j}_n>j_n^{*})=\sum_{j>j_n^{*}}^{j_{\mathrm{max}}}P_0(\widehat{j}_n=j)\lesssim\sum_{j>j_n^{*}}^{j_{\mathrm{max}}}
\sum_{i=j}^{j_{\mathrm{max}}}e^{-\mu\log{i}}+\sum_{j>j_n^{*}}^{j_{\mathrm{max}}}e^{-Qj}\lesssim\frac{1}{(j_n^{*})^{\mu-2}}.\qquad\qedhere
\end{align*}
\end{proof}
\noindent In what follows, let $\boldsymbol{M}=(\boldsymbol{A}^T\boldsymbol{A}+\boldsymbol{\Omega}^{-1})^{-1}$ and $\boldsymbol{M}_0=(\boldsymbol{A}^T\boldsymbol{A}+\boldsymbol{\Omega}^{-1})^{-1}\boldsymbol{A}^T\boldsymbol{A}(\boldsymbol{A}^T\boldsymbol{A}+\boldsymbol{\Omega}^{-1})^{-1}$.

\begin{proof}[Proof of Theorem \ref{th:rate}]
Let $\epsilon_n=l_{1,j_n^{*}}\sqrt{\log{(j_n^{*})}/\lambda_{\boldsymbol{A},j_n^{*}}}$ and $\xi$ some constant to be determined below, then by the law of total probability,
\begin{align}\label{eq:posterior}
\mathrm{E}_0\Pi(\|f-f_0\|_\infty>\xi\epsilon_n|\boldsymbol{Y})\leq\mathrm{E}_0\Pi(\|f-f_0\|_\infty>\xi\epsilon_n|\boldsymbol{Y})\mathbbm{1}_{\{\widehat{\sigma}_{\widehat{j}_n}\in\mathcal{U}_n,\hspace{5pt}\widehat{j}_n\leq j_n^{*}\}}\nonumber\\
+P_0\left(\widehat{\sigma}_{\widehat{j}_n}\notin\mathcal{U}_n,\widehat{j}_n\leq j_n^{*}\right)+P_0(\widehat{j}_n>j_n^{*}).
\end{align}
To bound the first term, we use the triangle inequality to write
\begin{align}\label{eq:inequality}
\|f-f_0\|_\infty&\leq\|f-\widetilde{f}_{\widehat{j}_n}\|_\infty+\|\widetilde{f}_{\widehat{j}_n}-\widetilde{f}_{j_n^{*}}\|_\infty
+\|\widetilde{f}_{j_n^{*}}-\mathrm{E}_0\widetilde{f}_{j_n^{*}}\|_\infty+\|\mathrm{E}_0\widetilde{f}_{j_n^{*}}-f_0\|_\infty.
\end{align}
By the definition of $\widehat{j}_n$ in \eqref{eq:optimalj} and intersecting with the event $\{\widehat{\sigma}_{\widehat{j}_n}\in\mathcal{U}_n,\hspace{5pt}\widehat{j}_n\leq j_n^{*}\}$, the second term is bounded above by $\tau(\sigma_0+o(1))\epsilon_n\leq C_1\tau\epsilon_n$ for some constant $C_1>0$ in $P_0$-probability. Applying Lemma \ref{lem:concentrate} for $k=j_n^{*}$ and $x=\log{j_n^{*}}$, the third term is with at least $1-(j_n^{*})^{-1}$ $P_0$-probability, bounded from above by $(G_1+\sqrt{2G_2})\epsilon_n$ for some constants $G_1,G_2>0$. By Lemma \ref{lem:pbias} and the definition of $j_n^{*}$ in \eqref{eq:j0}, it follows that the last term is of the order $l_{1,j_n^{*}}/\lambda_{\boldsymbol{A},j_n^{*}}+\|f_0\|_{\mathcal{G}}h(j_n^{*})\leq C_2\epsilon_n$ for some constant $C_2>0$ and when $n$ is large enough.

Note that $U:=f-\widetilde{f}_{\widehat{j}_n}$ is a centered Gaussian process under the conditional posterior with covariance kernel $\widehat{\sigma}_{\widehat{j}_n}^2\boldsymbol{a}_{\widehat{j}_n}(x)^T\boldsymbol{M}\boldsymbol{a}_{\widehat{j}_n}(y)$ for $x,y\in[0,1]$. In view of \eqref{eq:inequality} and the bounds established in the previous paragraph, the first term in \eqref{eq:posterior} is bounded above by $\mathrm{E}_0\Pi\left[\|U\|_\infty>(\xi-C_1\tau-G_1-\sqrt{2G_2}-C_2)\epsilon_n|\boldsymbol{Y}\right]\mathbbm{1}_{\{\widehat{\sigma}_{\widehat{j}_n}\in\mathcal{U}_n,\hspace{5pt}\widehat{j}_n\leq j_n^{*}\}}$. Now by Lemma \ref{lem:varinf} and intersecting the event $\{\widehat{\sigma}_{\widehat{j}_n}\in\mathcal{U}_n,\hspace{5pt}\widehat{j}_n\leq j_n^{*}\}$, we have that $\mathrm{E}(\|U\|_\infty|\boldsymbol{Y})\lesssim(\sigma_0+o(1))l_{1,\widehat{j}_n}[\log{(\widehat{j}_n)}/\lambda_{\boldsymbol{A},\widehat{j}_n}]^{1/2}\leq C_3\epsilon_n$ in $P_0$-probability for some constant $C_3>0$, since $l_{1,J}$ is increasing in $J$ by definition and $\lambda_{\boldsymbol{A},J}$ is decreasing in $J$ by assumption.

Define $\nu^2=\sup_{x\in[0,1]}\mathrm{Var}[U(x)|\boldsymbol{Y}]$. Take $\xi>C_1\tau+G_1+\sqrt{2G_2}+C_2+C_3=:\xi_0$. Then by the Borell's inequality (see Proposition A.2.1 of \citep{empirical}), the first term in \eqref{eq:posterior} is further bounded above by
\begin{align*}
&\mathrm{E}_0\Pi\left[\|U\|_\infty-\mathrm{E}(\|U\|_\infty|\boldsymbol{Y})>(\xi-\xi_0)\epsilon_n\middle|\boldsymbol{Y}\right]\mathbbm{1}_{\{\widehat{\sigma}_{\widehat{j}_n}\in\mathcal{U}_n,\hspace{5pt}\widehat{j}_n\leq j_n^{*}\}}\\
&\qquad\leq2\int_{\{\widehat{\sigma}_{\widehat{j}_n}\in\mathcal{U}_n,\hspace{5pt}\widehat{j}_n\leq j_n^{*}\}}\exp\left\{-(\xi-\xi_0)^2\epsilon_n^2/(2\nu^2)\right\}dP_0.
\end{align*}
Using the inequality $\boldsymbol{y}^T\boldsymbol{Ty}\leq\lambda_{\mathrm{max}}(\boldsymbol{T})\|\boldsymbol{y}\|^2$ for any square matrix $\boldsymbol{T}$ and \eqref{eq:AAO}, we know that under the event $\{\widehat{\sigma}_{\widehat{j}_n}\in\mathcal{U}_n,\hspace{5pt}\widehat{j}_n\leq j_n^{*}\}$,
\begin{align}\label{eq:var}
\nu^2&\leq(\sigma_0^2+o(1))\lambda_{\mathrm{max}}(\boldsymbol{M})l_{2,\widehat{j}_n}^2
\lesssim l_{2,j_n^{*}}^2/\lambda_{\boldsymbol{A},j_n^{*}},
\end{align}
since $l_{2,J}$ is increasing in $J$ by definition while $\lambda_{\boldsymbol{A},J}$ is decreasing in $J$ by assumption. Therefore, $\epsilon_n^2/\nu^2\gtrsim\log{j_n^{*}}$. As a result, the first term on the right hand side of \eqref{eq:posterior} approaches $0$ at the rate some power of $1/j_n^{*}$ when $\xi>\xi_0$ as $n\rightarrow\infty$.

For the second term in \eqref{eq:posterior}, we use a union bound and Proposition \ref{prop:sigma} to write
\begin{align*}
P_0(\widehat{\sigma}_{\widehat{j}_n}\notin\mathcal{U}_n,\hspace{5pt}\widehat{j}_n\leq j_n^{*})\leq\sum_{j\leq j_n^{*}}P_0(\widehat{\sigma}_j\notin\mathcal{U}_n,\hspace{5pt}\widehat{j}_n=j)\lesssim\sum_{j=j_{\mathrm{min}}}^{j_n^{*}}e^{-Qj}\lesssim e^{-Qj_{\mathrm{min}}}
\end{align*}
for some constant $Q>0$, and this tends to $0$ as $n\rightarrow\infty$ by the definition of $j_{\mathrm{min}}$. It now remains to show that the last term in \eqref{eq:posterior} goes to $0$ at some power of $1/j_n^{*}$, and this is ensured by Proposition \ref{lem:upper}.
\end{proof}

\begin{proof}[Proof of Theorem \ref{th:credible}]
We work instead with the complement
\begin{align}\label{eq:main0}
P_0(f_0\notin\mathcal{C}_{\widehat{j}_n})=P_0\left(\sup_{0\leq x\leq 1}\frac{|f_0(x)-\widetilde{f}_{\widehat{j}_n}(x)|}{\sqrt{\widetilde{V}_{\widehat{j}_n}(x)}}>w_{\gamma,\widehat{j}_n}\right),
\end{align}
and show that it is bounded above by $\gamma+o(1)$ for any $f_0\in\mathcal{G}\bigcap\mathcal{F}$. Now by the triangle inequality, we have $|f_0(x)-\widetilde{f}_{\widehat{j}_n}(x)|\leq|f_0(x)-\mathrm{E}_0\widetilde{f}_{\widehat{j}_n}(x)|+|\mathrm{E}_0\widetilde{f}_{\widehat{j}_n}(x)-\widetilde{f}_{\widehat{j}_n}(x)|$.
Then using the facts that $\sup(f+g)\leq\sup f+\sup g$ and $\sup fg=\sup f\sup g$, it follows that the right hand side of \eqref{eq:main0} is bounded above by
\begin{align*}
P_0\left[\sup_x\sqrt{\frac{\mathrm{Var}_0\widetilde{f}_{\widehat{j}_n}(x)}{\widetilde{V}_{\widehat{j}_n}(x)}}
\sup_x\frac{|\widetilde{f}_{\widehat{j}_n}(x)-\mathrm{E}_0\widetilde{f}_{\widehat{j}_n}(x)|}{\sqrt{\mathrm{Var}_0\widetilde{f}_{\widehat{j}_n}(x)}}>
w_{\gamma,\widehat{j}_n}-\sup_x\frac{|f_0(x)-\mathrm{E}_0\widetilde{f}_{\widehat{j}_n}(x)|}{\sqrt{\widetilde{V}_{\widehat{j}_n}(x)}}\right]
\end{align*}
Now notice that $\boldsymbol{M}_0$ can be reexpressed as $\boldsymbol{M}-\boldsymbol{M\Omega}^{-1}\boldsymbol{M}$ and consequently, $\boldsymbol{y}^T\boldsymbol{M}_0\boldsymbol{y}\leq\boldsymbol{y}^T\boldsymbol{My}$ for any vector $\boldsymbol{y}$, i.e., $\boldsymbol{M}-\boldsymbol{M}_0$ is nonnegative definite. Take $\boldsymbol{y}=\boldsymbol{a}_{\widehat{j}_n}(x)$ and the definition of the variances yield
\begin{align*}
\sup_{0\leq x\leq 1}\sqrt{\frac{\mathrm{Var}_0\widetilde{f}_{\widehat{j}_n}(x)}{\widetilde{V}_{\widehat{j}_n}(x)}}\leq\frac{\sigma_0}{\widehat{\sigma}_{\widehat{j}_n}}.
\end{align*}
As a result, the right hand side of \eqref{eq:main0} is bounded from above by
\begin{align}\label{eq:prob}
P_0\left(\sup_x\frac{|\widetilde{f}_{\widehat{j}_n}(x)-\mathrm{E}_0\widetilde{f}_{\widehat{j}_n}(x)|}{\sqrt{\mathrm{Var}_0\widetilde{f}_{\widehat{j}_n}(x)}}>
w_{\gamma,\widehat{j}_n}\frac{\widehat{\sigma}_{\widehat{j}_n}}{\sigma_0}\left\{1-w_{\gamma,\widehat{j}_n}^{-1}\sup_x\frac{|f_0(x)-
\mathrm{E}_0\widetilde{f}_{\widehat{j}_n}(x)|}{\sqrt{\widetilde{V}_{\widehat{j}_n}(x)}}\right\}\right).
\end{align}
By applying the inequality $\boldsymbol{y}^T\boldsymbol{Ty}\geq\lambda_{\mathrm{min}}(\boldsymbol{T})\|\boldsymbol{y}\|^2$ for any square matrix $\boldsymbol{T}$, we obtain
\begin{align}\label{eq:varlower}
\widetilde{V}_{\widehat{j}_n}(x)&\geq\widehat{\sigma}_{\widehat{j}_n}^2\lambda_{\mathrm{min}}(\boldsymbol{M})\|\boldsymbol{a}_{\widehat{j}_n}(x)\|^2
\gtrsim\widehat{\sigma}_{\widehat{j}_n}^2\inf_{x\in[0,1]}\|\boldsymbol{a}_{\widehat{j}_n}(x)\|^2/\lambda_{\boldsymbol{A},\widehat{j}_n}
\end{align}
where we have used the lower bound in \eqref{eq:AAO}. In view of Lemma \ref{lem:pbias}, the bias $|f_0(x)-\mathrm{E}_0\widetilde{f}_{\widehat{j}_n}(x)|\lesssim l_{1,\widehat{j}_n}/\lambda_{\boldsymbol{A},\widehat{j}_n}+Rh(\widehat{j}_n)$ for any $f_0\in\mathcal{G}$ and uniformly over $x\in[0,1]$. From the definition of $w_{\gamma,\widehat{j}_n}$ given in \eqref{eq:gamma} and the fact that $0\leq\Phi(x)\leq1$ for any $x$,
\begin{align}\label{eq:betalow}
\gamma=|\beta_{\widehat{j}_n}|\pi^{-1}e^{-w_{\gamma,\widehat{j}_n}^2/2}+2[1-\Phi(w_{\gamma,\widehat{j}_n})]\geq |\beta_{\widehat{j}_n}|\pi^{-1}e^{-w_{\gamma,\widehat{j}_n}^2/2},
\end{align}
and rearranging gives $w_{\gamma,\widehat{j}_n}\geq\sqrt{2\log{(\gamma^{-1}\pi^{-1}|\beta_{\widehat{j}_n}|)}}$. Therefore by the second assumption $|\beta_{\widehat{j}_n}|\gtrsim\widehat{j}_n$, we will have $w_{\gamma,\widehat{j}_n}\gtrsim[\log{(\widehat{j}_n)}]^{1/2}$. Collecting all these results,
\begin{align*}
1-w_{\gamma,\widehat{j}_n}^{-1}\sup_x\frac{|f_0(x)-\mathrm{E}_0\widetilde{f}_{\widehat{j}_n}(x)|}{\sqrt{\widetilde{V}_{\widehat{j}_n}(x)}}>1
-C\frac{[\log{(\widehat{j}_n)}]^{-1/2}}{\widehat{\sigma}_{\widehat{j}_n}\inf_{x\in[0,1]}\|\boldsymbol{a}_{\widehat{j}_n}(x)\|}\left(\frac{l_{1,\widehat{j}_n}}{\sqrt{\lambda_{\boldsymbol{A},\widehat{j}_n}}}
+\sqrt{\lambda_{\boldsymbol{A},\widehat{j}_n}}h(\widehat{j}_n)\right),
\end{align*}
where $C>0$ is some constant. Let us define $\mathcal{E}$ to be the event in \eqref{eq:prob} but with the right hand side replaced by the lower bound above. We know that under the event $\mathcal{I}:=\{j_n^{*}/\kappa\leq\widehat{j}_n\leq j_n^{*}\}$ with the constant $\kappa\geq1$ arising from the first assumption, the right hand side above is $1+o_{P_0}(1)$ by the third assumption. Then by the law of total probability, we can further bound the right hand side of \eqref{eq:prob} from above by
\begin{align}\label{eq:total}
P_0\left(\mathcal{E}\cap\mathcal{I}\cap\{\widehat{\sigma}_{\widehat{j}_n}\in\mathcal{U}_n\}\right)+P_0(\widehat{\sigma}_{\widehat{j}_n}\notin\mathcal{U}_n,\hspace{5pt}\mathcal{I})+P_0(\mathcal{I}^c).
\end{align}
Keeping in mind the aforementioned established bounds and using \eqref{eq:tube} with $\boldsymbol{\mu}=\mathrm{E}_0\mathrm{E}_j(\boldsymbol{\theta}|\boldsymbol{Y})$ and $\boldsymbol{\Sigma}=\mathrm{Var}_0\mathrm{E}_j(\boldsymbol{\theta}|\boldsymbol{Y})$, we can bound the first term by
\begin{align*}
&\sum_{j=j_n^{*}/\kappa}^{j_n^{*}}P_0\left[\sup_{0\leq x\leq 1}\frac{|f_0(x)-\mathrm{E}_0\widetilde{f}_{\widehat{j}_n}(x)|}{\sqrt{\mathrm{Var}_0\widetilde{f}_{\widehat{j}_n}(x)}}>w_{\gamma,\widehat{j}_n}\middle|\widehat{j}_n=j\right]P_0\left(\widehat{j}_n=j\right)\\
&\quad\leq\max_{j_n^{*}/\kappa\leq j\leq j_n^{*}}\left(\frac{|\beta_{0,j}|}{\pi}e^{-w_{\gamma,j}^2}+2[1-\Phi(w_{\gamma,j})]\right)\sum_{j=j_n^{*}/\kappa}^{j_n^{*}}P_0(\widehat{j}_n=j)\leq\gamma
\end{align*}
as $n\rightarrow\infty$, since $|\beta_{0,j}|\leq|\beta_j|$ for $j\in[j_n^{*}/\kappa,j_n^{*}]\subset\mathcal{J}$ by the second assumption, the definition of $w_{\gamma,j}$ in \eqref{eq:gamma}, and the first assumption for the last sum.

To complete the proof, we bound the second term in \eqref{eq:total} by
\begin{align*}
\sum_{j=j_n^{*}/\kappa}^{j_n^{*}}P_0\left(\widehat{\sigma}_j\notin\mathcal{U}_n\right)\lesssim\sum_{j=j_n^{*}/\kappa}^{j_n^{*}}e^{-Qj}\lesssim e^{-Qj_n^{*}/\kappa},
\end{align*}
in view of Proposition \ref{prop:sigma} for some constant $Q>0$. The last term of \eqref{eq:total} is $o(1)$ follows from the first assumption.
\end{proof}

\begin{proof}[Proof of Corollary \ref{cor:radius}]
Let us define the event $\mathcal{I}:=\{j_n^{*}/\kappa\leq\widehat{j}_n\leq j_n^{*}\}$. For the upper bound, we have by the tail probability estimate of a standard normal $1-\Phi(x)\leq e^{-x^2/2}$ that
\begin{align*}
\gamma=|\beta_{\widehat{j}_n}|\pi^{-1}e^{-w_{\gamma,\widehat{j}_n}^2/2}+2[1-\Phi(w_{\gamma,\widehat{j}_n})]\leq e^{-w_{\gamma,\widehat{j}_n}^2/2}(|\beta_{\widehat{j}_n}|\pi^{-1}+2).
\end{align*}
By rearranging we see that $w_{\gamma,\widehat{j}_n}\leq\sqrt{2\log{(|\beta_{\widehat{j}_n}|\gamma^{-1}\pi^{-1}+2\gamma^{-1})}}$. Since $|\beta_{\widehat{j}_n}|\lesssim\widehat{j}_n$ by assumption, it follows that $w_{\gamma,\widehat{j}_n}\lesssim(\log{\widehat{j}_n})^{1/2}$. Using the same bound as in \eqref{eq:var}, it holds that $\widetilde{V}_{\widehat{j}_n}(x)\lesssim\widehat{\sigma}_{\widehat{j}_n}^2\|\boldsymbol{a}_{\widehat{j}_n}(x)\|^2/\lambda_{\boldsymbol{A},\widehat{j}_n}$. Therefore under the event $\mathcal{I}\bigcap\{\widehat{\sigma}_{\widehat{j}_n}\in\mathcal{U}_n\}$, we will have $r(x)\lesssim\|\boldsymbol{a}_{j_n^{*}}(x)\|\sqrt{\log{(j_n^{*})}/\lambda_{\boldsymbol{A},j_n^{*}}}$ for any $x\in[0,1]$, since $\|\boldsymbol{a}_J(x)\|$ is increasing in $J$ by definition and $\lambda_{\boldsymbol{A},J}$ is decreasing in $J$ by assumption.

Now $w_{\gamma,\widehat{j}_n}\gtrsim(\log{\widehat{j}_n})^{1/2}$ was established in the proof of Theorem \ref{th:credible} because of \eqref{eq:betalow}. Using the same argument in \eqref{eq:varlower}, we will have $\widetilde{V}_{\widehat{j}_n}(x)\gtrsim\widehat{\sigma}_{\widehat{j}_n}^2\|\boldsymbol{a}_{\widehat{j}_n}(x)\|^2/\lambda_{\boldsymbol{A},\widehat{j}_n}$. Therefore under the event $\mathcal{I}\bigcap\{\widehat{\sigma}_{\widehat{j}_n}\in\mathcal{U}_n\}$, we deduce $r(x)\gtrsim\|\boldsymbol{a}_{j_n^{*}/\kappa}(x)\|\sqrt{\log{(j_n^{*})}/\lambda_{\boldsymbol{A},j_n^{*}/\kappa}}$ for $n$ large enough. The rest of the proof can be completed by the law of total probability as in \eqref{eq:total} in the proof of Theorem \ref{th:credible}, with $\mathcal{E}$ the complement of the event to be proven, i.e., $r(x)$ does not lie in between the upper and lower bounds stated in the corollary for some $x\in[0,1]$.
\end{proof}

It is now convenient to note down the specialization of \eqref{eq:AAO} to the B-splines and wavelet cases for any $J\in\mathcal{J}$:
\begin{align}
J/n&\lesssim\lambda_{\mathrm{min}}\left\{(\boldsymbol{B}^T\boldsymbol{B}+\boldsymbol{\Omega}^{-1})^{-1}\right\}\leq\lambda_{\mathrm{max}}\left\{(\boldsymbol{B}^T\boldsymbol{B}+\boldsymbol{\Omega}^{-1})^{-1}\right\}\lesssim J/n,\label{eq:BBO}\\
n^{-1}&\lesssim\lambda_{\mathrm{min}}\left\{\left(\boldsymbol{\Psi}^T\boldsymbol{\Psi}+\boldsymbol{\Omega}^{-1}\right)^{-1}\right\}\leq\lambda_{\mathrm{max}}\left\{\left(\boldsymbol{\Psi}^T\boldsymbol{\Psi}+\boldsymbol{\Omega}^{-1}\right)^{-1}\right\}\lesssim n^{-1}.\label{eq:WWO}
\end{align}

\begin{proof}[Proof of Lemmas \ref{lem:cgamma} and \ref{lem:cgammaw}]
For the reader's convenience, let us reproduce the arc length formula for general basis functions \eqref{eq:beta}:
\begin{align*}
|\beta_J|=\int_0^1\frac{[\boldsymbol{a}_J(x)^T\boldsymbol{M}\boldsymbol{R}^T\boldsymbol{M}\boldsymbol{RM}\boldsymbol{a}_J(x)]^{1/2}}{[\boldsymbol{a}_J(x)^T\boldsymbol{M}\boldsymbol{a}_J(x)]^{3/2}}dx,
\end{align*}
where $\boldsymbol{R}:=\boldsymbol{\dot{a}}_J(x)\boldsymbol{a}_J(x)^T-\boldsymbol{a}_J(x)\boldsymbol{\dot{a}}_J(x)^T$. We start with the lower bound of $|\beta_J|$, which requires a lower bound for the numerator and an upper bound for the denominator. By \eqref{eq:AAO}, we see that $\boldsymbol{a}_J(x)^T\boldsymbol{Ma}_J(x)\leq\lambda_{\mathrm{max}}(\boldsymbol{M})l_{2,J}^2$ for any $x$. Therefore,
\begin{align*}
|\beta_J|&\geq\lambda^{-3/2}_{\mathrm{max}}(\boldsymbol{M})l_{2,J}^{-3}\int_0^1[\boldsymbol{a}_J(x)^T\boldsymbol{M}\boldsymbol{R}^T\boldsymbol{M}\boldsymbol{RM}\boldsymbol{a}_J(x)]^{1/2}dx.
\end{align*}
For B-splines, we know that $\lambda_{\mathrm{max}}(\boldsymbol{M})\lesssim J/n$ by \eqref{eq:BBO} and $l_{2,J}\leq1$ by Lemma \ref{lem:b2}. Then in view of \eqref{eq:quad1} of Lemma \ref{cor:quad} with $\boldsymbol{\Sigma}=\boldsymbol{M}$ in the integrand, we will obtain
\begin{align*}
|\beta_J|\gtrsim(n/J)^{3/2}\lambda_{\mathrm{min}}(\boldsymbol{M})^{1/2}\lambda_{\mathrm{min}}(\boldsymbol{M})J\gtrsim J,
\end{align*}
where we have appealed again to \eqref{eq:BBO} to lower bound the minimum eigenvalue. The wavelet case is similar, but now $\lambda_{\mathrm{max}}(\boldsymbol{M})\lesssim n^{-1}$ by \eqref{eq:WWO} and $l_{2,J}\lesssim2^{J/2}$ by Lemma \ref{lem:w2}. By invoking \eqref{eq:quad2} of Lemma \ref{cor:quad} with $\boldsymbol{\Sigma}=\boldsymbol{M}$, we see that
\begin{align*}
|\beta_J|\gtrsim n^{3/2}2^{-3J/2}\lambda_{\mathrm{min}}(\boldsymbol{M})^{1/2}\lambda_{\mathrm{min}}(\boldsymbol{M})2^{5J/2}\gtrsim2^J.
\end{align*}

For the upper bound of $|\beta_J|$, we need a lower bound for the denominator and \eqref{eq:AAO} gives $\boldsymbol{a}_J(x)^T\boldsymbol{Ma}_J(x)\geq\lambda_{\mathrm{min}}(\boldsymbol{M})\inf_{x\in[0,1]}\|\boldsymbol{a}_J(x)\|^2$. Thus,
\begin{align*}
|\beta_J|&\leq\lambda_{\mathrm{min}}^{-3/2}(\boldsymbol{M})\left(\inf_{x\in[0,1]}\|\boldsymbol{a}_J(x)\|\right)^{-3}\int_0^1[\boldsymbol{a}_J(x)^T\boldsymbol{M}\boldsymbol{R}^T\boldsymbol{M}\boldsymbol{RM}\boldsymbol{a}_J(x)]^{1/2}dx.
\end{align*}
For B-splines, we have $\lambda_{\mathrm{min}}(\boldsymbol{M})\gtrsim J/n$ by \eqref{eq:BBO} and $\inf_{x\in[0,1]}\|\boldsymbol{a}_J(x)\|\geq q^{-1}$ by Lemma \ref{lem:b2}. Subsequently by \eqref{eq:quad1} of Lemma \ref{cor:quad},
\begin{align*}
|\beta_J|\lesssim(n/J)^{3/2}\lambda_{\mathrm{max}}(\boldsymbol{M})^{1/2}\lambda_{\mathrm{max}}(\boldsymbol{M})J\lesssim J,
\end{align*}
again by invoking \eqref{eq:BBO}. For wavelets however, we have $\lambda_{\mathrm{min}}(\boldsymbol{M})\gtrsim n^{-1}$ by \eqref{eq:WWO} and $\inf_{x\in[0,1]}\|\boldsymbol{a}_J(x)\|\gtrsim 2^{J/2}$ by Lemma \ref{lem:w2}. Consequently by \eqref{eq:quad2} of Lemma \ref{cor:quad},
\begin{align*}
|\beta_J|\lesssim n^{3/2}2^{-3J/2}\lambda_{\mathrm{max}}(\boldsymbol{M})^{1/2}\lambda_{\mathrm{max}}(\boldsymbol{M})2^{5J/2}\lesssim2^J,
\end{align*}
by appealing once more to \eqref{eq:WWO}.
\end{proof}

\begin{proof}[Proof of Lemmas \ref{lem:curve} and \ref{lem:curvew}]
The expression $|\beta_{0,J}|$ has the same functional form as $|\beta_J|$ in \eqref{eq:beta} except that $\boldsymbol{M}$ is replaced by $\boldsymbol{M}_0$, giving (for ease of reference)
\begin{align*}
|\beta_{0,J}|=\int_0^1\frac{[\boldsymbol{a}_J(x)^T\boldsymbol{M}_0\boldsymbol{R}^T\boldsymbol{M}_0\boldsymbol{RM}_0\boldsymbol{a}_J(x)]^{1/2}}{[\boldsymbol{a}_J(x)^T\boldsymbol{M}_0\boldsymbol{a}_J(x)]^{3/2}}dx,
\end{align*}
where $\boldsymbol{R}:=\boldsymbol{\dot{a}}_J(x)\boldsymbol{a}_J(x)^T-\boldsymbol{a}_J(x)\boldsymbol{\dot{a}}_J(x)^T$. We then need to upper bound the numerator and lower bound the denominator. Before we begin, note that $\boldsymbol{M}_0=\boldsymbol{M}-\boldsymbol{M\Omega}^{-1}\boldsymbol{M}$ and since $\boldsymbol{\Omega}$ is positive definite, we have $\boldsymbol{M}_0\leq\boldsymbol{M}$.

By our assumption on the prior matrix, we have $\lambda_{\mathrm{max}}(\boldsymbol{\Omega}^{-1})\leq c_1^{-1}$. Then
\begin{align*}
\boldsymbol{a}_J(x)^T\boldsymbol{M\Omega}^{-1}\boldsymbol{Ma}_J(x)\leq\lambda_{\mathrm{max}}(\boldsymbol{\Omega}^{-1})\lambda_{\mathrm{max}}(\boldsymbol{M})^2\|\boldsymbol{a}_J(x)\|^2\lesssim\lambda_{\mathrm{max}}(\boldsymbol{M})^2l_{2,J}^2;
\end{align*}
while $\boldsymbol{a}_J(x)^T\boldsymbol{Ma}_J(x)\geq\lambda_{\mathrm{min}}(\boldsymbol{M})\inf_{x\in[0,1]}\|\boldsymbol{a}_J(x)\|^2$. Putting these two together and using the relation between $\boldsymbol{M}_0$ and $\boldsymbol{M}$, we conclude that
\begin{align*}
\boldsymbol{a}_J(x)^T\boldsymbol{M}_0\boldsymbol{a}_J(x)&=\boldsymbol{a}_J(x)^T\boldsymbol{Ma}_J(x)\left(1-\frac{\boldsymbol{a}_J(x)^T\boldsymbol{M\Omega}^{-1}\boldsymbol{Ma}_J(x)}{\boldsymbol{a}_J(x)^T\boldsymbol{Ma}_J(x)}\right)\\
&\geq\boldsymbol{a}_J(x)^T\boldsymbol{Ma}_J(x)\left[1-\xi_3\frac{\lambda_{\mathrm{max}}(\boldsymbol{M})^2l_{2,J}^2}{\lambda_{\mathrm{min}}(\boldsymbol{M})\inf_{x\in[0,1]}\|\boldsymbol{a}_J(x)\|^2}\right]
\end{align*}
for some constant $\xi_3>0$. For B-splines, use \eqref{eq:BBO} and Lemma \ref{lem:b2} to deduce that the expression inside the square brackets reduces to $1-\xi_3(J/n)$. For the wavelets, use instead \eqref{eq:WWO} and Lemma \ref{lem:w2} to see that it is now $1-\xi_3n^{-1}$.

We now deal with the numerator. Now use $\boldsymbol{M}_0\leq\boldsymbol{M}$ on the middle $\boldsymbol{M}_0$ and substituting $\boldsymbol{M}_0=\boldsymbol{M}-\boldsymbol{M\Omega}^{-1}\boldsymbol{M}$ on the other two $\boldsymbol{M}_0$'s flanking the quadratic form, we know that $\boldsymbol{a}_J(x)^T\boldsymbol{M}_0\boldsymbol{R}^T\boldsymbol{M}_0\boldsymbol{R}\boldsymbol{M}_0\boldsymbol{a}_J(x)$ will be bounded above by $\boldsymbol{a}_J(x)^T\boldsymbol{MR}^T\boldsymbol{MRMa}_J(x)$ times
\begin{align}\label{eq:2ratio}
1-\frac{2\boldsymbol{a}_J(x)^T\boldsymbol{MR}^T\boldsymbol{MRM\Omega}^{-1}\boldsymbol{Ma}_J(x)}
{\boldsymbol{a}_J(x)^T\boldsymbol{MR}^T\boldsymbol{MRMa}_J(x)}+\frac{\boldsymbol{a}_J(x)^T\boldsymbol{M\Omega}^{-1}\boldsymbol{MR}^T\boldsymbol{MRM\Omega}^{-1}\boldsymbol{Ma}_J(x)}
{\boldsymbol{a}_J(x)^T\boldsymbol{MR}^T\boldsymbol{MRMa}_J(x)}.
\end{align}
By Corollary \ref{cor:quad} with $\boldsymbol{\Sigma}=\boldsymbol{M}$, we have for $\mu(J)=J^2$ (B-splines) or $\mu(J)=2^{5J}$ (CDV wavelets) that
\begin{align}\label{eq:quadA}
\mu(J)\lambda_{\mathrm{min}}(\boldsymbol{M})^3\lesssim\boldsymbol{a}_J(x)^T\boldsymbol{MR}^T\boldsymbol{MRMa}_J(x)\lesssim\mu(J)\lambda_{\mathrm{max}}(\boldsymbol{M})^3.
\end{align}
By another application of Lemma \ref{lem:quad} with $\boldsymbol{\Sigma}=\boldsymbol{M\Omega}^{-1}\boldsymbol{M}$, we can further deduce that $\boldsymbol{a}_J(x)^T\boldsymbol{M\Omega}^{-1}\boldsymbol{MR}^T\boldsymbol{MRM\Omega}^{-1}\boldsymbol{Ma}_J(x)\lesssim\mu(J)
\lambda_{\mathrm{max}}(\boldsymbol{M})\lambda_{\mathrm{max}}(\boldsymbol{M\Omega}^{-1}\boldsymbol{M})^2$. Note that since $\boldsymbol{M\Omega}^{-1}\boldsymbol{M}$ is positive definite, we have by the sub-multiplicative property of the $\|\cdot\|_{(2,2)}$-norm that $\lambda_{\mathrm{max}}(\boldsymbol{M\Omega}^{-1}\boldsymbol{M})=\|\boldsymbol{M\Omega}^{-1}\boldsymbol{M}\|_{(2,2)}\leq\|\boldsymbol{M}\|_{(2,2)}^2
\|\boldsymbol{\Omega}^{-1}\|_{(2,2)}=\lambda_{\mathrm{max}}(\boldsymbol{M})^2\lambda_{\mathrm{max}}(\boldsymbol{\Omega}^{-1})$. In conjunction with \eqref{eq:quadA}, the second ratio in \eqref{eq:2ratio} is then bounded up to some constant multiple by $\lambda_{\mathrm{max}}(\boldsymbol{M})^5\lambda_{\mathrm{max}}(\boldsymbol{\Omega}^{-1})^2\lambda^{-3}_{\mathrm{min}}(\boldsymbol{M})$. For B-splines, this is further bounded by $\xi_2(J/n)^2$ for some constant $\xi_2>0$ in view of \eqref{eq:BBO}; while this will be $\xi_2n^{-2}$ for the wavelet case in view of \eqref{eq:WWO}.

It now remains to bound the first ratio in \eqref{eq:2ratio}. The numerator of this ratio consists of a quadratic form that is asymmetric in the sense that the $\boldsymbol{\Omega}^{-1}$ on the right flank is matched by $\boldsymbol{M}$ and not $\boldsymbol{\Omega}^{-1}$, and hence inequality such as Corollary \ref{cor:quad} is not directly applicable. To work around this, we use Lemma \ref{lem:tr} by exploiting the cyclic permutation of the trace operator to rearrange the order of these matrices, and expelling the asymmetric matrices out so that the reduced quadratic form is symmetric. Here goes our argument. The numerator (a scalar quantity) is equal to $\mathrm{tr}[\boldsymbol{a}_J(x)^T\boldsymbol{MR}^T\boldsymbol{MRM\Omega}^{-1}\boldsymbol{Ma}_J(x)]$ and by rearrangement and twice application of Lemma \ref{lem:tr} is
\begin{align*}
\mathrm{tr}\left[\boldsymbol{\Omega}^{-1}\boldsymbol{Ma}_J(x)\boldsymbol{a}_J(x)^T\boldsymbol{MR}^T\boldsymbol{MRM}\right]
&\geq\lambda_{\mathrm{min}}(\boldsymbol{\Omega}^{-1})\mathrm{tr}\left[\boldsymbol{Ma}_J(x)\boldsymbol{a}_J(x)^T\boldsymbol{MR}^T\boldsymbol{MRM}\right]\\
&\geq\lambda_{\mathrm{min}}(\boldsymbol{\Omega}^{-1})\lambda_{\mathrm{min}}(\boldsymbol{M})\boldsymbol{a}_J(x)^T\boldsymbol{MR}^T\boldsymbol{MRM}\boldsymbol{a}_J(x)\\
&\gtrsim\lambda_{\mathrm{min}}(\boldsymbol{\Omega}^{-1})\lambda_{\mathrm{min}}(\boldsymbol{M})^4\mu(J),
\end{align*}
where the last inequality follows from the lower bound of \eqref{eq:quadA} above. Then in view of the upper bound in \eqref{eq:quadA}, we see that the first ratio of \eqref{eq:2ratio} is greater up to some constant than $\lambda_{\mathrm{min}}(\boldsymbol{\Omega}^{-1})\lambda_{\mathrm{min}}(\boldsymbol{M})^4\lambda^{-3}_{\mathrm{max}}(\boldsymbol{M})$. For B-splines, this is greater than $\xi_1(J/n)$ as a result of \eqref{eq:BBO}; and for the wavelets, it will be $\xi_1n^{-1}$ by \eqref{eq:WWO}. We will obtain our result by combining the upper and lower bounds established for the two ratios, and also noting that the integral integrates to $1$ since our bounds are uniform in $x\in[0,1]$.
\end{proof}

\begin{proof}[Proof of Lemma \ref{lem:lower}]
Let $j\in\mathcal{J}=[j_{\mathrm{min}},j_{\mathrm{max}}]\cap\mathbb{N}$ such that $j<j_n^{*}/\kappa$ for some constant $\kappa\in\mathbb{N}$ to be chosen below. By the definition of $\widehat{j}_n$ in \eqref{eq:optimalj}, it follows that
\begin{align}\label{eq:probjsmall}
P_0(\widehat{j}_n=j)\leq P_0\left(\|\widetilde{f}_j-\widetilde{f}_{j_n^{*}}\|_\infty\leq\tau\widehat{\sigma}_{j_n^{*}}\sqrt{\frac{j_n^{*}\log{j_n^{*}}}{n}},\hspace{5pt}\widehat{\sigma}_{j_n^{*}}\in\mathcal{U}_n\right)+P_0\left(\widehat{\sigma}_{j_n^{*}}\notin\mathcal{U}_n\right).
\end{align}
The second term is $O(e^{-Qj_n^{*}})$ for some constant $Q>0$ in view of Proposition \ref{prop:sigma}. For the first term, we apply the reverse triangle inequality twice to obtain
\begin{align*}
\|\widetilde{f}_j-\widetilde{f}_{j_n^{*}}\|_\infty&\geq\|\mathrm{E}_0\widetilde{f}_j-f_0\|_\infty-\|\mathrm{E}_0\widetilde{f}_{j_n^{*}}-f_0\|_\infty
-\|\widetilde{f}_j-\mathrm{E}_0\widetilde{f}_j-\widetilde{f}_{j_n^{*}}+\mathrm{E}_0\widetilde{f}_{j_n^{*}}\|_\infty.
\end{align*}
Observe that $\mathrm{E}_0\widetilde{f}_j(x)=\sum_{k=1}^j\mathrm{E}_0\mathrm{E}(\theta_k|\boldsymbol{Y})B_{k,q}(x)$ is a polynomial spline of order $q$ with quasi-uniform knots in $\mathcal{T}_n:=\{0=t_0<t_1<\cdots<t_N<t_{N+1}=1\}$. Let us denote $\Delta_{\mathcal{T}_n}:=\max_{1\leq i\leq N+1}(t_i-t_{i-1})$ to be the corresponding maximum knot increment. Since $f_0\in\mathcal{F}^{\alpha}$ (in view of \eqref{eq:similar}), the first term above is bounded below by
\begin{align*}
\inf_{p\in\mathcal{P}_q(\mathcal{T}_n)}\|p-f_0\|_{\infty}&\geq\delta_q\|f_0\|_{\mathcal{H}^{\alpha}([0,1])}\Delta_{\mathcal{T}_n}^{\alpha}\gtrsim\delta_q\|f_0\|_{\mathcal{H}^{\alpha}([0,1])}j^{-\alpha}\\
&\gtrsim\delta_q(\kappa/2)^{\alpha}\|f_0\|_{\mathcal{H}^{\alpha}([0,1])}(j_n^{*}/2)^{-\alpha}\geq C_1\delta_q\kappa^{\alpha}\sqrt{j_n^{*}\log{(j_n^{*})}/n},
\end{align*}
for some constant $C_1>0$. Here, the second inequality follows from the fact $\Delta_{\mathcal{T}_n}\asymp N^{-1}\asymp j^{-1}$ since the spacings between knots are of the same order and they sum to one. The third and last inequalities follow from the assumption $j<j_n^{*}/\kappa$ and \eqref{eq:j0}, where $\|f_0\|_{\mathcal{H}^{\alpha}([0,1])}i^{-\alpha}>\sqrt{i\log{(i)}/n}$ if $i<j_n^{*}$. By Lemma \ref{lem:pbias} restricted to B-splines bases and in view of the definition of $j_n^{*}$ given in \eqref{eq:j0}, it holds that $\|\mathrm{E}_0\widetilde{f}_{j_n^{*}}-f_0\|_\infty\lesssim j_n^{*}/n+\|f_0\|_{\mathcal{H}^{\alpha}([0,1])}(j_n^{*})^{-\alpha}\leq C_2\sqrt{j_n^{*}\log{(j_n^{*})}/n}$ for some constant $C_2>0$. Suppose $n$ is large enough so that $\widehat{\sigma}_{j_n^{*}}\leq2\sigma_0$ for $\widehat{\sigma}_{j_n^{*}}\in\mathcal{U}_n$. Then by combining the bounds established, the first term on the right hand side of \eqref{eq:probjsmall} is bounded above by
\begin{align*}
&P_0\left[\|\widetilde{f}_j-\mathrm{E}_0\widetilde{f}_j-\widetilde{f}_{j_n^{*}}+\mathrm{E}_0\widetilde{f}_{j_n^{*}}\|_\infty>(C_1\delta_q\kappa^{\alpha}-C_2-2\sigma_0\tau)\sqrt{\frac{j_n^{*}\log{j_n^{*}}}{n}}\right]\\
&\quad\leq P_0\left[\|\widetilde{f}_j-\mathrm{E}_0\widetilde{f}_j\|_\infty>\frac{\sqrt{\kappa}(C_1\delta_q\kappa^{\alpha}-C_2-2\sigma_0\tau)}{2}\sqrt{\frac{j\log{j}}{n}}\right]\\
&\qquad+P_0\left[\|\widetilde{f}_{j_n^{*}}-\mathrm{E}_0\widetilde{f}_{j_n^{*}}\|_\infty>\frac{\sqrt{\kappa}(C_1\delta_q\kappa^{\alpha}-C_2-2\sigma_0\tau)}{2}\sqrt{\frac{j\log{j}}{n}}\right],
\end{align*}
in view of the triangle inequality and since $j<j_n^{*}/\kappa$. We then take the constant $\kappa\in\mathbb{N}$, which depends only on $\delta_q,\alpha$ and $\tau$, to be large enough so that $C_1\delta_q\kappa^{\alpha}-C_2-2\sigma_0\tau>0$. In addition, we take $n$ large enough so that $j_{\mathrm{min}}\geq2$. Applying Lemma \ref{lem:concentrate} by letting $x=\mu\log{j}$ for both terms $J=j$ and $J=j_n^{*}$ such that $1<\mu\leq[\sqrt{\kappa}(C_1\delta_q\kappa^{\alpha}-C_2-2\sigma_0\tau)-2G_1]^2/(8G_2)$ with $G_1,G_2>0$ some constants, we can then conclude that
\begin{align*}
P_0(\widehat{j}_n<j_n^{*}/\kappa)=\sum_{j=j_{\mathrm{min}}}^{\lfloor j_n^{*}/\kappa\rfloor}P_0(\widehat{j}_n=j)\lesssim\sum_{j=j_{\mathrm{min}}}^{j_{\mathrm{max}}}e^{-\mu\log{j}}+\sum_{j=j_{\mathrm{min}}}^{\lfloor j_n^{*}/\kappa\rfloor}e^{-Qj_n^{*}}\lesssim\frac{1}{j_{\mathrm{min}}^{\mu-1}}+j_n^{*}e^{-Qj_n^{*}},
\end{align*}
where we have used the relation $\sum_{k=i}^\infty k^{-\beta}\leq\int_{i-1}^\infty x^{-\beta}dx$ for $\beta>0$ in the last inequality. Therefore, the right hand side tends to $0$ since $j_{\mathrm{min}}=(n/\log{n})^{1/(2q+1)}\rightarrow\infty$ and $j_n^{*}e^{-Qj_n^{*}}\to0$ as $n\rightarrow\infty$.
\end{proof}

\section{Appendix}\label{sec:appendix}
\begin{lemma}\label{lem:pbias}
For any $x\in[0,1]$, $J\in\mathcal{J}$, and uniformly over $f_0\in\mathcal{G}$,
\begin{align*}
|\mathrm{E}_0\mathrm{E}_J[f(x)|\boldsymbol{Y}]-f_0(x)|\lesssim l_{1,J}/\lambda_{\boldsymbol{A},J}+\|f_0\|_{\mathcal{G}}h(J).
\end{align*}
\end{lemma}

\begin{proof}
By definition \eqref{eq:approxG}, we know that for any $f_0\in\mathcal{G}$, there exists a $\boldsymbol{\theta}_0$ such that
\begin{align*}
|\mathrm{E}_0\widetilde{f}_J(x)-f_0(x)|\leq|\mathrm{E}_0\widetilde{f}_J(x)-\boldsymbol{a}_J(x)^T\boldsymbol{\theta}_0|+C_0\|f_0\|_{\mathcal{G}}h(J).
\end{align*}
In view of \eqref{eq:pmean}, the first term is
$|\boldsymbol{a}_J(x)^T\boldsymbol{M}[\boldsymbol{A}^T\boldsymbol{F}_0+\boldsymbol{\Omega}^{-1}\boldsymbol{\eta}
-\boldsymbol{M}^{-1}\boldsymbol{\theta}_0]|$. By H\"{o}lder's inequality and the sub-multiplicative property of the $\|\cdot\|_{(1,1)}$ and $\|\cdot\|_{(\infty,\infty)}$-norms, this first term can be bounded above by
\begin{align*}
&\|\boldsymbol{a}_J(x)^T\boldsymbol{M}\boldsymbol{A}^T\|_1\|\boldsymbol{F}_0-\boldsymbol{A\theta}_0\|_\infty
+\|\boldsymbol{a}_J(x)\|_1\|\boldsymbol{M}\boldsymbol{\Omega}^{-1}(\boldsymbol{\eta}-\boldsymbol{\theta}_0)\|_\infty\\
&\quad\leq l_{1,J}\|\boldsymbol{M}\|_{(1,1)}\|\boldsymbol{A}^T\|_{(1,1)}\|f_0-\boldsymbol{a}_J(\cdot)^T\boldsymbol{\theta}_0\|_\infty+l_{1,J}\|\boldsymbol{M}\|_{(\infty,\infty)}\|\boldsymbol{\Omega}^{-1}\|_{(\infty,\infty)}(\|\boldsymbol{\eta}\|_\infty+\|\boldsymbol{\theta}_0\|_\infty).
\end{align*}
By our prior assumption, $\|\boldsymbol{\eta}\|_\infty=O(1)$. By definition of $\mathcal{G}$, $\|\boldsymbol{\theta}_0\|_\infty=O(1)$ and $\|f_0-\boldsymbol{a}_J(\cdot)^T\boldsymbol{\theta}_0\|_\infty\leq C_0\|f_0\|_{\mathcal{G}}h(J)$. Using Lemma \ref{lem:BBinfty} with $g(x)=x^{-1}$, we have $\|\boldsymbol{\Omega}^{-1}\|_{(\infty,\infty)}=O(1)$ and $\|\boldsymbol{M}\|_{(\infty,\infty)}\lesssim\lambda_{\boldsymbol{A},J}^{-1}$ by exploiting the bandedness of both $\boldsymbol{A}^T\boldsymbol{A}$ and $\boldsymbol{\Omega}^{-1}$. Now since $\boldsymbol{M}$ is symmetric, its (absolute value of) row and column sums are equal. Thus, $\|\boldsymbol{M}\|_{(1,1)}$ (max of absolute value of column sums) is equal to $\|\boldsymbol{M}\|_{(\infty,\infty)}$ (max of absolute value of row sums) and this is $O(\lambda_{\boldsymbol{A},J}^{-1})$ as established previously. Furthermore, $\|\boldsymbol{A}^T\|_{(1,1)}=\max_{1\leq i\leq n}\sum_{j=1}^J|a_j(X_i)|\leq l_{1,J}$. Collecting these results together yield
\begin{align*}
|\mathrm{E}_0\widetilde{f}_J(x)-f_0(x)|\lesssim l_{1,J}/\lambda_{\boldsymbol{A},J}+l_{1,J}^2\lambda_{\boldsymbol{A},J}^{-1}\|f_0\|_{\mathcal{G}}h(J)+\|f_0\|_{\mathcal{G}}h(J)\lesssim l_{1,J}/\lambda_{\boldsymbol{A},J}+\|f_0\|_{\mathcal{G}}h(J),
\end{align*}
since $l_{1,J}\leq\sqrt{\lambda_{\boldsymbol{A},J}}$ by the assumption in \eqref{eq:Al1l2}.
\end{proof}

\begin{lemma}\label{lem:varinf}
For any $2\leq J\leq n$,
\begin{align*}
\mathrm{E}[\|f-\mathrm{E}_J(f|\boldsymbol{Y})\|_\infty|\boldsymbol{Y}]\lesssim\widehat{\sigma}_J l_{1,J}\sqrt{\log{(J)}/\lambda_{\boldsymbol{A},J}}.
\end{align*}
\end{lemma}

\begin{proof}
For any $x\in[0,1]$, $\left[f(x)-\widetilde{f}_J(x)\middle|\boldsymbol{Y}\right]$ is equal in distribution as $\widehat{\sigma}_J\boldsymbol{a}_J(x)^T(\boldsymbol{A}^T\boldsymbol{A}+\boldsymbol{\Omega}^{-1})^{-1/2}\boldsymbol{Z}$ for $\boldsymbol{Z}\sim\mathrm{N}_J(\boldsymbol{0},\boldsymbol{I}_J)$. Then by H\"{o}lder's inequality,
\begin{align*}
\mathrm{E}[\|f-\widetilde{f}_J\|_\infty|\boldsymbol{Y}]\leq\widehat{\sigma}_J l_{1,J}\|(\boldsymbol{A}^T\boldsymbol{A}+\boldsymbol{\Omega}^{-1})^{-1/2}\|_{(\infty,\infty)}\mathrm{E}\|\boldsymbol{Z}\|_\infty.
\end{align*}
Utilizing Lemma \ref{lem:BBinfty} with $g(x)=x^{-1/2}$ and the bandedness of both $\boldsymbol{A}^T\boldsymbol{A}$ and $\boldsymbol{\Omega}^{-1}$, we deduce that $\|(\boldsymbol{A}^T\boldsymbol{A}+\boldsymbol{\Omega}^{-1})^{-1/2}\|_{(\infty,\infty)}\lesssim\lambda_{\boldsymbol{A},J}^{-1/2}$. Then by Lemma 2.3.4 of \citep{nickl2016}, it follows that $\mathrm{E}\|\boldsymbol{Z}\|_\infty\lesssim\sqrt{\log{J}}$ for $J\geq2$, and the upper bound is proved by multiplying all the bounds established.
\end{proof}

\begin{lemma}\label{lem:concentrate}
For any $2\leq J\leq n$ and $x\geq0$, there exist constants $G_1,G_2>0$ such that
\begin{align}\label{eq:concentrate}
P_0\left(\|\mathrm{E}_J(f|\boldsymbol{Y})-\mathrm{E}_0\mathrm{E}_J(f|\boldsymbol{Y})\|_\infty\geq G_1l_{1,J}\sqrt{\frac{\log{J}}{\lambda_{\boldsymbol{A},J}}}+l_{2,J}\sqrt{2G_2\frac{x}{\lambda_{\boldsymbol{A},J}}}\right)\leq e^{-x}.
\end{align}
\end{lemma}
\begin{proof}
By our global assumption on the true model, $\widetilde{f}_J-\mathrm{E}_0\widetilde{f}_J$ is under $P_0$ a mean $0$ Gaussian process, such that its variance function is
\begin{align*}
\mathrm{Var}_0\widetilde{f}_J(x)&=\sigma_0^2\boldsymbol{a}_J(x)^T\boldsymbol{M}_0\boldsymbol{a}_J(x)
\leq\sigma_0^2\lambda_{\mathrm{max}}(\boldsymbol{M})\|\boldsymbol{a}_J(x)\|^2\leq G_2l_{2,J}^2/\lambda_{\boldsymbol{A},J}
\end{align*}
for any $x\in[0,1]$ since $\boldsymbol{M}_0\leq\boldsymbol{M}$, the fact $\boldsymbol{y}^T\boldsymbol{Ty}\leq\lambda_{\mathrm{max}}(\boldsymbol{T})\|\boldsymbol{y}\|^2$ for any square matrix $\boldsymbol{T}$, and the last inequality from \eqref{eq:AAO}. By the Borell's inequality (see Proposition A.2.1 from \citep{empirical} or Theorem 2.5.8 in \citep{nickl2016}), we have
\begin{align}\label{eq:gconcen}
P_0\left(\|\widetilde{f}_J-\mathrm{E}_0\widetilde{f}_J\|_\infty\geq\mathrm{E}_0\|\widetilde{f}_J-\mathrm{E}_0\widetilde{f}_J\|_\infty+\sqrt{2G_2l_{2,J}^2\lambda_{\boldsymbol{A},J}^{-1}x}\right)\leq e^{-x}.
\end{align}
Define $\boldsymbol{\eta}=\boldsymbol{A}^T\boldsymbol{\varepsilon}$. Then $\widetilde{f}_J(x)-\mathrm{E}_0\widetilde{f}_J(x)$ can be written as $\boldsymbol{a}_J(x)^T\boldsymbol{M}\boldsymbol{\eta}$ for any $x\in[0,1]$, and it follows from H\"{o}lder's inequality that
\begin{align*}
\mathrm{E}_0\|\widetilde{f}_J-\mathrm{E}_0\widetilde{f}_J\|_\infty
&\leq\sup_{x\in[0,1]}\|\boldsymbol{a}_J(x)\|_1\|\boldsymbol{M}\|_{(\infty,\infty)}\mathrm{E}_0\|\boldsymbol{\eta}\|_\infty
\lesssim l_{1,J}\lambda_{\boldsymbol{A},J}^{-1}\mathrm{E}_0\|\boldsymbol{\eta}\|_\infty,
\end{align*}
where we have applied Lemma \ref{lem:BBinfty} with $g(x)=x^{-1}$ to bound $\|\boldsymbol{M}\|_{(\infty,\infty)}\lesssim\lambda_{\boldsymbol{A},J}^{-1}$. Note that by our global assumption on the true model, $\boldsymbol{\eta}\sim\mathrm{N}(\boldsymbol{0},\sigma_0^2\boldsymbol{A}^T\boldsymbol{A})$ under $P_0$. Now by Lemma 2.3.4 of \citep{nickl2016}, we have for $Z_m\sim\mathrm{N}(0,1), m=1,\dotsc,J$, i.i.d.~that
\begin{align*}
\mathrm{E}_0\|\boldsymbol{\eta}\|_\infty\leq\max_{1\leq m\leq J}\sqrt{\sigma_0^2(\boldsymbol{A}^T\boldsymbol{A})_{m,m}}\mathrm{E}\left(\max_{1\leq m\leq J}|Z_m|\right)\lesssim\sqrt{\lambda_{\boldsymbol{A},J}\log{(J)}},
\end{align*}
and used \eqref{eq:AAO} to bound $(\boldsymbol{A}^T\boldsymbol{A})_{m,m}$. Thus, $\mathrm{E}_0\|\widetilde{f}_J-\mathrm{E}_0\widetilde{f}_J\|_\infty\leq G_1l_{1,J}\sqrt{\log{(J)}/\lambda_{\boldsymbol{A},J}}$ for some constant $G_1>0$. Substituting this bound back into \eqref{eq:gconcen} yields the desired result.
\end{proof}

\begin{lemma}\label{lem:quad}
For any $x\in[0,1]$ and any symmetric matrix $\boldsymbol{\Sigma}$ of conformable dimensions,
\begin{align*}
\boldsymbol{a}_J(x)^T\boldsymbol{\Sigma}^T\boldsymbol{R}^T\boldsymbol{M}\boldsymbol{R}\boldsymbol{\Sigma}\boldsymbol{a}_J(x)
\leq4\lambda_{\mathrm{max}}(\boldsymbol{M})\lambda_{\mathrm{max}}(\boldsymbol{\Sigma})^2\|\boldsymbol{a}_J(x)\|^4\|\boldsymbol{\dot{a}}_J(x)\|^2.
\end{align*}
For the lower bound, assume that for any $x\in[0,1]$ there is a $\boldsymbol{U}(x)$ such that $\|\boldsymbol{U}(x)\|^2\leq\xi$ for some $\xi>0$ that may depend on $J$, and $\langle\boldsymbol{U}(x),\boldsymbol{a}_J(x)\rangle=0$, then
\begin{align*}
\boldsymbol{a}_J(x)^T\boldsymbol{\Sigma}^T\boldsymbol{R}^T\boldsymbol{M}\boldsymbol{R}\boldsymbol{\Sigma}\boldsymbol{a}_J(x)\geq\lambda_{\mathrm{min}}(\boldsymbol{M})\lambda_{\mathrm{min}}(\boldsymbol{\Sigma})^2\|\boldsymbol{a}_J(x)\|^4
\langle\boldsymbol{U}(x),\boldsymbol{\dot{a}}_J(x)\rangle^2\xi^{-1}.
\end{align*}
\end{lemma}

\begin{proof}
Since the quadratic form in question is greater than $\lambda_{\mathrm{min}}(\boldsymbol{M})\boldsymbol{a}_J(x)^T\boldsymbol{\Sigma R}^T\boldsymbol{R\Sigma a}_J(x)$ and less than $\lambda_{\mathrm{max}}(\boldsymbol{M})\boldsymbol{a}_J(x)^T\boldsymbol{\Sigma R}^T\boldsymbol{R\Sigma a}_J(x)$, we can reduce the problem to finding upper and lower bounds for $\boldsymbol{a}_J(x)^T\boldsymbol{\Sigma R}^T\boldsymbol{R\Sigma a}_J(x)$.

For the upper bound, we begin by defining functions $g_1(x):=\boldsymbol{\dot{a}}_J(x)^T\boldsymbol{\Sigma}\boldsymbol{a}_J(x)$ and $g_2(x):=\boldsymbol{a}_J(x)^T\boldsymbol{\Sigma}\boldsymbol{a}_J(x)$. Using the fact that $\|\boldsymbol{x}-\boldsymbol{y}\|^2\leq2\|\boldsymbol{x}\|^2+2\|\boldsymbol{y}\|^2$ (Cauchy-Schwartz inequality and geometric-arithmetic mean inequality $\sqrt{ab}\leq(a+b)/2$), we have
\begin{align}
\boldsymbol{a}_J(x)^T\boldsymbol{\Sigma}^T\boldsymbol{R}^T\boldsymbol{R}\boldsymbol{\Sigma}\boldsymbol{a}_J(x)&=\|g_1(x)\boldsymbol{a}_J(x)-g_2(x)\boldsymbol{\dot{a}}_J(x)\|^2\label{eq:quad}\\
&\leq2g_1(x)^2\|\boldsymbol{a}_J(x)\|^2+2g_2(x)^2\|\boldsymbol{\dot{a}}_J(x)\|^2\nonumber.
\end{align}
Now $g_2(x)\leq\lambda_{\mathrm{max}}(\boldsymbol{\Sigma})\|\boldsymbol{a}_J(x)\|^2$. Then by the Cauchy-Schwartz inequality and norm sub-multiplicative property, we have $|g_1(x)|\leq\|\boldsymbol{\dot{a}}_J(x)\||\lambda_{\mathrm{max}}(\boldsymbol{\Sigma})|\|\boldsymbol{a}_J(x)\|$. These two bounds therefore imply that
\begin{align*}
\boldsymbol{a}_J(x)^T\boldsymbol{\Sigma}^T\boldsymbol{R}^T\boldsymbol{R}\boldsymbol{\Sigma}\boldsymbol{a}_J(x)\leq4\|\boldsymbol{a}_J(x)\|^4\|\boldsymbol{\dot{a}}_J(x)\|^2\lambda_{\mathrm{max}}(\boldsymbol{\Sigma})^2.
\end{align*}
For the lower bound, we know that by assumption $\|\boldsymbol{U}(x)\|^2\leq \xi$ and $\langle\boldsymbol{U}(x),\boldsymbol{a}_J(x)\rangle=0$. Since $g_2(x)\geq\lambda_{\mathrm{min}}(\boldsymbol{\Sigma})\|\boldsymbol{a}_J(x)\|^2$, then in view of \eqref{eq:quad} and the Cauchy-Schwartz inequality,
\begin{align*}
\boldsymbol{a}_J(x)^T\boldsymbol{\Sigma}^T\boldsymbol{R}^T\boldsymbol{R}\boldsymbol{\Sigma}\boldsymbol{a}_J(x)\xi&\geq\|g_1(x)\boldsymbol{a}_J(x)-g_2(x)\boldsymbol{\dot{a}}_J(x)\|^2\|\boldsymbol{U}(x)\|^2\\
&\geq\langle\boldsymbol{U}(x),g_1(x)\boldsymbol{a}_J(x)-g_2(x)\boldsymbol{\dot{a}}_J(x)\rangle^2\\
&=\left[g_1(x)\langle\boldsymbol{U}(x),\boldsymbol{a}_J(x)\rangle-g_2(x)\langle\boldsymbol{U}(x),\boldsymbol{\dot{a}}_J(x)\rangle\right]^2\\
&\geq\lambda_{\mathrm{min}}(\boldsymbol{\Sigma})^2\|\boldsymbol{a}_J(x)\|^4\langle\boldsymbol{U}(x),\boldsymbol{\dot{a}}_J(x)\rangle^2.\qedhere
\end{align*}
\end{proof}

\begin{corollary}\label{cor:quad}
By specializing Lemma \ref{lem:quad} to B-splines, we will obtain:
\begin{align}\label{eq:quad1}
J^2\lambda_{\mathrm{min}}(\boldsymbol{M})\lambda_{\mathrm{min}}(\boldsymbol{\Sigma})^2\lesssim\boldsymbol{b}_{J,q}(x)^T\boldsymbol{\Sigma}^T\boldsymbol{R}^T\boldsymbol{M}\boldsymbol{R}\boldsymbol{\Sigma}\boldsymbol{b}_{J,q}(x)\lesssim J^2\lambda_{\mathrm{max}}(\boldsymbol{M})\lambda_{\mathrm{max}}(\boldsymbol{\Sigma})^2;
\end{align}
and for the wavelet bases:
\begin{align}\label{eq:quad2}
2^{5J}\lambda_{\mathrm{min}}(\boldsymbol{M})\lambda_{\mathrm{min}}(\boldsymbol{\Sigma})^2\lesssim\boldsymbol{\psi}_J(x)^T\boldsymbol{\Sigma}^T\boldsymbol{R}^T\boldsymbol{M}\boldsymbol{R}\boldsymbol{\Sigma}
\boldsymbol{\psi}_J(x)\lesssim2^{5J}\lambda_{\mathrm{max}}(\boldsymbol{M})\lambda_{\mathrm{max}}(\boldsymbol{\Sigma})^2.
\end{align}
\end{corollary}

\begin{proof}
By Lemma \ref{lem:b2}, we see that $\|\boldsymbol{b}_{J,q}(x)\|^4\leq1$ for any $x\in[0,1]$. By (8) of Chapter X in \citep{deBoor} and Lemma \ref{lem:b2}, we know that
\begin{align*}
\|\boldsymbol{\dot{b}}_{J,q}(x)\|^2=\|\boldsymbol{Wb}_{J,q-1}(x)\|^2\leq\lambda_{\mathrm{max}}(\boldsymbol{W}^T\boldsymbol{W})\|\boldsymbol{b}_{J,q-1}(x)\|^2
\leq\|\boldsymbol{W}^T\boldsymbol{W}\|_{(2,2)}\leq\|\boldsymbol{W}^T\boldsymbol{W}\|_{(\infty,\infty)},
\end{align*}
with $\boldsymbol{W}$ given in \eqref{eq:matrix}. Note that $\boldsymbol{W}^T\boldsymbol{W}$ is $3$-banded where its $i$th diagonal entry is $2(q-1)^2/(t_i-t_{i+1-q})^2$ and its off-diagonals are of the form $-(q-1)^2(t_i-t_{i+1-q})^{-1}(t_{i+1}-t_{i+2-q})^{-1}$. Therefore by the quasi-uniformity of knots,
\begin{align*}
\|\boldsymbol{W}^T\boldsymbol{W}\|_{(\infty,\infty)}\lesssim\frac{1}{\min_{1\leq k\leq N}(t_k-t_{k-1})^2}\lesssim\frac{1}{\Delta_{\mathcal{T}}^2}\lesssim J^2.
\end{align*}
Substituting this bound back yields the stated upper bound.

For the lower bound, we need to construct a vector that is orthogonal to $\boldsymbol{b}_{J,q}(x)$, and we use a technique developed by \citep{localspline} for this purpose. Now for any $x\in[0,1]$. we can always find an index $i_x$ such that $x\in[t_{i_x-1},t_{i_x}]$. Define $\boldsymbol{U}(x)=(u_1(x),\dotsc,u_J(x))^T$ such that
\begin{align*}
u_j(x)=\begin{cases}
1-B_{i_x,q}(x),&j=i_x,\\
-B_{i_x,q}(x),&j=i_x+1,\dotsc,i_x+q-1,\\
0,&\text{otherwise}.
\end{cases}
\end{align*}
Using the fact that $0\leq B_{j,q}(x)\leq1$ for any $x$ and $j$, we have
\begin{align}\label{eq:u2}
\|\boldsymbol{U}(x)\|^2=(1-B_{i_x,q}(x))^2+(q-1)B_{i_x,q}(x)^2\leq q,
\end{align}
and we can take $\xi=q$. By the compact support of B-splines i.e., $B_{j,q}(x)=0$ for $j<i_x$ or $j>i_x+q-1$ and the partition of unity $\sum_jB_{j,q}=1$, we obtain
\begin{align}\label{eq:u0}
\langle\boldsymbol{U}(x),\boldsymbol{b}_{J,q}(x)\rangle&=\sum_{j=i_x}^{i_x+q-1}u_j(x)B_{j,q}(x)=(1-B_{i_x,q}(x))B_{i_x,q}(x)-B_{i_x,q}(x)\sum_{j=i_x+1}^{i_x+q-1}B_{j,q}(x)\nonumber\\
&=(1-B_{i_x,q}(x))B_{i_x,q}(x)-B_{i_x,q}(x)(1-B_{i_x,q}(x))=0.
\end{align}
Furthermore, using the derivative formula for B-splines as encoded in \eqref{eq:matrix} and the compact support of $B_{j-1,q}(x)$,
\begin{align}\label{eq:u1}
\langle\boldsymbol{U}(x),\boldsymbol{Wb}_{J,q-1}(x)\rangle&=(q-1)\sum_{j=i_x}^{i_x+q-2}\frac{u_{j+1}(x)-u_j(x)}{t_j-t_{j-q+1}}B_{j,q-1}(x)\nonumber\\
&=(q-1)\frac{-B_{i_x,q}(x)-(1-B_{i_x,q}(x))}{t_{i_x}-t_{i_x-q+1}}B_{i_x.q-1}(x)\nonumber\\
&=-(q-1)\frac{B_{i_x,q-1}(x)}{t_{i_x}-t_{i_x-q+1}}.
\end{align}
Now observe that $t_{i_x}-t_{i_x+q-1}\leq(q-1)\Delta_{\mathcal{T}}\lesssim(q-1)J^{-1}$ from the quasi-uniformity of the knots, and note that since $x\in[t_{i_x-1},t_{i_x}]$ and $B_{i_x,q-1}(x)$ is supported on $x\in[t_{i_x-q+1},t_{i_x}]$, it follows that $B_{i_x,q-1}(x)>0$ and hence $\min_{0\leq x\leq 1}B_{i_x,q-1}(x)^2$ is some positive constant not depending on $J$. Squaring both sides of \eqref{eq:u1} while keeping in mind of the auxiliary facts discussed, we see that $\langle\boldsymbol{U}(x),\boldsymbol{\dot{b}}_{J,q}(x)\rangle^2\gtrsim q^{-2}\min_{0\leq x\leq 1}B_{i_x,q-1}(x)^2J^2\lambda_{\mathrm{min}}(\boldsymbol{\Sigma})^2$ for any $x\in[0,1]$. The lower bound in the statement then follows from this bound and $\inf_{x\in[0,1]}\|\boldsymbol{b}_{J,q}(x)\|^4\geq q^{-2}$ by Lemma \ref{lem:b2}.

For the wavelets, we can without loss of generality assume that $\boldsymbol{R}\neq\boldsymbol{0}$ because if it were the zero matrix, then the statements of the lemma become vacuously true. Let us start with the upper bound. By Lemma \ref{lem:w2}, it holds that $\|\boldsymbol{\psi}_J(x)\|^4\lesssim2^{2J}$ for any $x\in[0,1]$. By the property of CDV wavelets, we know that $\psi^{'}$ is uniformly bounded, and since there are compactly supported,
\begin{align*}
\|\boldsymbol{\dot{\psi}}_J(x)\|^2=\sum_{j=N}^{J-1}\sum_{k=0}^{2^j-1}\left[2^{j/2}2^j\psi^{'}(2^jx-k)\right]^2\lesssim\sum_{j=N}^{J-1}2^{3j}\lesssim2^{3J}.
\end{align*}
Hence the $2^{5J}$-factor comes from multiplying these two bounds.

For the lower bound, first note that for a fixed $j$, $\psi_{j,k}(x)$ is nonzero for only a finite number of $k$'s due again to its compact support property. For the sake of subsequent argument, let us concretely enumerate these nonzero spatial shifts as $k=i_x,i_x+1,\dotsc,i_x+m-1$ for some $m>1$ following the notation for the B-splines case discussed previously. We then construct $\boldsymbol{U}(x)=(\boldsymbol{U}^T_N(x),\dotsc,\boldsymbol{U}^T_{J-1}(x))^T$ where the $k$th entry for the $j$th vector $\boldsymbol{U}_j(x)$ is
\begin{align*}
u_{j,k}(x)=\begin{cases}
\psi_{j,i_x+1}(x),&k=i_x,\\
-\psi_{j,i_x}(x),&k=i_x+1,\\
0,&\text{otherwise}.
\end{cases}
\end{align*}
Then its $\ell_2$-norm is
\begin{align*}
\|\boldsymbol{U}(x)\|^2=\sum_{j=N}^{J-1}\left[2^{j/2}\psi(2^jx-i_x-1)\right]^2+\sum_{j=N}^{J-1}\left[2^{j/2}\psi(2^jx-i_x)\right]^2\leq C_12^J,
\end{align*}
for some constant $C_1>0$. Hence we take $\xi=C_12^J$. Orthogonality can be checked by
\begin{align*}
\langle\boldsymbol{U}(x),\boldsymbol{\psi}_J(x)\rangle=\sum_{j=N}^{J-1}\sum_{k=0}^{2^j-1}u_{j,k}(x)\psi_{j,k}(x)=\sum_{j=N}^{J-1}\left[\psi_{j,i_x+1}(x)\psi_{j,i_x}(x)-\psi_{j,i_x}(x)\psi_{j,i_x+1}(x)\right]=0.
\end{align*}
Likewise, we have
\begin{align*}
\langle\boldsymbol{U}(x),\boldsymbol{\dot{\psi}}_J(x)\rangle&=\sum_{j=N}^{J-1}\sum_{k=0}^{2^j-1}u_{j,k}(x)\psi_{j,k}^{'}(x)=\sum_{j=N}^{J-1}\left[\psi_{j,i_x+1}(x)\psi_{j,i_x}^{'}(x)-\psi_{j,i_x}(x)\psi_{j,i_x+1}^{'}(x)\right]\\
&=\sum_{j=N}^{J-1}2^{2j}\underbrace{\left[\psi(2^jx-i_x-1)\psi^{'}(2^jx-i_x)-\psi(2^jx-i_x)\psi^{'}(2^jx-i_x-i)\right]}_{\text{call this $C_2$}}.
\end{align*}
Now since $\boldsymbol{R}\neq\boldsymbol{0}$ by assumption, it follows that $C_2\neq0$. Therefore,
\begin{align*}
\langle\boldsymbol{U}(x),\boldsymbol{\dot{\psi}}_J(x)\rangle^2=C_2^2\left(\sum_{j=N}^{J-1}2^{2j}\right)^2\geq C_2^2\sum_{j=N}^{J-1}2^{4j}\geq(C_2^2/16)2^{4J}.
\end{align*}
By Lemma \ref{lem:w2}, we have $\|\boldsymbol{\psi}_J(x)\|^4\gtrsim2^{2J}$. With everything now in place, the lower bound of Lemma \ref{lem:quad} takes the form of (up to some constant multiple) $\lambda_{\mathrm{min}}(\boldsymbol{M})\lambda_{\mathrm{min}}(\boldsymbol{\Sigma})^22^{2J}2^{4J}2^{-J}$ and this gives the lower bound for the present lemma.
\end{proof}

\begin{lemma}[$L_2$-norm of B-splines]\label{lem:b2}
For any $x\in[0,1]$ and any $1\leq J\leq n$,
\begin{align*}
q^{-1}\leq\|\boldsymbol{b}_{J,q}(x)\|^2\leq1.
\end{align*}
\end{lemma}

\begin{proof}
For any $x\in[0,1]$, we can always find a positive integer $i_x$ such that $x\in[t_{i_x-1},t_{i_x}]$. Using the compact support of B-splines $B_{j,q}(x)=0$ for $j<i_x$ or $j>i_x+q-1$, we can write
\begin{align*}
\|\boldsymbol{b}_{J,q}(x)\|^2=\sum_{j=1}^JB_{j,q}(x)^2=\sum_{j=i_x}^{i_x+q-1}B_{j,q}(x)^2\geq\frac{1}{q}\left(\sum_{j=i_x}^{i_x+q-1}B_{j,q}(x)\right)^2=\frac{1}{q},
\end{align*}
where we have used Jensen's inequality in the lower bound and the partition of unity property $\sum_jB_{j,q}(x)=1$ for any $x\in[0,1]$. For the upper bound, note that since $0\leq B_{j,q}(x)\leq1$ for all $j$, it holds
\begin{align*}
\|\boldsymbol{b}_{J,q}(x)\|^2=\sum_{j=1}^JB_{j,q}(x)^2\leq\sum_{j=1}^JB_{j,q}(x)=1,
\end{align*}
again using the B-spline partition of unity.
\end{proof}

\begin{lemma}[$L_2$-norm of wavelets]\label{lem:w2}
For any $x\in[0,1]$ and any $1\leq J\leq n$,
\begin{align*}
\|\boldsymbol{\psi}_J(x)\|\asymp2^{J/2}.
\end{align*}
\end{lemma}

\begin{proof}
For any $x\in[0,1]$, choose $k_x$ such that $2^{J-1}x-k_x$ is in the support of $\psi$, so that $\psi(2^{J-1}x-k_x)$ is a constant strictly larger than $0$. Then
\begin{align*}
\|\boldsymbol{\psi}_J(x)\|^2=\sum_{j=N}^{J-1}\sum_{k=0}^{2^j-1}\psi_{j,k}(x)^2\geq\sum_{k=0}^{2^{J-1}-1}\psi_{J-1,k}(x)^2\geq(0.5)2^J\psi(2^{J-1}x-k_x)^2\gtrsim2^J.
\end{align*}
For the upper bound, note that since $\psi$ is compact supported and bounded, we have $\|\sum_{k=0}^{2^j-1}\psi_{j,k}(\cdot)\|_\infty\lesssim2^{j/2}$ and $|\psi_{j,k}(x)|\lesssim2^{j/2}$ for any $x\in[0,1]$, therefore
\begin{align*}
\|\boldsymbol{\psi}_J(x)\|^2\leq\max_{N\leq j\leq J-1}\max_{0\leq k\leq2^j-1}|\psi_{j,k}(x)|\sum_{j=N}^{J-1}\sum_{k=0}^{2^j-1}\psi_{j,k}(x)\lesssim2^J.\qquad\qedhere
\end{align*}
\end{proof}

The following result relates the trace of matrix products to min and max eigenvalues of one of its product factor.
\begin{lemma}\label{lem:tr}
For any $J\times J$ non-negative definite matrices $\boldsymbol{A},\boldsymbol{B}$,
\begin{align*}
\lambda_{\mathrm{min}}(\boldsymbol{A})\mathrm{tr}(\boldsymbol{B})\leq\mathrm{tr}(\boldsymbol{AB})\leq\lambda_{\mathrm{max}}(\boldsymbol{A})\mathrm{tr}(\boldsymbol{B}).
\end{align*}
\end{lemma}

\begin{proof}
By eigen-decomposition, $\boldsymbol{A}=\boldsymbol{P\Lambda P}^T$ for some orthonormal matrix $\boldsymbol{P}$ and $\boldsymbol{\Lambda}=\mathrm{Diag}\{\lambda_1(\boldsymbol{A}),\dotsc,\lambda_J(\boldsymbol{A})\}$ are the eigenvalues. Using the cyclic permutation of the trace operator, we have
\begin{align*}
\mathrm{tr}(\boldsymbol{AB})=\mathrm{tr}(\boldsymbol{P\Lambda P}^T\boldsymbol{B})=\mathrm{tr}(\boldsymbol{\Lambda P}^T\boldsymbol{BP})=\sum_{i=1}^J\lambda_i(\boldsymbol{A})(\boldsymbol{P}^T\boldsymbol{BP})_{ii}.
\end{align*}
Note that $(\boldsymbol{P}^T\boldsymbol{BP})_{ii}\geq0$ since $\boldsymbol{B}$ is non-negative definite. Therefore, the right hand side above is bounded above by $\lambda_{\mathrm{max}}(\boldsymbol{A})\mathrm{tr}(\boldsymbol{P}^T\boldsymbol{BP})$ and below by $\lambda_{\mathrm{min}}(\boldsymbol{A})\mathrm{tr}(\boldsymbol{P}^T\boldsymbol{BP})$. The result follows by invoking the cyclic permutation of the trace operator again and the fact $\boldsymbol{PP}^T=\boldsymbol{I}$ to conclude that $\mathrm{tr}(\boldsymbol{P}^T\boldsymbol{BP})=\mathrm{tr}(\boldsymbol{BPP}^T)=\mathrm{tr}(\boldsymbol{B})$.
\end{proof}

\begin{lemma}\label{lem:BBinfty}
Let $\boldsymbol{H}$ be a $J\times J$ symmetric, positive definite and $w$-banded such that $h_{ij}=0$ if $|i-j|>w$. Assume that the eigenvalues of $\boldsymbol{H}$ are contained in $[a\tau_m,b\tau_m]$ for some fixed $0<a<b<\infty$ and some sequence $\tau_m$. Furthermore, let the function $g$ be analytic on $[a\tau_m,b\tau_m]$. Then,
\begin{align*}
\|g(\tau_m^{-1}\boldsymbol{H})\|_{(\infty,\infty)}=O(1).
\end{align*}
\end{lemma}

\begin{proof}
First observe that if $\boldsymbol{U}$ is $a$-banded and $\boldsymbol{V}$ is $b$-banded, then $\boldsymbol{UV}$ is $a+b$-banded. Indeed, $(\boldsymbol{UV})_{i,j}=\sum_{k=1}^Ju_{ik}v_{kj}\neq0$ when at least one term in the sum is not zero. Therefore, both $u_{ik}$ and $v_{kj}$ are nonzero for some $k$. By bandedness of the matrices involved, it must be that $|i-k|\leq a$ and $|j-k|\leq b$, and thus by the triangle inequality $|i-j|\leq a+b$, implying that $\boldsymbol{UV}$ is $a+b$-banded. By an induction argument, we deduce that $\boldsymbol{U}^n$ is $na$-banded.

Because we can always scale $\boldsymbol{H}$ by $\tau_m^{-1}$ such that its eigenvalues are in $[a,b]$, we let $\tau_m=1$ without loss of generality. Let $p_n$ be a polynomial function of degree $n$, and by the previous argument, we know that $p_n(\boldsymbol{H})$ is $nw$-banded. We will use a result by Bernstein (see Theorem 73 of \citep{bernstein}) on the best polynomial approximation to analytic functions. For any $f$ analytic on $[a,b]$, there exist constants $C_0>0,\delta<1$ such that
\begin{align}
\inf_{p\in\mathcal{P}}\|f-p\|_\infty\leq C_0\delta^{n+1},
\end{align}
where $\mathcal{P}$ is the space of polynomials of degree $n$ (or order $n+1$). Here $C_0$ and $\delta$ depend on half-axes of certain ellipse containing $[a,b]$ and also the supremum of $f$ on this ellipse. Now by spectral theory and the result above,
\begin{align*}
\|g(\boldsymbol{H})-p_n(\boldsymbol{H})\|_{(2,2)}=\sup_{x\in\Lambda(\boldsymbol{H})}|g(x)-p_n(x)|\leq C_0\delta^{n+1},
\end{align*}
where $\Lambda(\boldsymbol{H})$ is the set of eigenvalues of $\boldsymbol{H}$. As mentioned before, $p_n(\boldsymbol{H})_{i,j}=0$ for $|i-j|>nw$. Now suppose $i\neq j$ and choose $n$ such that $nw<|i-j|\leq(n+1)w$, then
\begin{align*}
|g(\boldsymbol{H})_{i,j}|=|g(\boldsymbol{H})_{i,j}-p_n(\boldsymbol{H})_{i,j}|\leq\|g(\boldsymbol{H})-p_n(\boldsymbol{H})\|_{(2,2)}\leq C_0\delta^{|i-j|/w},
\end{align*}
since $\delta<1$. For the case $i=j$, we have $g(\boldsymbol{H})_{ii}\leq\|g(\boldsymbol{H})\|_{(2,2)}$. Combining both cases, we conclude $|g(\boldsymbol{H})_{i,j}|\leq\max\{C_0,\|g(\boldsymbol{H})\|_{(2,2)}\}\delta^{|i-j|/w}$. Again recalling that $\delta<1$,
\begin{align*}
\|g(\boldsymbol{H})\|_{(\infty,\infty)}\leq\max\{C_0,\|g(\boldsymbol{H})\|_{(2,2)}\}\max_{1\leq i\leq J}\sum_{j=1}^J\delta^{|i-j|/w}\lesssim1+2\sum_{j=1}^\infty\delta^{j/w}<\infty.\qquad\qedhere
\end{align*}
\end{proof}

\bibliographystyle{apa}
\bibliography{bayeslepski}
\end{document}